\documentclass[a4paper,10pt%,draft
]{amsart}
\usepackage[english]{babel}
\usepackage{xr}
\usepackage{stmaryrd}
\usepackage{amsmath,amssymb,amsthm}
\usepackage[alphabetic, abbrev]{amsrefs}
\usepackage{enumerate}
\usepackage{mathtools, comment}
\usepackage[bookmarksopen]{hyperref}
\usepackage[all]{xy}
\newdir{ >}{{}*!/-5pt/@{>}} %% cf xyguide exercise 14
\SelectTips{cm}{} %% cf xyguide page 9
\numberwithin{equation}{section} 
\let\realequation\equation
\def\equation{\setcounter{equation}{\arabic{subsection}}%
   \refstepcounter{subsection}%
   \realequation}
\let\realmultline\multline
\def\multline{\setcounter{equation}{\arabic{subsection}}%
   \refstepcounter{subsection}%
   \realmultline}

\newtheorem{theorem}[subsection]{Theorem}
\newtheorem{corollary}[subsection]{Corollary}
\newtheorem{lemma}[subsection]{Lemma}
\newtheorem{proposition}[subsection]{Proposition}
\theoremstyle{definition}
\newtheorem{definition}[subsection]{Definition}

\newtheorem{remark}[subsection]{Remark}

\newtheorem{example}[subsection]{Example}

\newtheorem{construction}[subsection]{Construction}

\newcommand{\bC}{\mathbb{C}}

\newcommand{\bS}{\mathbb{S}}

\newcommand{\cB}{\mathcal{B}}
\newcommand{\cC}{\mathcal{C}}
\newcommand{\cD}{\mathcal{D}}
\newcommand{\cE}{\mathcal{E}}

\newcommand{\cI}{\mathcal{I}}

\newcommand{\cK}{\mathcal{K}}

\newcommand{\cR}{\mathcal{R}}
\newcommand{\cS}{\mathcal{S}}

\newcommand{\ucI}{*}

\DeclareMathOperator{\hocolim}{hocolim}

\DeclareMathOperator{\colim}{colim}
\DeclareMathOperator*{\colimsubscript}{colim}
\DeclareMathOperator{\const}{const}

\DeclareMathOperator{\Sp}{Sp}
\DeclareMathOperator{\Ev}{Ev}

\DeclareMathOperator{\Hom}{Hom}

\DeclareMathOperator{\Mod}{Mod}

\DeclareMathOperator{\Ar}{Ar}
\DeclareMathOperator{\Fun}{Fun}
\DeclareMathOperator{\Obj}{Ob}

\DeclareMathOperator{\Sing}{Sing}
\newcommand{\ot}{\leftarrow}

\newcommand{\op}{{\mathrm{op}}}
\newcommand{\id}{{\mathrm{id}}}
\newcommand{\tensor}{\otimes}
\newcommand{\sm}{\wedge}
\newcommand{\barsm}{\barwedge}
\newcommand{\iso}{\cong}
\newcommand{\concat}{\sqcup}
\newcommand{\bld}[1]{{\mathbf{#1}}}
\newcommand{\Spsym}[1]{{\mathrm{Sp}^{\Sigma}_{#1}}}
\newcommand{\sset}{\mathrm{sSet}}
\newcommand{\tp}{\mathrm{Top}}
\newcommand{\SpsymR}{{\mathrm{Sp}^{\Sigma}_{\cR}}}

\DeclareMathOperator{\capitalGL}{GL}

\newcommand{\GLoneIof}[1]{\capitalGL^{\cI}_1\!{#1}}
\newcommand{\GLoneof}[1]{\capitalGL_1\!{#1}}
\newcommand{\NFin}{{\mathrm{NFin}_*}}

\newcommand{\arxivlink}[1]{\href{http://arxiv.org/abs/#1}{\texttt{arXiv:#1}}}
\title[Symmetric spectra in retractive spaces]{Multiplicative parametrized homotopy theory via symmetric spectra in retractive spaces}

\author{Fabian Hebestreit} \address{Mathematical Institute, University
  of Bonn, Endenicher Allee 60, 53115 Bonn, \newline{}Germany}
\email{f.hebestreit@math.uni-bonn.de}

\author{Steffen Sagave} \address{IMAPP, Radboud University Nijmegen,
  PO Box 9010, 6500 GL Nijmegen, \newline{}The Netherlands}
\email{s.sagave@math.ru.nl}

\author{Christian Schlichtkrull} \address{Department of Mathematics,
  University of Bergen, P.O. Box 7803, 5020 Bergen, \newline{}Norway}
\email{christian.schlichtkrull@math.uib.no} \date{\today}

\begin{document}

\begin{abstract}
In order to treat multiplicative phenomena in twisted (co)ho\-mology, we introduce a new point-set level framework for parametrized homotopy theory. We provide a convolution smash product that descends to the corresponding $\infty$-categorical product and allows for convenient constructions of commutative parametrized ring spectra. As an immediate application, we compare various models for generalized Thom spectra. In a companion paper, this approach is used to compare homotopical and operator algebraic models for twisted $K$-theory.
\end{abstract}


\maketitle
\setcounter{tocdepth}{1}  
\tableofcontents

\section{Introduction}
Stable parametrized homotopy theory originally arose from the study of transfer maps and fiberwise duality for generalized (co)homology theories. To analyze these phenomena, Clapp and Puppe~\cite{Clapp-P_parametrized} introduced a first homotopy category of parametrized spectra. Later, May and Sigurdsson~\cite{May-S_parametrized} studied a more refined model category of orthogonal parametrized spectra over a base space $B$ enjoying favorable point-set topological properties. Both these approaches relate duality to a smash product obtained by first forming an external fiberwise smash product lying over $B \times B$ and then internalizing it by pullback along the diagonal $B \to B \times B$.

When studying cross and cup products in twisted (co)homology through the representing para\-metrized spectra, one needs a different symmetric monoidal structure. Suppose that the base space $B$ has a homotopy coherent commutative multiplication, that is, an $E_\infty$ structure. Then one can also attempt to internalize the external fiberwise smash product over $B \times B$ by pushout along the multiplication $\mu \colon B \times B \to B$. We will refer to this type of product as a \emph{convolution smash product} to distinguish it from the \emph{fiberwise smash product} considered above. However, the setup of May--Sigurdsson does not provide such a symmetric monoidal convolution smash product unless the multiplication of $B$ is strictly associative and commutative, ruling out the parameter spaces of many interesting parametrized spectra, such as those representing twisted $K$-theory or bordism theories. Ando, Blumberg, and Gepner~\cite{Ando-B-G_parametrized} implemented the convolution smash product in the $\infty$-categorical setup from \cite{Ando-B-G-H-R_infinity-Thom} and showed how it for example gives rise to twisted Umkehr maps. 

The primary aim of the present paper is to introduce a convenient point set level category of parametrized spectra that admits a symmetric monoidal convolution smash product descending to the $\infty$-categorical product of~\cite{Ando-B-G_parametrized}. Our main new idea is to also allow the base space to vary for the different levels of a parametrized spectrum. More precisely, the base spaces will assemble to an $\cI$-space, i.e., a functor from the category $\cI$  of finite sets $\bld{n} = \{1,\dots, n\}$, $n \geq 0$, and injective maps to the category of spaces $\cS$. Replacing the cartesian product of base spaces, the category $\cS^{\cI}$ of $\cI$-spaces is equipped with a symmetric monoidal Day convolution product $\boxtimes$ induced by the concatenation in $\cI$ and the cartesian product of spaces. We call commutative monoids with respect to $\boxtimes$ \emph{commutative $\cI$-space monoids}.

It is proved in \cite{Sagave-S_diagram} that every $E_{\infty}$ homotopy type arises as the homotopy colimit $M_{h\cI} = \hocolim_{\cI}M$ for a commutative $\cI$-space monoid $M$. We think of $M_{h\cI}$ as the underlying $E_{\infty}$ space of $M$. This point of view often leads to simple and explicit models of $E_{\infty}$ spaces. Working with symmetric spectra parametrized over commutative $\cI$-space monoids allows us to implement the notion of a convolution smash product in a convenient fashion. Different point-set frameworks for parametrized homotopy theory were recently developed for example in~\cite{HNP} and \cite{VBM-Thesis}, but to our knowledge these approaches again do not allow for a convolution smash product in sufficient generality for the applications we have in mind.

In a companion paper~\cite{HS-twisted}, the first named authors use the setup developed here to prove that twisted K-theory as defined via operator algebraic methods coincides with the version defined via homotopy theoretic methods. More specifically, it is shown that these theories agree as commutative parametrized ring spectra with respect to the convolution smash product. In this case, the commutative $\cI$-space monoids serving as base spaces model the classifying space of the projective orthogonal group of a Hilbert space. 

\subsection{Symmetric spectra in retractive spaces} To implement our approach, we first consider the category $\cS_{\cR}$ of retractive spaces. Objects of $\cS_{\cR}$ are pairs of spaces $(U,K)$ with structure maps $K \to U \to K$ that compose to the identity. Morphisms in $\cS_{\cR}$ are pairs of maps making the two obvious squares commutative. The external fiberwise smash product $\barsm$ provides a symmetric monoidal structure on $\cS_{\cR}$ with unit $(S^0,*)$. On base spaces, this $\barsm$-product is just the cartesian product. 

Following work of Hovey~\cite{Hovey_symmetric-general}, we form the category $\SpsymR$ of symmetric spectrum objects in $\cS_{\cR}$ with $-\barsm (S^1,*)$ as the suspension functor. It comes with a \emph{local} model category structure whose fibrant objects are $\Omega$-spectra. (We avoid the term \emph{stable} since in lack of a zero object, $\mathrm{Ho}(\SpsymR)$ is not stable in the technical sense.) The category  $\SpsymR$ inherits a symmetric monoidal product $\barsm$ from $\cS_{\cR}$ that will play the role of the external smash product. An object $(E,X)$ of $\SpsymR$ is a sequence of retractive spaces $(E_n,X_n)$ with an action of the symmetric group $\Sigma_n$ and structure maps $(E_n,X_n) \barsm (S^1,*) \to (E_{n+1},X_{n+1})$ compatible with the $\Sigma_n$-actions. Inspecting definitions, the projection $\pi_b \colon \cS_{\cR} \to \cS$ to the base space induces a projection $\pi_b \colon \SpsymR \to \cS^{\cI}$ to base $\cI$-spaces. If $X$ is an $\cI$-space, we define the category of \emph{$X$-relative symmetric spectra} $\Spsym{X}$ to be the fiber of $\pi_b$ over $X$. We stress that unless $X$ is constant, $\Spsym{X}$ is not the category of symmetric spectrum objects in some base category since the levels of $(E,X)$ take values in different categories.

A map of $\cI$-spaces  $f\colon X \to Y$ induces an adjunction $f_! \colon \Spsym{X} \rightleftarrows \Spsym{Y}\colon f^*$  where $f^*$ denotes degreewise pullback. We say that $f$ is an \emph{$\cI$-equivalence} if the map of homotopy colimits $f_{h\cI}\colon X_{h\cI} \to Y_{h\cI}$ is a weak homotopy equivalence and note that the $\cI$-equivalences are the weak equivalences in an \emph{$\cI$-model structure} on $\cS^{\cI}$ (see~\cite{Sagave-S_diagram}).

\begin{theorem}\label{thm:model-str-introduction}
Let $X$ be an $\cI$-space. The category $\Spsym{X}$ of $X$-relative symmetric spectra admits a local model structure where a map is a cofibration, fibration, or weak equivalence if and only if it is so as a map in $\SpsymR$. 

With respect to the local model structure, $(f_!,f^*)$ is a Quillen adjunction that is a Quillen equivalence if $f$ is an $\cI$-equivalence. In particular, $\Spsym{X}$ is Quillen equivalent to the stabilization of the category of spaces over and under $X_{h\cI}$.
\end{theorem}

Consequently, $\Spsym{X}$ models the same homotopy theory as the category of $X_{h\cI}$-parametrized spectra in the sense of May--Sigurdsson~\cite{May-S_parametrized} or \cite{Ando-B-G-H-R_infinity-Thom}. We also show that $\mathbb R f^* \colon \mathrm{Ho}(\Spsym{Y}) \to \mathrm{Ho}(\Spsym{X})$ admits a right adjoint $\mathbb R f_*$ that does, however, not arise from a right Quillen functor.

We do in fact provide two versions of the local model structure in the theorem, an absolute and a positive one, where as usual the positive version is necessary to lift the model structures to categories of commutative monoids (see~\eqref{eq:mon-str-intro} below). The theorem also has a much easier unstable analogue: the category of retractive spaces $\cS_{\cR}$ inherits a model structure from the category of spaces where a map is a weak equivalence if both of its components are, and the standard model structure on the category of spaces $\cS_K$ over and under $K$ can be viewed as a ``restriction'' of this model structure to the subcategory $\cS_K$ of $\cS_{\cR}$. However, the proof of Theorem~\ref{thm:model-str-introduction} turns out to be not as easy as it may look. The problem is that the factorizations needed for the model category structure on $\Spsym{X}$ are not inherited from $\SpsymR$ since the factorizations in the latter category may change the base object. To circumvent this problem, we give an intrinsic description of the category $\Spsym{X}$ and its local model structure in terms of section categories and then show that its cofibrations, fibrations, and weak equivalences are detected in $\SpsymR$. 
 
It is also useful to notice that the model category $\SpsymR$ can be recovered from the $\Spsym{X}$ for varying $X$. As a category, $\SpsymR$ is equivalent to the Grothendieck construction of the pseudofunctor $X \mapsto \Spsym{X}$ sending an $\cI$-space $X$ to $\Spsym{X}$ and a map $f\colon X \to Y$ to $f_!$. Harpaz and Prasma~\cite{Harpaz-P_Grothendieck-construction} have identified conditions under which a model structure on the base category and model structures on the values of a pseudofunctor assemble to a so-called \emph{integral} model structure on the Grothendieck construction, and we verify these conditions in the case at hand. 

\begin{theorem}\label{thm:integral-introduction}
The $\cI$-model structure on $\cS^{\cI}$ and the local model structures on the $\Spsym{X}$ induce an integral model structure on the Grothendieck construction of $X \mapsto \Spsym{X}$. Under the equivalence of the Grothendieck construction with $\SpsymR$, the integral model structure corresponds to the local model structure on~$\SpsymR$. 
\end{theorem}
This theorem is again analogous to the unstable situation where it is easy to check that $\cS_{\cR}$ is equivalent to the Grothendieck construction of the pseudofunctor $K \mapsto \cS_K$ on $\cS$, and that the integral model structure exists and is equivalent to the model structure on $\cS_{\cR}$ considered earlier. 

The fiberwise smash product on $\cS_{\cR}$ induces a fiberwise smash
product $\barsm$ on $\SpsymR$ that is the $\boxtimes$-product on the
base $\cI$-spaces. Therefore, it restricts to an external fiberwise
smash product $\Spsym{X} \times \Spsym{X} \to \Spsym{X \boxtimes X}$.
For a commutative $\cI$-space monoid $M$, we can now use the
pushforward along the multiplication $\mu \colon M \boxtimes M \to M$
to define a symmetric monoidal convolution product
\begin{equation}\label{eq:mon-str-intro} \Spsym{M} \times \Spsym{M} \xrightarrow{\barsm} \Spsym{M \boxtimes M} \xrightarrow{\mu_!} \Spsym{M} 
\end{equation}
on the category of $M$-relative symmetric spectra. 

Now any parametrized spectrum $(E,X) \in \Spsym{\cR}$ gives rise to a parametrized (co)homology theory as expected. To define it, recall that there is an $\cI$-spacification functor $\cS/X_{h\cI} \to \cS^{\cI}/X, \tau \mapsto \tau_\cI$ that is homotopy inverse to $\hocolim_{\cI}$~\cite[\S 4]{Schlichtkrull_Thom-symmetric}. We then set
\begin{align*}(E,X)_n \colon \cS/X_{h\cI} \longrightarrow \mathrm{Ab}, \quad & (\tau \colon A \rightarrow X_{h\cI}) \longmapsto \pi_{n}(\mathbb L\Theta)(\mathbb R \tau_\cI^*) (E,X) \quad\text{ and}\\
(E,X)^n \colon \cS/X_{h\cI} \longrightarrow \mathrm{Ab}, \quad & (\tau \colon A \rightarrow X_{h\cI}) \longmapsto \pi_{-n}(\mathbb R\Gamma)(\mathbb R \tau_\cI^*) (E,X),
\end{align*}
where for any $\cI$-space $Y$ the functors $\mathbb L\Theta, \mathbb R\Gamma \colon \mathrm{Ho}(\Spsym{Y}) \rightarrow \mathrm{Ho}(\Spsym{})$ denote the left and right adjoint, respectively, of the derived pullback functor along the unique map $Y \rightarrow *$.  

Let $(R,M)$ be a parametrized ring spectrum, that is, a monoid object in $\Spsym{\cR}$. Then we obtain the cross product displayed as the upper horizontal arrow in the following square: 
\[\xymatrix@-1pc{(R,M)_*(K,\tau) \otimes (R,M)_*(L,\sigma) \ar[rr]^-\times \ar[d] && (R,M)_*(K \times L, \tau \times \sigma) \ar[d] \\
            (R,M)_*(L,\sigma) \otimes (R,M)_*(K,\tau) \ar[rr]^-\times &&        (R,M)_*(L \times K, \sigma \times \tau)}\]
When $(R,M)$ is commutative, the square commutes up to the usual sign where the right hand vertical map is induced by the essentially unique homotopy between the following two maps arising from the $E_\infty$-structure on $M_{h\cI}$: 
\begin{gather*} L \times K \xrightarrow{\mathrm{tw}} K \times L \xrightarrow{\tau \times \sigma} M_{h\cI} \times M_{h\cI} \xrightarrow{\mu} M_{h\cI}\\  L \times K \xrightarrow{\sigma \times \tau} M_{h\cI} \times M_{h\cI} \xrightarrow{\mu} M_{h\cI}
\end{gather*}
An analogous statement holds for the cup and cross products in cohomology. 

\subsection{Comparison to the \texorpdfstring{$\infty$}{infinity}-categorical setup} 
We can also use Theorem~\ref{thm:integral-introduction} to compare the
categories $\Spsym{X}$ to the $\infty$-categorical set-up of parametrized
homotopy theory. There the category of parametrized
spectra over a space $K$ is given by $\Fun(K,\Sp_\infty)$, and these categories
also assemble into a category of parametrized spectra with varying base space
by Lurie's higher categorical version of the Grothen\-dieck construction. The
resulting category is also known as the tangent category $T\cS_\infty$ of the $
\infty$-category of spaces $\cS_\infty$, and we shall adopt this name to ease
notation in the comparison results. 

\begin{theorem}\label{thm:infty-categorical-products-intro}
For an $\cI$-space $X$, the underlying $\infty$-category of the local model
structure on $\Spsym{X}$ is canonically equivalent to $\Fun(X_{h\cI}, \Sp_\infty)$, which translates the Quillen adjunction $(f_!,f^*)$ for any map $f\colon X \rightarrow X'$ into its $\infty$-categorical counterpart. Therefore, as $X$
varies, these equivalences assemble into an equivalence between $(\SpsymR)_\infty$ and $T\cS_\infty$.

Furthermore, this equivalence is symmetric monoidal with respect to the
exterior smash product on both sides, and for $M$ a commutative $\cI$-space
monoid, the equivalence $\Spsym{M} \simeq \Fun(M_{h\cI}, \Sp_\infty)$ is
symmetric monoidal for the convolution smash product on the left and Day
convolution on the right.
\end{theorem}

In fact, as far as we know, the symmetric monoidal structure on $T\cS_\infty$
has not appeared in the literature before. Therefore, extending recent work of
Nikolaus \cite{nik-stable}, we provide the necessary material on the stabilization of fibrations of $\infty$-operads needed to construct it, which may be of independent interest.

\subsection{Universal bundles and Thom spectra}
When $R$ is a (sufficiently fibrant) commutative symmetric ring spectrum, then its underlying multiplicative infinite loop space and its units arise as commutative $\cI$-space monoids $\GLoneIof{R} \subset \Omega^{\cI}(R)$ (see~\cite{Schlichtkrull_units}).  By definition, we have $\Omega^{\cI}(R)(\bld{n}) = \Omega^n R_n$ with structure maps and multiplication maps induced by those of $R$, and  $\GLoneIof{R}$ is the subobject whose path components represent units in $\pi_0(R)$. The underlying $E_{\infty}$ space of $\GLoneIof{R}$ is a model of what is usually denoted $\GLoneof{R}$.  Writing $G$ for (a suitable cofibrant replacement of) $\GLoneIof{R}$, the inclusion $G \to  \Omega^{\cI}(R)$ is adjoint to a map  of commutative symmetric ring spectra $\bS^{\cI}[G] \to R$ from the spherical monoid ring of $G$. 

There also is a parametrized suspension spectrum $\bS^{\cI}_t[G]$ of $G$ with $G$ as base commutative $\cI$-space monoid. The above map $\bS^{\cI}[G] \to R$ and the projection to the terminal $\cI$-space $G \to *$ induce  commutative $\bS^{\cI}_t[G]$-algebra structures on $\bS$ and $R$ that allow us to form the two-sided bar construction $B^{\barsm}(\bS, \bS^{\cI}_t[G], R)$ in commutative parametrized ring spectra. Its base commutative $\cI$-space monoid is the bar construction $BG$ of $G$ with respect to $\boxtimes$ which models the infinite loop space $B\GLoneof{R}$.  We write $\gamma_R$ for (a suitable fibrant replacement of) $B^{\barsm}(\bS, \bS^{\cI}_t[G], R)$ and view this $BG$-relative commutative symmetric ring spectrum as the \emph{universal $R$-line bundle} over $BG$. It only depends on the stable equivalence type of $R$ and is mapped to the $\infty$-categorical version of the universal $R$-line bundle from \cite{Ando-B-G_parametrized} under the equivalence from Theorem~\ref{thm:infty-categorical-products-intro} (including multiplicative structures). 
Moreover, $\gamma_R$ gives rise to an $R$-module Thom spectrum functor 
\[ T_{\cR}^{\cI} \colon \cS^{\cI}/BG \to \Mod_R,\quad (f\colon X \to BG)
\mapsto (X \to *)_! (f^*\gamma_R)\]
on $\cI$-spaces over $BG$ that is homotopy invariant and lax symmetric monoidal. Precomposing with the $\cI$-spacification $\cS/(BG)_{h\cI} \to
  \cS^{\cI}/BG$ mentioned above, we also get a Thom spectrum functor 
$T_{\cR}\colon \cS/(BG)_{h\cI}\to \Mod_R$ defined on maps of spaces to $(BG)_{h\cI} \simeq B\GLoneof{R}$. 

The functor $T_{\cR}^{\cI}$ provides a different construction of the $R$-module Thom spectra studied by Ando et al.~\cites{Ando-B-G-H-R_infinity-Thom, Ando-B-G-H-R_units-Thom} making the underlying parametrized spectra explicit on the point-set level. We give the following multiplicative comparison of these two and the multiplicative Thom spectrum functors studied by Basu--Sagave--Schlichtkrull~\cite{Basu_SS_Thom}. 
\begin{theorem}\label{thm:Thom-spectrum-comparison}
The Thom spectrum functor $T_{\cR}$ defined in terms of parametrized spectra, the $R$-module Thom spectrum functor $T_{EG}$ from~\cite{Basu_SS_Thom}, and the Thom spectrum functor from~\cite{Ando-B-G-H-R_infinity-Thom} are all equivalent, and the equivalences respect the monoidal structures and the operad actions preserved by these functors. \end{theorem} 
Over the sphere spectrum $\bS$, we provide a new interpretation of the ``classical'' Thom spectrum functor considered in \cite{LMS} and \cite{Schlichtkrull_Thom-symmetric}: Let $F(\bld n)$ denote  the topological monoid of base point preserving homotopy equivalences of $S^n$. Letting $\bld n$ vary, we show that the usual one-sided bar construction on $F(\bld n)$ gives rise to a commutative parametrized symmetric ring spectrum $B^{\times}(\ucI,F,\bS)$ that is locally equivalent to the universal line bundle for $\bS$. Using this, we get an explicit multiplicative equivalence relating the classical description of the Thom spectrum functor to the parametrized approach in the present paper. 
\subsection{Homotopical and operator algebraic models for twisted \texorpdfstring{$K$}{K}-theory}
As mentioned, in a companion paper~\cite{HS-twisted}, the first two authors use the framework developed here to relate operator algebraic models for various twisted $K$-theory spectra to their homotopical counterparts defined in terms of pullbacks of the universal bundle just discussed. To obtain the comparison we there generalize the construction of $B^{\times}(\ucI,F,\bS)$ by considering actions on symmetric ring spectra of what we term \emph{cartesian $\cI$-monoids}, that is, $\cI$-diagrams in topological monoids. For such an action of $H$ on $R$ one can produce a homotopy quotient spectrum $R \sslash H \in \Spsym{B^\times(H)}$, where $B^\times(H)$ denotes the bar construction of $H$ with respect to the cartesian product on $\cS^\cI$. In particular, $B^{\times}(\ucI,F,\bS) = \bS\sslash F$. When $H$ is grouplike and $R$ and $H$ are fibrant, we establish a comparison between $R \sslash H$ and the pullback of $\gamma_R$ along a certain map $BH \rightarrow B\GLoneIof(R)$ induced by evaluating the action of $H$ on the unit of~$R$~\cite[Proposition~4.2]{HS-twisted}. When $R, H$, and the action are suitably commutative, then this comparison is one of commutative parametrized ring spectra.

The whole construction can then be applied to actions of the cartesian $\cI$-monoid $\mathrm P \mathcal O$ formed by the projective orthogonal groups of the Hilbert spaces $L^2(\mathbb R^n)$ on the K-theory spectra introduced by Joachim in \cite{Jo-coherence}, and this action satisfies the commutativity assumptions mentioned above. The spectrum $\mathrm{KO} \sslash \mathrm P \mathcal O$ is easily related to operator algebraic definitions of twisted K-theory, whereas by the comparison results we produce here, the pullback of $\gamma_{\mathrm{KO}}$ is an incarnation of the homotopy theoretic definition. This allows us to deduce the equivalence of the usual definitions of twisted $K$-theory as considered by operator algebraists and homotopy theorists. Furthermore, we describe the resulting map $B\mathrm P \mathcal O \rightarrow B\GLoneIof{\mathrm{KO}}$ in purely homotopical terms, completing partial results by Antieau, Gepner and Gomez \cite{AGG-Uniqueness}.

\subsection{Organization}
In Section \ref{retrsp} we recall the category of retractive spaces and describe the relevant features of its model structure. In Section \ref{twisepc} we then introduce our categories of parametrized spectra and their level model structures both in the absolute setting and relative to an $\cI$-space. After discussing the external fiberwise and convolution smash products in the short Section \ref{sec:smash}, we establish the local model structures in Section \ref{sec:local-model}. In Section \ref{compmodel} we compare the different local model structures, prove Theorems~\ref{thm:model-str-introduction} and \ref{thm:integral-introduction}, and show how the positive local model structures lift to commutative parametrized ring spectra. Section \ref{sec:tw-coho} is about the (co)homology theories associated with parametrized spectra. In Section~\ref{sec:univ-bundle} we introduce the universal line bundle. Section~\ref{sec:thom-comparison} is about Thom spectra and provides the proof of Theorem~\ref{thm:Thom-spectrum-comparison}. The final Section~\ref{sec:infty-categorical} compares our constructions to the $\infty$-categorical approach and provides the proof of Theorem~\ref{thm:infty-categorical-products-intro}. 

\subsection{Acknowledgments}
The authors would like to thank Samik Basu, Emanuele Dotto, Gijs Heuts, Michael Joachim, Cary Mal\-kie\-wich, Irakli Patchkoria, and Tomer Schlank for useful conversations and particularly Thomas Nikolaus for help with the final chapter. The first author held a scholarship of the German Academic Exchange Service during a year at the University of Notre Dame, when initial work on this project was undertaken. The third was supported by Meltzers H\o{}yskolefond.  Furthermore, the authors would like to thank the Isaac Newton Institute for Mathematical Sciences for support and hospitality during the program ``Homotopy harnessing higher structures".

This work was supported by EPSRC grants EP/K032208/1 and EP/R014604/1
and the Hausdorff Center for Mathematics, DFG GZ 2047/1, project ID
390685813.

\subsection{Notations and conventions}
We write $\cS$ for the category of spaces, which can either be the
category of compactly generated weak Hausdorff spaces $\tp$ or the
category of simplicial sets $\sset$. We will only distinguish between
the simplicial and the topological case when arguments
differ. Moreover, we freely use the language of model categories and
refer to~\cite{Hirschhorn_model} as our primary source.

\section{Retractive spaces}\label{retrsp}
In this section we collect basic results about retractive spaces. This material is mostly easy and for example treated in~\cite[\S 1.1]{VBM-Thesis}. We carry out some details for later reference and to fix notations. 

\begin{definition}
  A \emph{retractive space} is a pair of spaces $(U,K)$ with structure
  maps $K \to U$ and $U \to K$ that compose to the identity of $K$. A
  map of retractive spaces $(U,K) \to (V,L)$ is a pair of maps $U\to
  V$ and $K \to L$ such that the two squares in
\[\xymatrix@-1pc{
  K \ar[d] \ar[r] & U \ar[d] \ar[r] & K \ar[d] \\
  L \ar[r] & V \ar[r] & L }\] commute. We refer to $U$ as the
\emph{total space} of the retractive space $(U,K)$, call $K$ its
\emph{base space}, and write $\cS_{\cR}$ for the category of retractive
spaces.
\end{definition}

\begin{construction}\label{constr:spaces-retractive-spaces}
Spaces and retractive spaces are related in various ways. The following diagram summarizes the constructions relevant for us:
\[\xymatrix@-.8pc{
\cS \ar@<.4ex>[rrr]^-{\iota_d}\ar@<-.3ex>[drrr]^(.6){\iota_{t}}  &&& \mathrm{Ar}(\cS) \ar@<-.4ex>[d]_-{\iota_{\mathrm{ar}}}  \ar@<.4ex>[lll]^-{\pi_d}  \ar@<-.4ex>[rrr]_-{\pi_c}&&& \cS \ar@<-.4ex>[lll]_-{\iota_c} \ar@<.3ex>[dlll]_(.6){\iota_{b}} &&& \cS_{\cR}  \ar@<-.4ex>[d]_-{\Theta}\\
&&& \cS_{\cR} \ar@<-.4ex>[u]_-{\pi_{\mathrm{ar}}}  \ar@<-1.1ex>[urrr]_(.4){\pi_{b}} \ar@<1.1ex>[ulll]^(.4){\pi_{t}}  &&&  &&& \cS_* \ar@<-.4ex>[u]_-{\pi_{\mathrm{pt}}} 
}\]
Here $\mathrm{Ar}(\cS)$ is the arrow category of $\cS$, $\pi_d(V\to L) = V$ and $\pi_c(V \to L) = L$ are the forgetful functors projecting to the domain and codomain, and $\iota_c(K) = (\emptyset \to K)$ and $\iota_d(K) = \id_K$ are their left adjoints. The forgetful functor $\pi_{\mathrm{ar}}(V,L) = {(V \to L)}$ remembers only the projection to the base and has a left adjoint $\iota_{\mathrm{ar}}{(f \colon A \to B)} = (B\amalg A, B)$ with structure maps $\mathrm{incl}_B \colon B \to B \amalg A$ and $(\id_B,f)\colon B \amalg A \to B$. Consequently, we obtain composite functors $\pi_b(V,L) = L$ and $\pi_t(V,L) = V$ projecting to base and total space. Their left adjoints are given by $\iota_b(K) = (K,K)$ and $\iota_t(K) = (K\amalg K,K)$. A based space $T$ can be viewed as a retractive space $(T,*) = (* \to T \to *)$ that we often denote by $T$. The functor $\pi_{\mathrm{pt}}\colon \cS_* \to \cS_{\cR}, T\mapsto T$ is right adjoint to $\Theta(U,K) = U/K$. 
\end{construction}

\begin{lemma}\label{lem:iota-b-left-right-adjoint}
The functor $\iota_b \colon \cS \to \cS_{\cR}$ is both left and right adjoint 
to $\pi_b$. \qed
\end{lemma}

The category of retractive spaces $\cS_{\cR}$ is complete and
cocomplete since it can be viewed as a category of functors with
values in $\cS$. The last lemma implies that both
$\iota_b$ and $\pi_b$ preserve limits and colimits.

\subsection{Retractive spaces as a model category}
In the following we use the standard model structures on $\cS = \sset$ and $\cS = \tp$ with weak equivalences the weak homotopy equivalences. We say that a map of retractive spaces $(U,K) \to (V,L)$ is
\begin{itemize}
\item a weak equivalence if both $K \to L$ and $U \to V$ are weak
  equivalences in $\cS$,
\item a cofibration if both $K \to L$ and $U\cup_K L \to V$ are cofibrations in $\cS$, and 
\item a fibration if both $K \to L$ and $U \to V\times_{L} K$ are fibrations in $\cS$.  
\end{itemize}
\begin{proposition}\label{prop:model-str-on-R}
These classes of maps provide a model structure on $\cS_{\cR}$. 
\end{proposition}
\begin{proof}
  This follows from standard model category arguments. Alternatively,
  one can identify $\cS_{\cR}$ with the diagrams in $\cS$ indexed by a
  Reedy category with two objects $B$ and $T$, one non-identity degree
  raising morphism $s \colon B \to T$, one non-identity degree
  lowering morphism $r \colon T \to B$, and one non-identity
  endomorphism $sr$ of~$T$. With this, the proposition follows from
  the general theory of Reedy model structures (see e.g.\
  \cite[Theorem 15.3.4]{Hirschhorn_model}).
\end{proof}

To establish more properties of the model category $\cS_{\cR}$, we note that the following lifting properties
hold. 
\begin{lemma}\label{lem:lifting-properties-in-R}
Let $i\colon A\to B$ and  $(U,K) \to (V,L)$ be maps in $\cS$ and $\cS_{\cR}$, respectively. 
\begin{enumerate}[(i)]
\item The maps $i\colon A\to B$ and $K\to L$ have the lifting property in $\cS$
if and only if the maps $\iota_b(i\colon A\to B) = \left((A,A) \to (B,B)\right) $ and $(U,K)\to (V,L)$ have the lifting property in~$\cS_{\cR}$.  
\item The maps $i\colon A\to B$ and $U \to V\times_LK$ have the lifting
  property in $\cS$ if and only if $\iota_{\mathrm{ar}}(i \to \id_B) =
  \left((B\amalg A,B) \to (B\amalg B,B)\right)$ and $(U,K)\to (V,L)$ have the
  lifting property in~$\cS_{\cR}$.\qed
\end{enumerate}
\end{lemma}

If $S$ is a set of maps in $\cS$, we write $S_{\cR}$ for the set
\[ \{\iota_b(i) \, | \, i \in S\} \cup \{\iota_{\mathrm{ar}}(i \to \id_B) \, | \, i\colon A \to B \in S\} \] of maps in $\cS_{\cR}$. Moreover, we
let $I$ and $J$ be the standard sets of generating cofibrations
and generating acyclic cofibrations for $\cS$.
\begin{proposition}\label{prop:model-on-R}
  The model category $\cS_{\cR}$ is cofibrantly generated with
  generating cofibrations $I_{\cR}$ and generating acyclic
  cofibrations $J_{\cR}$. \qed
\end{proposition}

We refer to \cite[Definition 12.1.1]{Hirschhorn_model}
or~\cite[Appendix A]{Hovey_symmetric-general} for the notion of a
\emph{cellular} model category. This property is useful because left proper cellular model categories admit left Bousfield localizations \cite[Theorem 4.1.1]{Hirschhorn_model}.

\begin{proposition}
The model category $\cS_{\cR}$ is proper and cellular. 
\end{proposition}
\begin{proof}
It is easy to see that the properness of $\cS_{\cR}$ is inherited from $\cS$. Cellularity is inherited since $\cS$ is cellular and the projection $\cS_{\cR} \to \cS \times \cS$ which forgets the structure maps preserves and detects colimits and limits and sends cofibrations in $\cS_{\cR}$ to objectwise cofibrations in $\cS \times \cS$. 
\end{proof}

\subsection{Retractive spaces as a Grothendieck construction}\label{subsec:retractive-parametrized-spaces} Recall that
\begin{equation}\label{eq:projection-R-S}
\pi_b \colon \cS_{\cR} \to \cS, (U,K) \mapsto K
\end{equation}
denotes the projection to the base space. Let $\cS_K$ be the fiber of
$\pi_b$ over a space $K$, i.e., the subcategory of $\cS_{\cR}$ whose objects
have $K$ as the base space and whose morphisms are the identity on the
base. The category $\cS_K$ is equivalent to the category of spaces
over and under $K$, i.e., to the category of pointed objects in the
over-category $\cS/K$. Other common notations for $\cS_K$ are $(\cS/K)_*$
or $\cS^K_K$. 

Every map of spaces $f\colon K \to L$ induces an adjoint pair of functors 
\begin{equation}\label{eq:res-ext-adj-retractive-spaces}
f_{!}\colon \cS_K \rightleftarrows \cS_L\colon f^*. 
\end{equation}
The left adjoint sends $(U,K)$ to $(U\cup_{K}L,L)$ with its canonical
structure maps, and the right adjoint sends $(V,L)$ to $(f^*(V),K)$ with
its canonical structure maps. 

We write $\mathrm{Cat}$ for the ``category'' of (not necessarily small) categories.  Recall that a pseudofunctor $\cC\colon \cK \to \mathrm{Cat}$ on a category $\cK$ consists of categories $\cC_{K}$ for every object $K$
of $\cK$ and functors $\alpha_{!} \colon \cC_{K} \to \cC_{L}$ for
every morphism $\alpha\colon K \to L$ of $\cK$. The condition that
$\cC_{-}$ is a pseudofunctor (rather than a functor) amounts to saying
that there are coherent isomorphisms (rather than identities) $
\id_{\cC_{K}}\cong (\id_{K})_!$ and $g_! f_!  \cong (gf)_!$ for
composable morphisms $f$ and $g$ in $\cK$. We refer
to~\cite[Definition~7.5.1]{Borceux-1} for a complete definition.

The universal property of the pushout implies:
\begin{lemma}\label{lem:cS_K-pseudofunctor}
The categories $\cS_K$ assemble to a pseudofunctor $\cS \to \mathrm{Cat}$ given
by 
\begin{equation}\label{eq:retractive-spaces-as-Grothendieck-construction}
K \mapsto \cS_K,\quad (f\colon K \to L)\mapsto (f_!\colon\cS_K \to  \cS_L).
\end{equation}\qed
\end{lemma}
The \emph{Grothendieck construction} of a pseudofunctor
$F \colon \cK \to \mathrm{Cat}$ is the category $\cK\int F$ with objects 
the pairs $(K,U)$ with $K \in \Obj \cK$ and $U \in \Obj F(K)$. Morphisms
$(K,U) \to (L,V)$ are pairs of morphisms $f\colon K\to L$ in $\cK$ and $F(f)(U) \to V$ in $F(L)$, and the composite is defined in the obvious way \cite[Definition 3.1.2]{Thomason-homotopy-colimt}.

\begin{lemma}\label{lem:retractive-spaces-as-Grothendieck-construction}
The Grothendieck construction of~\eqref{eq:retractive-spaces-as-Grothendieck-construction} is equivalent to  $\cS_{\cR}$.
\end{lemma}
\begin{proof}
  If $(U,K) \to (V,L)$ is a morphism in $\cS_{\cR}$, then its map of total
  spaces factors as $U \to (K\to L)_!(U) \to V$ where the second map
  is over and under $L$.
\end{proof}

\begin{remark}
  Equivalently, one checks that the
  projection~\eqref{eq:projection-R-S} is a (cartesian)
  \emph{fibration}, see e.g. \cite[Definition 8.1.2 and Theorem
  8.3.1]{Borceux-2}.
\end{remark}

We equip the categories $\cS_K$ with the standard model structures
where a map is a cofibration, fibration, or weak equivalence if the
map of total spaces has this property in $\cS$. With these
model structures, the categories $\cS_K$ are cofibrantly generated, cellular, and proper~\cite{Hirschhorn_over_under}. The following
homotopical properties of the
adjunction~\eqref{eq:res-ext-adj-retractive-spaces} easily
follow from the properness of $\cS$:
\begin{lemma}\label{lem:extension-restriction-R-homotopical} Let $f\colon K \to L$ be a map of spaces. 
\begin{enumerate}[(i)]
\item The adjunction $(f_!,f^*)$ is a Quillen adjunction.  
\item If $f$ is a weak equivalence, then $(f_!,f^*)$ is a Quillen equivalence.
\item If $f$ is an acyclic cofibration, then $f_!$ preserves weak equivalences.
\item If $f$ is an acyclic fibration, then $f^*$ preserves weak equivalences. \qed
\end{enumerate}
\end{lemma} 
Next we recall the terminology of~\cite[Definition 3.0.4]{Harpaz-P_Grothendieck-construction}. 
\begin{definition}\label{def:integral-model-str}
Let $\cK$ be a model category and let $F\colon \cK \to \mathrm{Cat}$
be a pseudo\-functor such that each $F(K)$ is equipped with a model structure
and such that $F$ maps each morphism $f\colon K \to L$ in $\cK$ to a
left Quillen functor $ F(f) \colon F(K) \to F(L)$. We write $f_!$ for
$F(f)$ and $f^*$ for its right adjoint. We say that a morphism in
$\cK\int F$ consisting of $f\colon K \to L$ and $f_!(U) \to V$ is
\begin{itemize}
\item an \emph{integral cofibration} if $f$ and $f_!(U) \to V$ are
  cofibrations,
\item an \emph{integral fibration} if $f$ and $U \to f^*(V)$ are
  fibrations, and
\item an \emph{integral weak equivalence} if $f$ is a weak equivalence
  and for any cofibrant replacement $U^{\mathrm{cof}} \to U$ in $F(K)$
  the composite of $f_!(U^{\mathrm{cof}}) \to f_!(U) \to V$ is a weak
  equivalence.
\end{itemize}
\end{definition}
It is shown in~\cite[Theorem
3.0.12]{Harpaz-P_Grothendieck-construction} that these classes of maps
form a model structure when $F$ is a \emph{proper relative}
pseudofunctor (in the language of~\cite[\S
3]{Harpaz-P_Grothendieck-construction}). This amounts to requiring
that the conditions of
Lemma~\ref{lem:extension-restriction-R-homotopical} hold for
$F$. Applying this discussion to the
pseudofunctor~\eqref{eq:retractive-spaces-as-Grothendieck-construction} thus provides the following statement. 
\begin{proposition}\label{prop:R-model-as-Grothendieck-construction}
  These classes form a model structure on the Grothendieck
  construction of $K \mapsto \cS_K$, called the \emph{integral}
  model structure.  Under the equivalence of
  Lemma~\ref{lem:retractive-spaces-as-Grothendieck-construction}, it
  coincides with the model structure on $\cS_{\cR}$ considered earlier.\qed
\end{proposition}

\subsection{The monoidal structure on retractive spaces}
Let $(U,K)$ and $(V,L)$ be retractive spaces. Their structure maps induce
a commutative diagram 
 \begin{equation}\label{eq:fiberwise-smash-as-colimt}
\xymatrix@-1pc{
U\times V & U\times L \ar[l] \ar[r] & K \times L \\
U\times V \ar[u]^{=} \ar[d]_{=} & K\times L \ar[u] \ar[d] \ar[l] \ar[r]& K \times L \ar[u]_{=} \ar[d]^{=} \\
U\times V & K\times V \ar[l] \ar[r]& K \times L \ .
}
\end{equation}
\begin{definition}\label{def:fiberwise-smash}
  Let $U\barsm V$ be the colimit of~\eqref{eq:fiberwise-smash-as-colimt}.
  It is the total space of the  \emph{fiberwise smash product}
  $(U,K)\barsm (V,L) = (U\barsm V, K \times L)$ of $(U,K)$ and
  $(V,L)$. Its structure maps are induced by the universal property of the
  pushout and the structure maps of $(U,K)$ and $(V,L)$.  (We stress
  that we use the symbol $\barsm$ both for the fiberwise smash product
  in $\cS_{\cR}$ and for its total space.)
\end{definition}

Equivalently, $U\barsm V$ is the iterated pushout obtained by first
forming the vertical pushouts in~\eqref{eq:fiberwise-smash-as-colimt} and
then forming the pushout of the resulting diagram
\begin{equation}\label{eq:fiberwise-smash-as-iterated-colimt} U \times V \ot  \left((U \times L)\cup_{ K \times L}(K \times V)\right) \to K\times L. \end{equation}

Writing $W^V$ for the cotensor in $\cS$, there is an internal Hom-functor
\begin{equation}\label{eq:cotensor} \Hom_{\cR}\colon \cS_{\cR}^{\op} \times \cS_{\cR} \to \cS_{\cR}, \quad ((V,L),(W,M)) \mapsto (W^V\times_{(W^L \times_{M^L} M^V)}M^L,M^L)  
\end{equation}
where the latter pair of spaces has the obvious structure maps making it an object of $\cS_{\cR}$. 

\begin{proposition}\label{prop:R-closed-sym-monoidal}
  The fiberwise smash product is a closed symmetric monoidal product on the
  category of retractive spaces $\cS_{\cR}$ with monoidal unit $S^0$. In particular, there is a natural isomorphism 
\begin{equation} \label{eq:R-closed-sym-monoidal}
\cS_{\cR}((U,K)\barsm (V,L), (W,M)) \iso \cS_{\cR}((U,K),\Hom_{\cR}((V,L),(W,M))). 
\end{equation}
\end{proposition}
\begin{proof}
It is clear that the fiberwise smash product is symmetric monoidal. To establish the adjunction isomorphism~\eqref{eq:R-closed-sym-monoidal}, we note that on both sides a morphism is given by a pair of maps $U \times V \to W$ and $K\times L \to M$ in $\cS$ such that the following three diagrams commute:
\[\xymatrix@-1pc{
U\times V \ar[d] \ar[r] & W \ar[d]  & U\times L \ar[d] \ar[r] & U\times V \ar[dr] & & K\times V \ar[r] \ar[d] & U \times V \ar[rd] & \\
 K\times L \ar[r] & M &  K\times L \ar[r] & M \ar[r] & W & K\times L \ar[r] & M \ar[r] & W 
}\qedhere\]
\end{proof}
\begin{remark}\label{rem:internal-hom}
The Hom-object considered here is an ``external'' one in that the base space of $\Hom_{\cR}((V,L),(W,L))$ is $L^L$ and not $L$. It does not appear to give rise to an ``internal'' Hom-object in the category $\cS_L$. While the latter ``internal'' Hom is not relevant for our work, it plays an important role in the approach by May--Sigurdsson. Implementing it requires them to deal with considerably more involved point set topological issues (see e.g.~\cite[\S 1.3]{May-S_parametrized}).
\end{remark}

For later reference, we describe some $\barsm$-products of retractive
spaces considered so far. In the notation of
Construction~\ref{constr:spaces-retractive-spaces} and
Lemma~\ref{lem:lifting-properties-in-R}, we have:
\begin{align}
\label{eq:barsmiso-zero}  (A,A) \barsm (V,L) &\iso (A\times L,A\times L)\\
\label{eq:barsmiso-one}   (B\amalg A,B) \barsm (V,L) &\iso ((B\times L)\amalg_{(A\times L)}(A\times V),B\times L) \\
\label{eq:barsmiso-three} (B\amalg A,B) \barsm (B'\amalg A', B') &\iso ( (B\times B')\amalg (A\times A'), (B\times B'))
\end{align} 

\begin{proposition}\label{prop:pushout-product-for-cScR}
The model category $\cS_{\cR}$ satisfies the pushout product axiom. 
\end{proposition}
\begin{proof}
  By \cite[Lemma 3.5(i)]{Schwede-S_algebras_modules}, it suffices to
  verify the pushout product axiom for the generating (acyclic)
  cofibrations $I_{\cR}$ and $J_{\cR}$ in
  Proposition~\ref{prop:model-on-R}. Using the above
  isomorphisms~\eqref{eq:barsmiso-zero}, \eqref{eq:barsmiso-one}
  and~\eqref{eq:barsmiso-three}, the claim follows from the pushout
  product axiom for $\cS$.
\end{proof}

\begin{lemma}\label{lem:cobase-barsm-map}
  Let $(U,K)$ and $(V,L)$ be retractive spaces, and let $f \colon K
  \to K'$ and $g \colon L \to L'$ be maps in $\cS$. Then
  there is a natural isomorphism 
\[  (f_!(U),K')\barsm (g_!(V), L') \xrightarrow{\iso} ((f\times g)_! (U\barsm V), K' \times L') \ .\]
\end{lemma}
\begin{proof}
The products $f_!(U) \times g_!(V)$ and $K'\times L'$ can be identified
with the colimits of the following diagrams: 
\[{\scriptscriptstyle  \vcenter{\xymatrix@-1.4pc{
K'\!\!\times\!\! V&  K'\!\!\times\!\! L\ar[r] \ar[l] &  K'\!\!\times\!\! L'\\
K\!\!\times\!\! V\ar[u] \ar[d] &K\!\!\times\!\! L  \ar[r] \ar[l] \ar[u] \ar[d] &K\!\!\times\!\!  L' \ar[u] \ar[d] \\
U\!\!\times\!\! V &U\!\!\times\!\!  L \ar[r] \ar[l] &U\!\!\times\!\!  L'}}}
\qquad
{\scriptscriptstyle \vcenter{\xymatrix@-1.4pc{
K'\!\!\times\!\! L&  K'\!\!\times\!\! L\ar[r] \ar[l] &  K'\!\!\times\!\! L'\\
K\!\!\times\!\! L\ar[u] \ar[d] &K\!\!\times\!\! L  \ar[r] \ar[l] \ar[u] \ar[d] &K\!\!\times\!\!  L' \ar[u] \ar[d] \\
K\!\!\times\!\! L &K\!\!\times\!\!  L \ar[r] \ar[l] &K\!\!\times\!\!  L'}}}\]
Moreover, $(f_!(U)\times L')\cup_{K'\times L'}(K'\times g_!(V))$
is isomorphic to the colimit of the following diagram:
\[
{\scriptscriptstyle 
\left(\!\! \vcenter{\xymatrix@-1.4pc{
K'\!\!\times\!\! L&  K'\!\!\times\!\! L\ar[r] \ar[l] &  K'\!\!\times\!\! L'\\
K\!\!\times\!\! L\ar[u] \ar[d] &K\!\!\times\!\! L  \ar[r] \ar[l] \ar[u] \ar[d] &K\!\!\times\!\!  L' \ar[u] \ar[d] \\
U\!\!\times\!\! L &U\!\!\times\!\!  L \ar[r] \ar[l] &U\!\!\times\!\!  L'}}\!\!\right)
\ot
\left(\!\! \vcenter{\xymatrix@-1.4pc{
K'\!\!\times\!\! L&  K'\!\!\times\!\! L\ar[r] \ar[l] &  K'\!\!\times\!\! L'\\
K\!\!\times\!\! L\ar[u] \ar[d] &K\!\!\times\!\! L  \ar[r] \ar[l] \ar[u] \ar[d] &K\!\!\times\!\!  L' \ar[u] \ar[d] \\
K\!\!\times\!\! L &K\!\!\times\!\!  L \ar[r] \ar[l] &K\!\!\times\!\!  L'}}\!\!\right)
\to
\left(\!\! \vcenter{\xymatrix@-1.4pc{
K'\!\!\times\!\! V&  K'\!\!\times\!\! L\ar[r] \ar[l] &  K'\!\!\times\!\! L'\\
K\!\!\times\!\! V\ar[u] \ar[d] &K\!\!\times\!\! L  \ar[r] \ar[l] \ar[u] \ar[d] &K\!\!\times\!\!  L' \ar[u] \ar[d] \\
K\!\!\times\!\! V &K\!\!\times\!\!  L \ar[r] \ar[l] &K\!\!\times\!\!  L'}}\!\!\right)
}\]
Hence taking the colimit of each of the five $3\times 3$-diagrams and then forming the iterated colimit as in~\eqref{eq:fiberwise-smash-as-iterated-colimt} provides $f_!(U)\barsm g_!(V)$. Since colimits commute among each other, we may alternatively first form the colimit as in~\eqref{eq:fiberwise-smash-as-iterated-colimt} in each of the 9 entries of the $3\times 3$-diagrams and then form the colimit
of the resulting $3\times 3$-diagram:
\[{\scriptscriptstyle  \xymatrix@-1.4pc{
K'\!\times\! L&  K'\!\times\! L\ar[r] \ar[l] &  K'\!\times\! L'\\
K\!\times\! L\ar[u] \ar[d] &K\!\times\! L  \ar[r] \ar[l] \ar[u] \ar[d] &K\!\times\!  L' \ar[u] \ar[d] \\
U\!\barsm\! V &K\!\times\!  L \ar[r] \ar[l] &K\!\times\!  L'\ }}\]
The colimit of the latter diagram is isomorphic to $(f\times g)_!(U\barsm V)$. 
\end{proof}

The fiberwise smash product also commutes with base change: 
\begin{lemma}\label{lem:fiberwise-smash-pullback} 
  Let $(U,K)$ and $(V,L)$ be retractive spaces, and let $f \colon K'
  \to K$ and $g \colon L' \to L$ be maps in $\cS$. Then there
is a natural isomorphism \begin{equation}\label{eq:pullback-barsm-map}
    (f^*U,K')\barsm (g^*V, L') \to ((f\times g)^* (U \barsm V), K'
    \times L')\ .\end{equation}  
\end{lemma}
\begin{proof} It is clear that $(f \times g)^*(U \times V) \iso f^*(U)
  \times g^*(V)$. Moreover, as a functor between the underlying categories
  of sets, $(f \times g)^*$ preserves the pushouts that are used to
  form the fiberwise smash products. This shows the claim if $\cS =
  \sset$. 

  When $\cS = \tp$, the argument is more involved since forming
  colimits in compactly generated weak Hausdorff spaces may change the
  underlying sets.  By~\cite[Proposition 1.3]{Lewis_fibre-spaces},
  base change along a map of compactly generated weak Hausdorff spaces
  preserves colimits in compactly generated spaces.  The pushout in
  compactly generated spaces of a diagram of compactly generated weak
  Hausdorff spaces in which one map is a closed inclusion coincides
  with its pushout in compactly generated weak Hausdorff spaces, and
  the cobase change of the closed inclusion is again a closed
  inclusion in this case~\cite[App. A, Proposition 7.5]{Lewis_thesis}. The
  structure map from the base to the total space of a retractive space
  (in compactly generated weak Hausdorff spaces) is a closed
  inclusion~\cite[Lemma 1.6.2]{May-S_parametrized} and closed
  inclusions are preserved under products. Hence the pushout defining
  $(U \times L)\cup_{ K \times L}(K \times V)$ is preserved under base
  change. Because $K\to U$ and $L\to V$ are closed inclusions, so is
  $(U \times L)\cup_{ K \times L}(K \times V) \to U \times V$, and it
  follows from the description
  in~\eqref{eq:fiberwise-smash-as-iterated-colimt} that the pushout
  defining $U\barsm V$ is preserved under base change.
\end{proof}

\begin{remark} Given retractive spaces $(U,K)$ and $(L,V)$ and points
  $x \in K$ and $y \in L$, the proposition shows that the fiber of
  $U\barsm V \to K \times L$ over $(x,y)$ is isomorphic to the smash
  product of the fiber of $U\to K$ over $x$ with the fiber of $V \to
  L$ over $y$. This justifies the name fiberwise smash product.
\end{remark}
\subsection{Simplicial structure of $\cS_K$} The symmetric monoidal
structure on $\cS_{\cR}$ can be used to define simplicial and pointed
simplicial structures on the category $\cS_{K}$ for a fixed space
$K$. If $Q$ is an unbased simplicial set, we define a functor
\[ Q \tensor - \colon \cS_K \to \cS_K, \qquad (U,K) \mapsto (Q_+,*) \barsm (U,K) .\]
Here we implicitly compose with the cobase change along $K \times \{*\} \xrightarrow{\iso} K$ and apply the geometric realization to $Q$ when working with topological spaces.
\begin{proposition}\label{prop:S_K-simplicial}
This action equips $\cS_K$ with the structure of a simplicial model category. 
\end{proposition}
\begin{proof}
An application of~\cite[Lemma II.2.4]{Goerss-J_simplicial} shows that 
$\cS_K$ becomes a simplicial category. The compatibility with the model 
structure follows from Proposition~\ref{prop:pushout-product-for-cScR} 
and the compatibility of the model structures on $\cS_K$ and $\cS_{\cR}$. 
\end{proof}
Since $\cS_K$ is pointed, its simplicial structure induces
a pointed simplicial structure. The tensor of $(U,K)$
with a pointed simplicial set $P$ is the pushout of $(K,K) \ot
\{*\} \tensor (U,K) \to P \tensor (U,K)$. It follows that this
tensor is $(P,*) \barsm (U,K)$. Consequently, the cotensor is $\Hom_{\cR}((P,*),(U,K))$, and we deduce from~\eqref{eq:cotensor} that it has the total space
$U^P\times_{(U \times  K^P)}K$. So a point in the total space consists of a
map $h\colon P \to U$ whose image is contained in a single fiber and which
sends  the basepoint of $P$ to the canonical basepoint of the fiber.

\section{Twisted symmetric spectra}\label{twisepc}
We will now introduce a generalized form of symmetric spectra for
which we allow the individual levels of a symmetric spectrum $X$ to
take values in different categories. We will also construct level
model structures on these categories that we will use in
Section~\ref{sec:local-model} to build the local model structures we are really
after. Both the level and the local model structures come in an
absolute and a positive version with different cofibrations. The
positive version will be needed to get a lifted model structure on
commutative parametrized ring spectra (see Section~\ref{subsec:commutative-monoids}).

Let $\cI$ be the category with objects the finite sets
$\bld{m}=\{1,\dots, m\}$, $m \geq 0$, and morphisms the injections.
The ordered concatenation $\bld{m}\concat \bld{n} = \bld{m+n}$ of
finite sets makes $\cI$ a symmetric strict monoidal category with unit
$\bld{0}$. Its symmetry isomorphism is the shuffle
$\chi_{m,n}\colon \bld{m}\concat\bld{n} \to \bld{n}\concat\bld{m}$
moving the first $m$ elements past the last $n$ elements.

We first recall some notions needed for a description of symmetric
spectra using the category $\cI$ (see e.g.~\cite[\S
3.1]{Schlichtkrull_Thom-symmetric}). This viewpoint will be convenient
for the discussion of convolution products in Section~\ref{sec:smash}. Given a finite set $P$, we let $S^{P} = \bigwedge_{P}S^1$ be the $P$ fold
smash power of $S^1$. If $\alpha\colon\bld{m}\to\bld{n}$
is a morphism in $\cI$, we write $\bld n-\alpha$ for the complement of
its image.  The canonical extension of $\alpha$ to a bijection
$\bld m\amalg(\bld n-\alpha)\to\bld n$ induces a homeomorphism
\begin{equation}\label{eq:reindexing-basic} S^{\bld{m}} \sm
S^{\bld{n}-\alpha}\xrightarrow{\iso} S^{\bld{n}}.
\end{equation} 
More generally, if $\alpha\colon \bld{m}\to\bld{n}$ and $\beta\colon\bld{n}\to\bld{p}$ are
composable morphisms in $\cI$, then the canonical bijection
$(\bld{n}-\alpha)\amalg(\bld{p}-\beta)\to\bld{p}-\beta\alpha$ induces
a homeomorphism 
\begin{equation}\label{eq:reindexing-composite} S^{\bld{n}-\alpha} \sm
S^{\bld{p}-\beta} \xrightarrow{\iso} S^{\bld{p}-\beta\alpha}.
\end{equation}

\subsection{Quillen \texorpdfstring{$\cI$}{I}-categories}
The next definition again uses the language of pseudofunctors~\cite[Definition 7.5.1]{Borceux-1} with values in $\mathrm{Cat}$.

\begin{definition}
  A \emph{Quillen $\cI$-category} is a pseudofunctor
  $\cC\colon \cI \to \mathrm{Cat}$ with each $\cC_{\bld{m}}$
  equipped with a cofibrantly generated model structure and each
  $\alpha_!\colon \cC_{\bld{m}} \to \cC_{\bld{n}}$ left Quillen.
\end{definition}
\begin{remark}
 This notion of a Quillen $\cI$-category corresponds to the ``right Quillen presheaves'' of \cite[Definition 2.21]{Barwick_left-right} given by contravariant pseudofunctors on $\cI$ that send maps to right Quillen functors. We use the present terminology since our examples below make it more natural to treat the left adjoints as primary data.
\end{remark}

\begin{definition}
Let $\cC\colon \cI \to \mathrm{Cat}$ be a Quillen $\cI$-category. Its \emph{section category} $\cC^{\cI}$ has as objects $X$ families of objects $X_{\bld{m}}$ in $\cC_{\bld{m}}$ equipped with structure maps ${\alpha_!X_{\bld{m}} \to X_{\bld{n}}}$ for every $\alpha \colon \bld{m} \to \bld{n}$ in $\cI$ such that the structure map $(\id_{\bld m})_!X_{\bld m}\to X_{\bld m}$ is the isomorphism given by the pseudofunctor and such that the square
\[\xymatrix@-1pc{\beta_! (\alpha_! X_{\bld{m}}) \ar[r] \ar[d]_{\iso} & \beta_! X_{\bld{n}} \ar[d] \\
  (\beta \alpha)_!X_{\bld{m}} \ar[r] & X_{\bld{p}}} \] commutes for
all maps $\alpha \colon \bld{m} \to\bld{n}$ and $\beta\colon \bld{n}
\to \bld{p}$ in $\cI$. Here the left hand vertical map is the coherence isomorphism of the pseudofunctor $\cC$, while the three other maps
are the structure maps associated with $\alpha, \beta$ and $\beta\alpha$. Morphisms $X \to Y$ in $\cC^{\cI}$ are families of morphisms $X_{\bld{m}} \to Y_{\bld{m}}$ that make the obvious squares commutative. 
\end{definition}

If $\cC\colon \cI \to \mathrm{Cat}$ is a Quillen $\cI$-category and $\bld{m}$ is an object in $\cI$, we get an adjunction 
\begin{equation}\label{eq:free-evaluation-adjunction-pseudofunctor} F_{\bld{m}} \colon \cC_{\bld{m}} \rightleftarrows \cC^{\cI} \colon \Ev_{\bld{m}}
\end{equation}
with right adjoint the evaluation functor sending $X$ in $\cC^{\cI}$ to $X_{\bld{m}}$. The left adjoint is given in level $\bld{n}$ by
\begin{equation}\label{eq:free-explicit} (F_{\bld{m}}(Z))_{\bld{n}} = \textstyle\coprod_{\alpha\colon \bld{m} \to \bld{n}} \alpha_!Z \end{equation}
where the coproduct is indexed over $\cI(\bld{m},\bld{n})$ and formed in $\cC_{\bld{n}}$. The composition in $\cI$ induces structure maps turning the $(F_{\bld{m}}(Z))_{\bld{n}}$ into an object of $\cC^{\cI}$. 

\subsection{Level model structures}
Let $f\colon X \to Y$ be a morphism in the section category of a
Quillen $\cI$-category. Then $f$ is an absolute \emph{level fibration}
(resp.\ \emph{level equivalence}) if
$f_{\bld{m}} \colon X_{\bld{m}} \to Y_{\bld{n}}$ is a fibration
(resp.\ weak equivalence) in $\cC_{\bld{m}}$ for all
$\bld{m}$. Positive level fibrations and level weak equivalences are
defined by only requiring this condition if $|\bld{m}|\geq 1$. The
absolute (resp.\ positive) level cofibrations are the maps with the
left lifting properties against all maps that are both absolute
(resp.\ positive) level fibrations and level equivalences. The next
statement is analogous to~\cite[Theorem 2.28]{Barwick_left-right}.
\begin{proposition}\label{prop:level-model-on-section}
  Let $\cC\colon \cI \to \mathrm{Cat}$ be a Quillen $\cI$-category. Then the above classes of maps form absolute and positive
  level model structures on $\cC^{\cI}$.  Both level model structures
  are cofibrantly generated, and they are proper if each $\cC_{\bld{m}}$ is.
\end{proposition}
\begin{proof}
For the absolute level model structure, we define 
\begin{equation}
I = \{ F_{\bld{m}}(i)\, |\, \bld{m} \in \cI, i \in I_{\bld{m}}\}\quad\text{ and }\quad J = \{ F_{\bld{m}}(j)\, |\, \bld{m} \in \cI, j \in J_{\bld{m}}\}
\end{equation}
where $I_{\bld{m}}$ (resp.\ $J_{\bld{m}}$) is a set of generating
(resp.\ generating acyclic) cofibrations for $\cC_{\bld{m}}$. Now we
apply the recognition theorem for cofibrantly generated model
structures~\cite[Theorem 2.1.19]{Hovey_model}. The least obvious
condition to check is that the relative $J$-cell complexes are absolute level
equivalences. To see this, we use that colimits in $\cC^{\cI}$ are
formed levelwise and deduce from~\eqref{eq:free-explicit} that for $j
\in J_{\bld{m}}$ and $\bld{n}$ in $\cI$, the map
$F_{\bld{m}}(j)_{\bld{n}}$ is an acyclic cofibration in
$\cC_{\bld{m}}$ because the $\alpha_!$ are left Quillen. The fact that
transfinite compositions of cobase changes of acyclic cofibrations in
$\cC_{\bld{n}}$ are weak equivalences shows that the relative $J$-cell
complexes are absolute level equivalences.

To treat the positive level model structure, we write $\cI_{\geq 1}$ for the full subcategory of $\cI$ on the objects $\bld{m}$ with $|\bld{m}|\geq 1$, define
\begin{equation}
I = \{ F_{\bld{m}}(i)\, |\, \bld{m} \in \cI_{\geq 1}, i \in I_{\bld{m}}\}\quad\text{ and }\quad J = \{ F_{\bld{m}}(j)\, |\, \bld{m} \in \cI_{\geq 1}, j \in J_{\bld{m}}\},
\end{equation}
and argue as before. The statement about the relative $J$-cell complexes in the absolute case implies the one in the positive case. 
\end{proof}

\begin{example}\label{ex:pseudo}
We now discuss how various well-known categories of relevance for us can be expressed in terms of Quillen $\cI$-categories $\cC\colon \cI \to \mathrm{Cat}$. Here $\cS$ and $\cS_*$ are equipped with the standard model structures, and $\cS_{\cR}$ and $\cS_K$ are equipped with the model structures discussed in the previous section. 
\begin{enumerate}[(i)]
\item Let $\cC_{\bld{m}} = \cS$ and $\alpha_! = \id$ for all $\bld{m}$
  and $\alpha$. Then the section category $\cC^{\cI}$ is the functor
  category $\cS^{\cI}$ of $\cI$-spaces, and the model structures of
  Proposition~\ref{prop:level-model-on-section} are the absolute and
  positive level model structure on $\cI$-spaces that arise from~\cite[Proposition 6.7]{Sagave-S_diagram}.
\item For each $\bld{m}$ in $\cI$, let $\cC_{\bld{m}}$ be the category
  of based spaces $\cS_*$.  The functor
  $\alpha_! \colon \cS_* \to \cS_*$ induced by
  $\alpha\colon \bld{m} \to \bld{n}$ is defined to be
  $- \sm S^{\bld{n}-\alpha}$, the smash product with the sphere
  $S^{\bld{n}-\alpha}$ indexed by the finite set $\bld{n}-\alpha$. We
  note that the coherence isomorphism of the smash product and the
  isomorphisms~\eqref{eq:reindexing-composite} equip $\cC$ with the
  structure of a pseudofunctor. The section category of this
  pseudofunctor is equivalent to the usual category of symmetric
  spectra $\Spsym{}$. Here the structure maps in the definition of the
  section category correspond to the generalized structure maps of
  symmetric spectra (see e.g.~\cite[\S
  3.1]{Schlichtkrull_Thom-symmetric}). Under this equivalence of
  categories, the model structures of
  Proposition~\ref{prop:level-model-on-section} correspond to the
  absolute and positive level model structures on $\Spsym{}$.
\item Analogously to (i), the section category of the constant
  pseudofunctor with value $\cS_{\cR}$ is equivalent to
  $\cS^{\cI}_{\cR}$, the category of $\cI$-diagrams in retractive
  spaces which is in turn equivalent to the category of retractive
  objects in $\cI$-spaces. Under the latter equivalence, the model
  structures of Proposition~\ref{prop:level-model-on-section}
  correspond to the model structures on retractive objects in the
  absolute or positive level model structures on
  $\cS^{\cI}$ that arise by the argument in the proof of Proposition~\ref{prop:model-str-on-R}. 
\item Let $X$ be an $\cI$-space. We define $\cC_{\bld{m}}$ to be
  $\cS_{X(\bld{m})}$, the category of spaces over and under
  $X(\bld{m})$, and $\alpha_! = (X(\bld{m}) \to X(\bld{n}))_! \colon
  \cS_{X(\bld{m})} \to \cS_{X(\bld{n})}$ as
  in~\eqref{eq:res-ext-adj-retractive-spaces}. The universal property
  of the pushout gives rise to coherence isomorphisms making this
  a pseudofunctor. Its section category is equivalent to
  $\cS^{\cI}_{X}$, the category of $\cI$-spaces over and under
  $X$. The model structures of
  Proposition~\ref{prop:level-model-on-section} correspond to those
  induced by the absolute and positive $\cI$-model structure on the
  category of objects over and under $X$ in the usual way. 
\end{enumerate}
\end{example}
By mixing (ii) and (iii), we obtain a pseudofunctor with each
$\cC_{\bld{m}}$ the category of retractive spaces $\cS_{\cR}$ and with
$\alpha_! \colon \cS_{\cR} \to \cS_{\cR}$ being the functor
$- \barsm S^{\bld{n}-\alpha}$ (where $S^{\bld{n}-\alpha}$ is a
shorthand notation for $(S^{\bld{n}-\alpha},*)$, see
Construction~\ref{constr:spaces-retractive-spaces}).
\begin{definition}\label{def:spsymR}
  The section category of this pseudofunctor defines the \emph{category of
    symmetric spectra in retractive spaces} $\Spsym{\cR}$.
\end{definition}
Explicitly, an object $(E,X)$ in $\Spsym{\cR}$ is a sequence of
retractive spaces $(E,X)_{\bld{m}}$ for $\bld{m}$ in $\cI$ with
structure maps
$\alpha_! \colon (E,X)_{\bld{m}} \barsm S^{\bld{n}-\alpha} \to
(E,X)_{\bld{n}}$
for each $\alpha\colon \bld{m} \to\bld{n}$ in $\cI$ such that the
obvious diagrams commute.  As we will see (and heavily exploit) below,
the base spaces $X(\bld{m})$ of $(E,X)$ assemble to an
$\cI$-space. Analogous to the discussion in (ii), one can check that
$\Spsym{\cR}$ is equivalent to the category $\Spsym{}(\cS_{\cR},S^1)$
of symmetric spectrum objects in $\cS_{\cR}$ with suspension functor
$-\barsm S^1$;
see~\cite[Definition~7.2]{Hovey_symmetric-general}. Under this
equivalence, the absolute level model structure on $\Spsym{\cR}$
corresponds to the level model structure established
in~\cite[Theorem~8.2]{Hovey_symmetric-general}.

Let $\cC$ and $\cD$ be Quillen $\cI$-categories and let
$\Phi\colon \cC \to \cD$ be a \emph{pseudo-natural
  transformation}~\cite[Definition 7.5.2]{Borceux-1}, i.e., a family
of functors $\Phi_{\bld{m}}\colon \cC_{\bld{m}} \to \cD_{\bld{m}}$
together with natural isomorphisms
$\alpha_!\Phi_{\bld{m}} \xrightarrow{\iso} \Phi_{\bld{n}}\alpha_!$ of
functors $\cC_{\bld{m}}\to \cD_{\bld{n}}$ that are compatible with the
coherence isomorphisms of $\cC$ and $\cD$. Then $\Phi$ induces a
functor $\Phi \colon \cC^{\cI} \to \cD^{\cI}$ of section
categories. For $X$ in $\cC^{\cI}$, the object $\Phi(X)$ consists of
the family of objects $\Phi_{\bld{m}}(X_{\bld{m}})$ together with
structure maps
\[\alpha_!(\Phi_{\bld{m}}(X_{\bld{m}})) \xrightarrow{\iso}
\Phi_{\bld{n}}(\alpha_!(X_{\bld{m}})) \to
\Phi_{\bld{n}}(X_{\bld{n}})\]
induced by the structure maps of $X$ and the coherence isomorphism of
$\Phi$.

\begin{lemma}\label{lem:adjunction-on-section-cat}
  Let $\Phi\colon \cC \to \cD$ be a pseudo-natural transformation of
  Quillen $\cI$-categories with each
  $\Phi_{\bld{m}} \colon \cC_{\bld{m}} \to \cD_{\bld{m}}$  a left
  Quillen functor. Then ${\Phi\colon \cC^{\cI} \to \cD^{\cI}}$ is a
  left Quillen functor with respect to the absolute and positive level model
  structures.
\end{lemma}
\begin{proof}
  We fix right adjoints $\Psi_{\bld{n}}$ and units and counits for
  each of these adjunctions. The inverses of the coherence
  isomorphisms $\alpha_! \Phi_{\bld{m}} \to \Phi_{\bld{n}}\alpha_!$
  and the units and counits give rise to natural maps
  $\alpha_! \Psi_{\bld{m}} \to \Psi_{\bld{n}}\alpha_!$ that equip the
  $\Psi_{\bld{m}}$ with the structure of a \emph{left op-lax natural
    transformation} in the sense of \cite[Definition
  3.1.2]{Thomason-homotopy-colimt}. The $\Psi_{\bld{m}}$ induce a
  functor $\Psi \colon \cD^{\cI}\to\cC^{\cI}$ on section categories
  where $\Psi(Y)$ has structure maps
  $\alpha_! \Psi_{\bld{m}}(Y_{\bld{m}}) \to
  \Psi_{\bld{n}}\alpha_!(Y_{\bld{m}}) \to \Psi_{\bld{n}}(Y_{\bld{n}})$.
  Then $\Phi\colon \cC^{\cI} \rightleftarrows \cD^{\cI} \colon \Psi$
  is an adjunction, and it is immediate from the definition of the
  fibrations and weak equivalences that $\Phi$ is right Quillen.
\end{proof}
\begin{construction}\label{const:various-adjunctions}
  We will now apply Lemma~\ref{lem:adjunction-on-section-cat} to various pseudo-natural
  transformations relating the Quillen $\cI$-categories appearing in
  Example~\ref{ex:pseudo}(i)-(iii) and in
  Definition~\ref{def:spsymR}. Using the functors of
  Construction~\ref{constr:spaces-retractive-spaces} and
  Lemma~\ref{lem:iota-b-left-right-adjoint}, the values of the
  pseudo-natural transformations are the horizontal arrows in the
  following diagrams where all squares commute up to
  isomorphism:
\[\xymatrix@-1pc{
\bld{m} \ar[d]^{\alpha}  && \cS \ar^{\iota_t}[rr] \ar[d]^{\id} && \cS_{\cR}  \ar[d]^{\id} && \cS \ar^{\iota_b}[rr] \ar[d]^{\id} && \cS_{\cR} \ar^{-\barsm S^{\bld{m}}}[rr] \ar[d]^{\id} && \cS_{\cR} \ar^(.6){\pi_b}[rrr] \ar[d]^{-\barsm S^{\bld{n}-\alpha}} &&& \cS\ar[d]^{\id}  \\
\bld{n}  && \cS \ar^{\iota_t}[rr]&& \cS_{\cR} && \cS \ar^{\iota_b}[rr]&& \cS_{\cR} \ar^{-\barsm S^{\bld{n}}}[rr] && \cS_{\cR} \ar^(.6){\pi_b}[rrr] &&& \cS 
}\]
Hence we get left Quillen functors 
\begin{equation}\label{eq:various-left-adjoints} \xymatrix@-1.2pc{&& \cS^{\cI} \ar^{\iota_t}[rr] && \cS_{\cR}^{\cI} && \cS^{\cI} \ar^{\iota_b}[rr]  && \cS_{\cR}^{\cI} \ar^{\bS^{\cI}_{\cR}}[rr]  && \Spsym{\cR} \ar^-{\pi_b}[rrr]  &&& \cS^{\cI} \ .
}
\end{equation}\end{construction}
Explicitly, the values of the functors $\iota_t$ and $\iota_b$ on an $\cI$-space $X$ are 
\[ \iota_t(X)(\bld{m}) = (X(\bld{m})\textstyle\amalg X(\bld{m}),
X(\bld{m})) \quad\text{ and }\quad \iota_b(X)(\bld{m}) =
(X(\bld{m}),X(\bld{m}))\ , \] and their right adjoints
$\cS^{\cI}_{\cR} \to \cS^{\cI}$ are the projections to the base and
total $\cI$-spaces. The functor $\bS^{\cI}_{\cR}$ sends an object
$(Z,X)$ in $\cS_{\cR}^{\cI}$ to the symmetric spectrum in $\cS_{\cR}$
given in degree $m$ by $(Z(\bld{m}),X(\bld{m}))\barsm
S^{\bld{m}}$. Its right adjoint $\Omega^{\cI}_{\cR}$ is
given by $\Omega^{\cI}_{\cR}(E,X)(\bld{m}) =
\Hom_{\cR}((S^{\bld{m}},*),(E,X)(\bld{m}))$.  
We obtain two composite
functors 
\[ \bS^{\cI}_b = \bS^{\cI}_{\cR} \circ \iota_b \colon \cS^{\cI} \to \Spsym{\cR}\quad \text{ and } \quad  \bS^{\cI}_t = \bS^{\cI}_{\cR} \circ \iota_t \colon \cS^{\cI} \to \Spsym{\cR}.  \]
Their evaluations on an $\cI$-space $X$ are given in level $\bld{m}$ by 
\begin{equation}\label{eq:bsIb-and-bSIt-explicit}
\bS^{\cI}_b[X]_m = (X(\bld{m}),X(\bld{m}))\quad \text{ and }\quad \bS^{\cI}_t[X]_m = (X(\bld{m}) \times S^{\bld{m}},X(\bld{m})).
\end{equation}
Their right adjoints $\pi_b$ and $\Omega^{\cI}_t$ are given by composing
$\Omega^{\cI}_{\cR}$ with the projection to the base and total $\cI$-space. Finally, the functor $\pi_b$ in~\eqref{eq:various-left-adjoints} is
the projection to the underlying $\cI$-space $X$ of a symmetric
spectrum in retractive spaces $(E,X)$.
\begin{lemma}\label{lem:iota-b-left-right-adjoint-I}
  The functor $\bS^{\cI}_b \colon \cS^{\cI} \to \Spsym{\cR}$ is both
  left and right adjoint to $\pi_b$, and both left and right Quillen
  with respect to the absolute and positive level model
  structures. 
\end{lemma}
\begin{proof} The adjunction statement follows from
  Lemma~\ref{lem:iota-b-left-right-adjoint}. Both
  $(\bS^{\cI}_b,\pi_b)$ and $(\pi_b,\bS^{\cI}_b)$ are Quillen
  adjunctions since $ \bS^{\cI}_b = \bS^{\cI}_{\cR} \circ \iota_b $
  and $\pi_b$ are left Quillen.
\end{proof}

\subsection{The category of $X$-relative symmetric spectra}
Let $X$ be an $\cI$-space. By combining parts (ii) and (iv) of
Example~\ref{ex:pseudo}, we get another Quillen $\cI$-category with values $\cC_{\bld{m}} = \cS_{X(\bld{m})}$ on objects. In
this case, we define $\alpha_!\colon \cS_{X(\bld{m})} \to
\cS_{X(\bld{n})}$ to be the composite
\begin{equation}\label{eq:Quillen-I-cat-for-SpsymX}
\cS_{X(\bld{m})} \xrightarrow{X(\alpha)_!}  \cS_{X(\bld{n})} \xrightarrow{- \barsm (S^{\bld{n}-\alpha})}  \cS_{X(\bld{n})}\ .
\end{equation}
The universal property of the pushout, the coherence isomorphism of
the symmetric monoidal product $\barsm$ on $\cS_{\cR}$ and the
isomorphisms~\eqref{eq:reindexing-composite} provide the coherence
isomorphisms for this pseudofunctor.

\begin{definition}\label{def:X-relative-sym-spectra}
  Let $X$ be an $\cI$-space. Then the section category of the previous
  Quillen $\cI$-category is the category of \emph{$X$-relative
    symmetric spectra} $\Spsym{X}$. We will also refer to it as the
  category of symmetric spectra parametrized by $X$.
\end{definition}
When $K$ is a space, the category $\Spsym{K} = \Spsym{\const_{\cI}K}$
is equivalent to the category $\Spsym{}(\cS_K,S^1)$ of symmetric
spectrum objects in the category $\cS_K$ of spaces over and under $K$,
and the absolute level model structure on $\Spsym{K}$ corresponds to
the level model structure
from~\cite[Theorem~8.2]{Hovey_symmetric-general}.
\begin{lemma}\label{lem:base-change-Q-adj-for-SpsymX-level}
A map of $\cI$-spaces $f\colon X \to Y$ induces a Quillen adjunction 
\begin{equation}\label{eq:base-change-adj-for-SpsymX} f_!\colon \Spsym{X} \rightleftarrows \Spsym{Y} \colon f^* 
\end{equation}
with respect to the absolute and positive model structures. 
\end{lemma}
\begin{proof} We apply  Lemma~\ref{lem:adjunction-on-section-cat} to the functors
  $f(\bld{m})_! \colon \cS_{X(\bld{m})} \to \cS_{Y(\bld{m})}$.
\end{proof}
Analogous
to Lemma~\ref{lem:cS_K-pseudofunctor}, we obtain a pseudofunctor
\begin{equation}\label{eq:SpsymX-pseudofunctor} X \mapsto \Spsym{X},\quad (f\colon X \to Y)\mapsto (f_!\colon \Spsym{X} \to  \Spsym{Y}) \ .
\end{equation}
\begin{lemma}\label{lem:SpsymR-Grothendieck-construction}
The Grothendieck construction of~\eqref{eq:SpsymX-pseudofunctor} is equivalent to  $\Spsym{\cR}$.
\end{lemma}
\begin{proof}
This follows from a levelwise application of Lemma~\ref{lem:retractive-spaces-as-Grothendieck-construction}.
\end{proof}
Under this equivalence, the category $\Spsym{X}$ corresponds to the
fiber of $\pi_b$ over $X$. This identification of $\Spsym{X}$ allows
us to give a different description of the
adjunction~\eqref{eq:base-change-adj-for-SpsymX}: there are natural
isomorphisms
\begin{equation}\label{eq:base-change-cobase-change} f_!(E,X) \iso (E,X)\cup_{\bS^{\cI}_b[X]}\bS^{\cI}_b[Y] \quad\text{ and }\quad f^*(F,Y) \iso (F,Y)\times_{\bS^{\cI}_b[Y]}\bS^{\cI}_b[X]
\end{equation}
where the pushouts and pullbacks are formed in $\Spsym{\cR}$. In other words,
$f_!$ corresponds to the cobase change along $\bS^{\cI}_b[f]$, and $f^*$ corresponds to the base change along $\bS^{\cI}_b[f]$. 

We write $F_{\bld{m}}^{\Spsym{\cR}}\colon \cS_{\cR} \to \Spsym{\cR}$,
$F_{\bld{m}}^{\Spsym{X}}\colon \cS_{X(\bld{m})} \to \Spsym{X}$, and
$F^{\cS^{\cI}}_{\bld{m}}\colon \cS \to \cS^{\cI}$  for the free functors
obtained by
implementing~\eqref{eq:free-evaluation-adjunction-pseudofunctor} in
the categories $\Spsym{\cR}$, $\Spsym{X}$, and $\cS^{\cI}$. 

For later use we record how the free functors to $\Spsym{\cR}$ and $\Spsym{X}$ are related. 
\begin{lemma}\label{lem:free-cR-and-free-relative-X-pushout}
  Let $(V,L)$ be an object in $\cS_L$, let
  $f\colon {F_{\bld{m}}^{\cS^{\cI}}(L) \to X}$ be a map in $\cS^{\cI}$, and let $\widetilde{f}\colon L \to X(\bld{m})$ be the
  adjoint of $f$.  Then there is a natural isomorphism
$F_{\bld{m}}^{\Spsym{X}}(\widetilde{f}_!(V,L)) \iso {f}_!  F_{\bld{m}}^{\Spsym{\cR}}(V,L) \ .$\qed
\end{lemma}

Applying Definition~\ref{def:integral-model-str} to the pseudofunctor~\eqref{eq:SpsymX-pseudofunctor}, the absolute and positive level model structures on $\cS^{\cI}$ and the $\Spsym{X}$ give rise to integral cofibrations, fibrations, and weak equivalences on the Grothendieck construction. 
\begin{proposition}\label{prop:level-Grothendieck-construction}
  These classes of maps form absolute and positive integral level
  model structures on the Grothendieck construction.  Under the
  equivalence with $\Spsym{\cR}$, they correspond to the absolute
  and positive level model structures on $\Spsym{\cR}$. 
\end{proposition}
\begin{proof}
  Inspecting the generating cofibrations of $\cS^{\cI}$, we see that a
  cofibration in the absolute level model structure on $\cS^{\cI}$ is
  level-wise a cofibration in $\cS$. A similar result holds for the
  absolute level model structure on $\Spsym{X}$. It follows that the
  analogue of Lemma~\ref{lem:extension-restriction-R-homotopical}
  holds for the adjunction \eqref{eq:base-change-adj-for-SpsymX}.
  Hence~\cite[Theorem 3.0.12]{Harpaz-P_Grothendieck-construction}
  applies and shows the existence of the integral model structure. It
  matches with the absolute level model structure on $\Spsym{\cR}$
  since the fibrations and weak equivalences are the same. The case of
  the positive model structures is analogous.
\end{proof}
This proposition, the definition of the integral model structure, and
the identification of $\Spsym{X}$ with the fiber of $\pi_b\colon
\Spsym{\cR}\to \cS^{\cI}$ over $X$ imply:
\begin{corollary}\label{cor:compatible-level-model-str}
  A map in the absolute (resp.\ positive) level model structure on
  $\Spsym{X}$ is a cofibration, fibration, or weak equivalence if and
  only if it is so when viewed as a map in the absolute
  (resp.\ positive) model structure on $\Spsym{\cR}$.\qed
\end{corollary}

\section{The convolution smash product}\label{sec:smash}
The reason for using $\cI$ as an indexing category in the definition of a Quillen $\cI$-category 
is that this allows us to define symmetric monoidal products on section categories. 

In the
case of $\cI$-spaces mentioned in Example~\ref{ex:pseudo}(i), the monoidal
product $X \boxtimes Y$ of two functors $X,Y \colon \cI \to \cS$ is
the left Kan extension of the $\cI\times \cI$-diagram $(\bld{m},\bld{n}) \mapsto X(\bld{m})\times Y(\bld{n})$ along the concatenation $-\concat-\colon \cI \times \cI \to \cI$. This is an example of a Day convolution product, and more explicitly we have 
\[(X\boxtimes Y)(\bld{p}) = \colim_{\bld{m}\concat\bld{n} \to
  \bld{p}} X(\bld{m}) \times Y(\bld{n}) \] with the colimit taken over
the over category $-\concat-\downarrow \bld{p}$. The
$\boxtimes$-product provides a symmetric monoidal product on
$\cS^{\cI}$ with unit $\ucI = F^{\cS^{\cI}}_{\bld{0}}(*) \iso \cI(\bld{0},-)\iso \const_{\cI}(*)$.
\begin{definition}\label{def:commutative-I-space-monoid}
A \emph{commutative $\cI$-space monoid} is a commutative monoid in $(\cS^{\cI},\boxtimes, \ucI)$.
\end{definition}
Equivalently, a commutative $\cI$-space monoid is a lax symmetric
monoidal functor
$(\cI,\concat,\bld{0}) \to (\cS^{\cI},\boxtimes, \ucI)$. Every
$E_{\infty}$ space can be represented by a commutative $\cI$-space
monoid in the sense explained in~\cite[Corollary~3.7]{Sagave-S_diagram}.

In the case of symmetric spectra mentioned in
Example~\ref{ex:pseudo}(ii), the monoidal product is the well-known
smash-product of symmetric spectra. In this description of symmetric
spectra employing $\cI$, the smash product $E \sm F$ of $E,F$ is in
level $p$ given by the colimit
\[ (E\sm F)_{p} =
\colim_{\alpha\colon\bld{m}\concat\bld{n}\to\bld{p}}E_m\sm F_{n}\sm
S^{\bld{p}-\alpha} \] 
taken over the over category $-\concat - \downarrow \bld{p}$. The maps in the colimit system arise from the structure maps of $E$ and $F$, the isomorphism~\eqref{eq:reindexing-basic}, and the isomorphism 
\begin{equation}\label{eq:reindexing-product}
S^{(\bld{n_1}\concat\bld{n_2})-(\alpha_1\concat\alpha_2)} \iso S^{\bld{n_1}-\alpha_1} \sm S^{ \bld{n_2}-\alpha_2}
\end{equation} that is induced by the canonical bijection $(\bld{n_1}-\alpha_1)\amalg(\bld{n_2}-\alpha_2)\to (\bld{n_1}\concat\bld{n_2})-(\alpha_1\concat \alpha_2)$ associated with a pair of morphisms $\alpha_1\colon\bld{m_1}\to\bld{n_1}$ and $\alpha_2\colon\bld{m_2}\to\bld{n_2}$ in $\cI$. The structure maps of $E\sm F$ also arise from~\eqref{eq:reindexing-basic}.
\subsection{The monoidal structure on symmetric spectra in retractive spaces}
In analogy with the smash product in $\Spsym{}$ and the $\boxtimes$-product of $\cI$-spaces, there is a symmetric monoidal product \[- \barsm - \colon \SpsymR \times \SpsymR \to \SpsymR \] 
given in level $p$ by the colimit
\[ ((E,X) \barsm (F,Y))_{p} =
\colim_{\alpha\colon\bld{m}\concat\bld{n}\to\bld{p}}(E_m,X_m)\barsm
(F_{n},X_{n}) \barsm S^{\bld{p}-\alpha}\]
in $\cS_{\cR}$ taken over the category
$-\concat - \downarrow \bld{p}$. The maps in the colimit system and
the structure maps of $(E,X) \barsm (F,Y)$ are defined as for
symmetric spectra.  We also note that there are natural isomorphisms
\begin{equation}\label{eq:barsm-of-free}
F_{\bld{m}}^{\Spsym{\cR}}(U,K) \barsm F_{\bld{n}}^{\Spsym{\cR}}(V,L) \iso F_{\bld{m}\concat \bld{n}}^{\Spsym{\cR}}((U,K)\barsm (V,L))  \ .
\end{equation}

\begin{proposition}\label{prop:barsm-pushout-product-level}
  The $\barsm$-product defines a closed symmetric monoidal structure
  on $\SpsymR$ with unit $\bS$ that satisfies the pushout product
  axiom with respect to the absolute and positive level model
  structures.
\end{proposition}
\begin{proof}
  Since the pushout product axiom can be checked on the generating
  (acyclic) cofibrations, it follows from the
  isomorphisms~\eqref{eq:barsm-of-free} and the pushout product axiom
  for $\cS_{\cR}$ established in
  Proposition~\ref{prop:pushout-product-for-cScR}. The
  $\barsm$-product on $\SpsymR$ is closed because
  $(\cS_{\cR},\barsm, S^0)$ is closed by
  Proposition~\ref{prop:R-closed-sym-monoidal} (compare~\cite[\S
  7]{Hovey_symmetric-general}).
\end{proof}

We write $E \barsm F$ for the total space of $(E,X)
\barsm (F,Y)$. Observing that its base $\cI$-space can be
identified with $X\boxtimes Y$, we have 
\[ (E,X) \barsm (F,Y) = (E\barsm F, X\boxtimes Y). \]

\begin{lemma}\label{lem:cobase-change-spym-lax-monoidal}
If $(E,X)$ and $(F,Y)$ are objects in $\SpsymR$ and $f \colon X \to X'$ and ${g\colon Y \to Y'}$ are morphisms in $\cS^{\cI}$, then there is an isomorphism
\[ f_!(E,X)\barsm g_!(F,Y)  \xrightarrow{\iso} (f\boxtimes g)_!((E,X)\barsm (F,Y)). \]
It is natural with respect to the coherence isomorphisms $(f'f)_! \iso f'_! f_!$ for composable maps of $\cI$-spaces $f$ and $f'$. 
\end{lemma}

\begin{proof}
  Commuting the colimit over $-\concat- \downarrow \bld{p}$ with the
  pushout computing the total space identifies the total space in
  level $p$ of the right hand expression with
  $\colim_{\alpha\colon \bld{m}\concat\bld{n}\to \bld{p}}
  \left(f(\bld{m})\times g(\bld{n})\right)_!(E_m\barsm F_m \barsm
  S^{\bld{p}-\alpha})$.  Composing it with the colimit over $\alpha$ in
  $-\concat- \downarrow \bld{p}$ of the natural isomorphisms
\begin{multline*} f(\bld{m})_! (E_m,X(\bld{m})) \barsm g(\bld{n})_! (F_n,Y(\bld{n}))\barsm S^{\bld{p}-\alpha}\\ \xrightarrow{\iso} \left(f(\bld{m})\times g(\bld{n})\right)_! \left( (E_m,X(\bld{m})) \barsm (F_n,F(\bld{n})) \barsm S^{\bld{p}-\alpha}\right) 
\end{multline*}
provided by Lemma~\ref{lem:cobase-barsm-map} gives the desired isomorphism.
\end{proof}
As a consequence, we note that given maps of $\cI$-spaces $f\colon X' \to X$ and $g\colon Y' \to Y$ as well as objects $(E,X)$ in $\Spsym{X}$ and $(F,Y)$ in $\Spsym{Y}$, there is a chain of maps 
\begin{multline}\label{eq:upper-star-lax-monoidal} 
f^*(E,X) \barsm g^*(F,Y) \to (f\boxtimes g)^*(f\boxtimes g)_!( f^*(E,X) \barsm g^*(F,Y) )\\ \xrightarrow{\iso} (f\boxtimes g)^*(f_! f^*(E,X) \barsm g_! g^*(F,Y) ) \to (f\boxtimes g)^*((E,X) \barsm (F,Y)) 
\end{multline}
induced by the adjunction unit of $({(f\boxtimes g)_!},(f\boxtimes g)^*)$, the isomorphism of Lemma~\ref{lem:cobase-change-spym-lax-monoidal}, and the adjunction counits of $(f_!,f^*)$ and $(g_!,g^*)$. We will show in Proposition~\ref{prop:monoidal_pull} that this morphism descends to an isomorphism between the derived functors in the homotopy category. 

The category $\cS^{\cI}_{\cR}$ of $\cI$-diagrams in $\cS_{\cR}$ has a Day convolution product induced by the $\barsm$-product on $\cS_{\cR}$ and the concatenation in $\cI$. Analogously, the cartesian product on $\mathrm{Ar}(\cS)$ induces a Day convolution product on  $(\mathrm{Ar}(\cS))^{\cI}$ that coincides with the objectwise $\boxtimes$-product in $\cS^{\cI}$ under the identification  $(\mathrm{Ar}(\cS))^{\cI} \iso \mathrm{Ar}(\cS^{\cI})$. In the next diagram, the first adjunction is induced by the corresponding space level adjunction from Construction~\ref{constr:spaces-retractive-spaces} and the second is from Construction~\ref{const:various-adjunctions}.
 \begin{equation}\label{eq:SIar-OmegaIar-adjunction}
\xymatrix@-1pc{\mathrm{Ar}(\cS^{\cI})  \ar@<.4ex>[rrr]^-{\iota_{\mathrm{ar}}} &&& \cS^{\cI}_{\cR} \ar@<.4ex>[rrr]^-{\bS^{\cI}_{\cR}} \ar@<.4ex>[lll]^-{\pi_{\mathrm{ar}}} &&& \SpsymR \ar@<.4ex>[lll]^-{\Omega^{\cI}_{\cR}} 
}
\end{equation}
\begin{lemma}\label{lem:SIar-OmegaIar-adjunction-monoidal}
The left adjoint functors $\iota_{\mathrm{ar}}$ and $\bS^{\cI}_{\cR}$ in~\eqref{eq:SIar-OmegaIar-adjunction} are strong symmetric monoidal. Hence so is their composite $\bS^{\cI}_{\mathrm{ar}} = \bS^{\cI}_{\cR} \circ \iota_{\mathrm{ar}} $, and the right adjoints $\pi_{\mathrm{ar}} $, $\Omega^{\cI}_{\cR} $, and $\Omega^{\cI}_{\mathrm{ar}} =  \pi_{\mathrm{ar}}\circ  \Omega^{\cI}_{\cR}$ are lax symmetric monoidal. \qed
\end{lemma}
\subsection{The monoidal structure on $M$-relative symmetric spectra}
Let $X$ and $Y$ be $\cI$-spaces. Via the identification of $\Spsym{X}, \Spsym{Y}$, and $\Spsym{X\boxtimes Y}$ with subcategories of $\Spsym{\cR}$,  the $\barsm$-product on $\Spsym{\cR}$ induces an \emph{external} product 
\begin{equation}\label{eq:spsymrel-external-product}
-\barsm-\colon \Spsym{X}\times \Spsym{Y} \to \Spsym{X\boxtimes Y}.
\end{equation}
If $M$ is a commutative $\cI$-space monoid, then this external product
and the multiplication $\mu \colon M\boxtimes M \to M$ of $M$ induce a symmetric monoidal convolution product
\begin{equation}\label{eq:spsymrel-M-product} \Spsym{M}\times \Spsym{M} \to \Spsym{M\boxtimes M}  \xrightarrow{\mu_!} \Spsym{M}\ . 
\end{equation}
Let $\iota\colon *\to M$ be the unit and write $\bS_M=\iota_!(\bS)$, where $\bS = (\bS,*)$ is the monoidal unit in $\Spsym{}\!\!=\Spsym{*}$. It follows from Lemma~\ref{lem:cobase-change-spym-lax-monoidal} that $\bS_M$ is the monoidal unit for~$\Spsym{M}$.
\begin{proposition} This symmetric monoidal
  product satisfies the pushout product
  axiom with respect to the absolute level model structure on $\Spsym{M}$.
\end{proposition}
\begin{proof}
  This  follows from the pushout product axiom for
  $\SpsymR$ and the fact that $\mu_!$ preserves
  cofibrations and acyclic cofibrations by
  Lemma~\ref{lem:base-change-Q-adj-for-SpsymX-level}.
\end{proof}
In a similar fashion, the category $\cS^{\cI}_M$ of $\cI$-spaces over
and under $M$ inherits a symmetric monoidal product from $\cS^{\cI}_{\cR}$. For later use we note the following
compatibility.
\begin{lemma}\label{lem:bSIt-symm-mon}
  The functor $\bS^{\cI}_{\cR}\colon \cS^{\cI}_M \to \Spsym{M}$ is strong
  symmetric monoidal. \qed
\end{lemma}

If $M \to N$ is a morphism of commutative $\cI$-space monoids, then
Lemma~\ref{lem:cobase-change-spym-lax-monoidal} implies that the
induced functor $(M\to N)_!\colon \Spsym{M} \to \Spsym{N}$ is strong
symmetric monoidal. In particular,
\[ \Theta =  (M \to \ucI)_! \colon \Spsym{M} \to \Spsym{} \]
is strong symmetric monoidal, so that commutative monoids in $\Spsym{M}$
give rise to commutative symmetric ring spectra if their base
$\cI$-space is collapsed.

\subsection{The simplicial structure on $X$-relative symmetric spectra}
Let $X$ be an $\cI$-space. If $Q$ is an unbased simplicial set, we define
a functor \[Q \tensor - \colon \Spsym{X} \to  \Spsym{X}, \quad (E,X)\mapsto
F_{\bld{0}}^{\Spsym{\cR}}(Q_+) \barsm (E,X)\ .\]
Here we again identify $\Spsym{X}$ with a subcategory of $\Spsym{\cR}$, view $Q_+$ as the retractive space $(Q_+,*)$, and apply geometric realization to $Q$ when working with $\cS =\tp$. 
\begin{proposition}\label{prop:SpsymX-simplicial}
  This action turns $\Spsym{X}$ into a simplicial model category.
\end{proposition}
\begin{proof}
  An application of~\cite[Lemma II.2.4]{Goerss-J_simplicial} shows
  that $\Spsym{X}$ becomes a simplicial category since $\barsm$ is a
  closed symmetric monoidal structure on $\Spsym{\cR}$. The
  compatibility with the model structure follows from the pushout
  product axiom for $(\Spsym{\cR},\barsm)$ established in
  Proposition~\ref{prop:barsm-pushout-product-level} and the
  compatibility of the model structures on $\Spsym{X}$ and
  $\Spsym{\cR}$ established in
  Corollary~\ref{cor:compatible-level-model-str}. 
\end{proof}
It follows from the definitions that $(Q\tensor (E,X))_{n}$ can be
identified with the tensor $Q \tensor (E_n,X_n)$ in $\cS_{X_n}$ (see Proposition~\ref{prop:S_K-simplicial}).
Since the category $\Spsym{X}$ has a zero-object, the tensor structure
over $\sset$ induces a tensor over $\sset_*$, and one can check that
for a based simplicial set $B$ the based tensor is just the levelwise
$\barsm$-product with $(B,*)$. Particularly, the suspension is the
levelwise $\barsm$-product with~$S^1$.

\subsection{Tensor structures over \texorpdfstring{$\cI$}{I}-spaces}\label{subsec:tensor} 
There is a functor
\begin{equation}\label{eq:SIR-barsm-SpsymR} \cS^{\cI} \times  \SpsymR \to \SpsymR,\quad (Y, (E,X)) \mapsto \bS^{\cI}_t[Y] \barsm (E,X) 
\end{equation}
that exhibits $\SpsymR$ as a category tensored over
$(\cS^{\cI},\boxtimes, *)$, meaning that  $\SpsymR$ is a $\cS^{\cI}$-module in the sense of~\cite[Definition 4.1.6]{Hovey_model}.  For the applications in~\cite{HS-twisted} 
and in Section~\ref{subsec:Thom-over-S} below, it is
important that $\SpsymR$ is also tensored over
$(\cS^{\cI},\times,\ucI)$, i.e., $\cI$-spaces with the cartesian
product. To define this tensor, we first introduce a monoidal product
on the category $\cS^{\cI}_{\cR}$ of $\cI$-diagrams in $\cS_{\cR}$ and
an accompanying tensor structure on $\SpsymR$. The monoidal structure
on $\cS^{\cI}_{\cR}$ is the degreewise $\barsm$-product and will be
denoted by $\widetilde{\sm}$. Its unit is
$\iota_t(\ucI) = \const_{\cI}(S^0,*)$, and the functor
$\iota_t \colon (\cS^{\cI},\times,\ucI) \to (\cS^{\cI}_{\cR},
\widetilde{\sm},\iota_t(\ucI))$
sending $X$ to $(X\amalg X,X)$ is strong symmetric monoidal
by~\eqref{eq:barsmiso-three}.

The category $\SpsymR$ is tensored over
$(\cS^{\cI}_{\cR},\widetilde{\sm},\iota_t(\ucI))$ with tensor
structure
\begin{equation}\label{eq:SIR-tensor-SpsymR}
-\widetilde{\sm}- \colon \cS^{\cI}_{\cR} \times \SpsymR \to \SpsymR,\quad \big((Z,Y),(E,X)\big) \mapsto \big(\bld{n}\mapsto (Z,Y)(\bld{n}) \barsm (E,X)(\bld{n})\big).
\end{equation}
Here the structure maps act diagonally, i.e., $\alpha \colon \bld{m}\to\bld{n}$ acts via 
\[(Z,Y)(\bld{m}) \barsm (E,X)(\bld{m}) \barsm S^{\bld{n}-\alpha}  \xrightarrow{\alpha_* \barsm \alpha_*}   (Z,Y)(\bld{n}) \barsm (E,X)(\bld{n}).
\]
Restricting~\eqref{eq:SIR-tensor-SpsymR} along $\iota_t$ in the first variable, we get the desired tensor structure of $\SpsymR$ over $  (\cS^{\cI},\times,\ucI)$ given by 
\begin{equation}\label{eq:SI-tensor-SpsymR} -\times- \colon \cS^{\cI} \times \SpsymR \to \SpsymR 
\end{equation}
with $(Y \times (E,X))(\bld{n}) \iso (Y(\bld{n}) \times E_n, Y(\bld{n}) \times X(\bld{n}))$. The latter isomorphism results from~\eqref{eq:barsmiso-one} and justifies the symbol $\times$. Although $\times$ admits this easier description, we have chosen to define it via $\widetilde{\sm}$ since this makes the structure maps on $Y \times (E,X)$ more transparent. We shall primarily use the action $\times$ when $(E,X)$ is just a symmetric spectrum $E$ viewed as the object $E = (E,\ucI)$ of $\SpsymR$.

Now we relate this tensor structure to that of~\eqref{eq:SIR-barsm-SpsymR}:
\begin{construction}\label{const:rhoZYEX}
There is a natural transformation \begin{equation}\label{eq:rhoZYEX} \rho_{Y,(E,X)} \colon \bS^{\cI}_{t}[Y]\barsm (E,X) \to Y \times (E,X)
\end{equation} of functors $\cS^{\cI}\times \SpsymR \to
\SpsymR$. On the term in the colimit system defining the $\barsm$-product that is indexed by $\alpha\colon \bld{k}\concat \bld{l} \to \bld{n}$, it is given by the composite 
\begin{multline*} (\iota_tY)(\bld{k}) \barsm S^{\bld{k}} \barsm (E,X)(\bld{l}) \barsm S^{\bld{n}-\alpha} \xrightarrow{\iso}  (\iota_tY)(\bld{k}) \barsm (E,X)(\bld{l}) \barsm S^{\bld{n}-\alpha|_{\bld{l}}}\\ \to   (\iota_tY)(\bld{n}) \barsm  (E,X)(\bld{n}) = (Y \times (E,X))(\bld{n}).
\end{multline*}
Here the first map interchanges the two inner factors and uses the isomorphism of spheres induced by the bijection $\bld{k} \concat (\bld{n}-\alpha) \to (\bld{n}-\alpha|_{\bld{l}})$ defined by $\alpha$, and the second map is given by the
action of $\alpha|_{\bld{k}}$ and the structure map of $(E,X)$. 
\end{construction}
It is shown in \cite[Proposition~4.1]{HS-twisted} that under suitable conditions on $Y$ and $(E,X)$, the map~\eqref{eq:rhoZYEX} is a local equivalence in the sense of Section~\ref{subsec:local-model-spsymR} below. The latter result plays a  central role in our applications to models of twisted $K$-theory spectra in \cite{HS-twisted}. We also note that on base $\cI$-spaces, $\rho_{Y,(E,X)}$ is just the natural map $Y \boxtimes X \to Y \times X$ studied in~\cite[Proposition 2.27]{Sagave-S_group-compl}. 

It will also be useful to know that the different products are related
by the following commutative square explained below:
\begin{equation}\label{eq:distributivity} \xymatrix@-1pc{
\bS^{\cI}_t[Y] \!\barsm\! (E,X)  \!\barsm\! \bS^{\cI}_t[Y'] \!\barsm\! (E',X') \ar[rrr]^{\iso} \ar[d]^{\rho_{Y,(E,X)} \barsm \rho_{Y',(E',X')}}&&&  
\bS^{\cI}_t[Y \boxtimes Y'] \!\barsm\! (E,X) \!\barsm\! (E',X') \ar[d]_{\rho_{Y\boxtimes Y',(E,X)\barsm (E',X')}} \\
(Y \times (E,X) ) \!\barsm\! (Y' \times (E',X') ) \ar[rrr]^{\delta} &&&  (Y \boxtimes Y') \!\times \! (E,X) \!\barsm\! (E',X')}
\end{equation}
The vertical maps are instances of~\eqref{eq:rhoZYEX}, the upper horizontal map is the composite of the twists of the middle terms and the isomorphism ${\bS^{\cI}_t[Y] \barsm \bS^{\cI}_t[Y']} \to \bS^{\cI}_t[Y \boxtimes Y']$, and the lower horizontal map $\delta$ is the \emph{distributivity} map induced by the maps
\[ (Y(\bld{k}) \times (E,X)(\bld{k})) \barsm (Y'(\bld{l}) \times
(E',X')(\bld{l})) \barsm S^{\bld{n}-\alpha}\!\! \to (Y \boxtimes
Y')(\bld{n}) \times ( (E,X) \barsm (E',X'))(\bld{n}) \]
for $\alpha \colon \bld{k}\concat \bld{l} \to \bld{n}$ in $\cI$ which
are given by the twist of the middle factors and the canonical maps to
the $\boxtimes$ and $\barsm$-products.

\section{Local model structures}\label{sec:local-model}
Let $\cC\colon \cI \to \mathrm{Cat}$ be a Quillen $\cI$-category. If
$\alpha \colon \bld{m} \to \bld{n}$ is a map in $\cI$ and $Z$ is an
object in $ \cC_{\bld{m}}$, then the inclusion
$\alpha_!Z \to \textstyle\coprod_{\beta\colon \bld{m} \to \bld{n}}
\beta_!Z = (F_{\bld{m}}(Z))_{\bld{n}}$  
of the summand indexed by $\alpha$ gives rise to an adjoint map
\begin{equation}\label{eq:maps-we-are-localizing-at}
\alpha_Z \colon F_{\bld{n}}(\alpha_!Z) \to F_{\bld{m}}(Z)
\end{equation}
in the section category $\cC^{\cI}$. We define $S_{\cC}$ to be the set of
maps $\alpha_Z$ where $\alpha \colon\bld{m}\to \bld{n}$ is any
morphism in $\cI$ and $Z$ is the cofibrant replacement of a domain or
codomain of a generating cofibration of $\cC_{\bld{m}}$. Writing $\cI^+$ for
the full subcategory of $\cI$ on the objects $\bld{m}$ with $|\bld{m}|\geq 1$, 
we let $S_{\cC}^+$ be the subset of $S_{\cC}$ where $\alpha$ runs through the
morphisms in $\cI^+$.

Our aim is to form the left Bousfield localizations~\cite[\S 3]{Hirschhorn_model} of the level model
structures on $\cC^{\cI}$ at $S_{\cC}$ and $S_{\cC}^+$. We need an
additional hypothesis to ensure their existence and say that a
Quillen $\cI$-category $\cC$ is \emph{cellular} and \emph{left proper} if each
$\cC_{\bld{m}}$ is.
\begin{proposition}\label{prop:existence-local-model-structure}
  Let $\cC\colon \cI \to \mathrm{Cat}$ be a cellular and left proper
  Quillen $\cI$-category. Then the left Bousfield localizations 
  of the
  absolute level model structure on $\cC^{\cI}$ at the set $S_{\cC}$
  and the positive level model structure at $S_{\cC}^+$ exist and are
  cellular and left proper again. 
\end{proposition}
\begin{proof} This follows from~\cite[Theorem 4.1]{Hirschhorn_model}
  once we verified that the absolute level model structure on
  $\cC^{\cI}$ is cellular and left proper. Since cofibrations in
  $\cC^{\cI}$ are in particular cofibrations in each level and
  colimits in $\cC^{\cI}$ are formed levelwise, this is immediate.
\end{proof}

\begin{definition} The model structures from the previous proposition are called the \emph{absolute} and \emph{positive local model structures} on the section category $\cC^{\cI}$.
\end{definition}
\begin{lemma}\label{lem:fibrant-in-local}
  An object $X$ in $\cC^{\cI}$ is fibrant in the absolute
  (resp.\ positive) local model structure if and only if for each
  $\alpha$ in $\cI$ (resp.\ $\cI^+$), the adjoint structure map
  $X_{\bld{m}} \to \alpha^*(X_{\bld{n}})$ is a weak equivalence
  between fibrant objects in $\cC_{\bld{m}}$.
\end{lemma}
\begin{proof}
  We write $\mathrm{Map}_{\cC^{\cI}}$ and
  $\mathrm{Map}_{\cC_{\bld{m}}}$ for the homotopy function complexes
  in $\cC^{\cI}$ and in $\cC_{\bld{m}}$ (see~\cite[\S
  17.4]{Hirschhorn_model}). By definition, an object $X$ is fibrant in
  the absolute local model structure on $\cC^{\cI}$ if and only if it
  is absolute level fibrant and
  $\mathrm{Map}_{\cC^{\cI}}(\alpha_Z,X)$ is a weak equivalence of
  simplicial sets for all $\alpha_Z$ in $S_{\cC}$. Since homotopy
  function complexes are compatible with Quillen
  adjunctions~\cite[Proposition~17.4.15]{Hirschhorn_model},
  the latter
  condition is equivalent to asking that
  $\mathrm{Map}_{\cC_{\bld{m}}}(Z,X_{\bld{m}}) \to
  \mathrm{Map}_{\cC_{\bld{m}}}(Z,\alpha^*(X_{\bld{n}}))$
  is a weak equivalence of simplicial sets when $Z$ is the cofibrant
  replacement of a domain or codomain of a
  generating cofibration for $\cC_{\bld m}$ and $\alpha\colon \bld{m}\to \bld{n}$ is a map in
  $\cI$. By \cite[Proposition A.5]{Dugger_replacing}, for fixed
  $\alpha$ and varying $Z$ this condition is equivalent to
  $X_{\bld{m}} \to \alpha^*(X_{\bld{n}})$ being a weak equivalence in
  $\cC_{\bld{m}}$. The positive case is analogous.
\end{proof}

 The argument in the last proof also implies the following statement.
 \begin{corollary} \label{cor:alpha-Z-weak-equivalence} If $Z$ is any
   cofibrant object in $\cC_{\bld{m}}$ and $\alpha$ is a map in $\cI$
   (resp.\ $\cI^+$), then $\alpha_Z$ is a weak equivalence in the
   absolute (resp.\ positive) local model structure on $\cC^{\cI}$. \qed
\end{corollary}

The identifications of the model structures in the next example
uses the fact that a model structure is determined by its cofibrations and fibrant objects~\cite[Proposition E.1.10]{Joyal-quasi-categories}. 
\begin{example}\phantomsection \label{ex:pseudo-local}
\begin{enumerate}[(i)]
\item In the situation of Example~\ref{ex:pseudo}(i), the absolute and
positive local model structures are the absolute and positive $\cI$-model structures on $\cS^{\cI}$; see~\cite[Proposition 3.2]{Sagave-S_diagram}. The weak equivalences in these model structure are called \emph{$\cI$-equivalences} and are given by the maps $X \to Y$ that induce weak homotopy equivalences $X_{h\cI} \to Y_{h\cI}$ on the (Bousfield--Kan) homotopy colimits of the $\cI$-diagrams $X$ and $Y$.  
\item In the situation of Example~\ref{ex:pseudo}(ii), the absolute
  and positive local model structures are the respective stable model structures
  on $\Spsym{}$;
  see~\cite[Theorem~3.4.4]{HSS} and~\cite[\S 14]{MMSS}.
\item In the situation of Example~\ref{ex:pseudo}(iii), we obtain absolute and positive $\cI$-model structures on the category $\cS^{\cI}_{\cR}$ of $\cI$-diagrams in $\cS_{\cR}$. They can also be constructed by identifying $\cS^{\cI}_{\cR}$ with the category of retractive objects in $\cS^{\cI}$ and applying the argument in the proof of Proposition~\ref{prop:model-str-on-R} to the $\cI$-model structures on $\cS^{\cI}$. Moreover, the absolute local model structure on $\cS^{\cI}_{\cR}$ coincides with the hocolim model structure obtained from \cite[Theorem~5.2]{Dugger_replacing}.
\item In the situation of Example~\ref{ex:pseudo}(iv), the absolute local model structure on $\cS^{\cI}_X$ corresponds to the model structure on the category  $(\cS^{\cI})_X$ of $\cI$-spaces over and under $X$ induced by the $\cI$-model structure on $\cS^{\cI}$. To see this, we note that the explicit description of the $\cI$-fibrations in $\cS^{\cI}$ in terms of homotopy cartesian squares~\cite[\S 3.1]{Sagave-S_diagram} implies that the fibrant objects in $(\cS^{\cI})_X$ match the local objects in $\cS^{\cI}_X$. 
\end{enumerate}
\end{example}
 
\subsection{Local model structures on symmetric spectra in retractive spaces}\label{subsec:local-model-spsymR}
Next we consider the category of symmetric spectra in retractive
spaces $\Spsym{\cR}$ introduced in Definition~\ref{def:spsymR}. An
object is fibrant in the resulting absolute (resp.\ positive) local
model structure if and only if it is absolute (resp.\ positive) level
fibrant and the adjoint structure maps
$(E,X)(\bld{m}) \to \Hom_{\cR}((S^{\bld{n}-\alpha},*),(E,X)(\bld{n}))$ 
are
weak equivalences in $\cS_{\cR}$ 
for all $\alpha$ in $\cI$
(resp.\ $\cI^+$). In view of the definition of the cotensor
in~\eqref{eq:cotensor}, the latter condition means that the horizontal
maps in the following diagram are required to be weak equivalences:
\begin{equation}\label{eq:fibrant-in-SpsymR-testsquare}
\xymatrix@-1pc{ 
E_m \ar[rr] \ar[d] && (E_n)^{S^{\bld{n}-\alpha}} \times_{E_n \times X(\bld{n})^{ S^{\bld{n}-\alpha}}} X(\bld{n}) \ar[d] \\
X(\bld{m}) \ar[rr] && X(\bld{n}) }
\end{equation}
Our absolute and positive local model structures on $\SpsymR$ coincide
with the corresponding model structures on symmetric spectra in
$\cS_{\cR}$ considered elsewhere in the literature (compare
e.g.~\cites{Hovey_symmetric-general, Gorchinskiy-G_positive}).

\begin{proposition}\label{prop:abs-pos-local-equiv-SpsymR}
The weak equivalences in the absolute and positive local model structures
on $\SpsymR$ coincide. 
\end{proposition}
\begin{proof}
This follows from~\cite[Theorem 10]{Gorchinskiy-G_positive}.
\end{proof}

\begin{remark}
  We resist from calling the model structures from
  Proposition~\ref{prop:existence-local-model-structure} \emph{stable}
  since they are not necessarily stable in the sense that suspension
  becomes invertible on the homotopy category. In fact, $\cS^{\cI}$
  and $\Spsym{\cR}$ have no zero objects and cannot be stable in the
  latter sense.  
\end{remark}

\begin{proposition}\label{prop:local-pushout-product}
  The absolute and positive local model structures on $\Spsym{\cR}$ satisfy the
  pushout product axiom with respect to $\barsm$. 
\end{proposition}
\begin{proof}
  The absolute case follows from~\cite[Theorem
  8.11]{Hovey_symmetric-general}. Using
  Proposition~\ref{prop:abs-pos-local-equiv-SpsymR}, the positive case
  follows from the pushout product axiom for the absolute local and
  the positive level model structure.
\end{proof}

\begin{lemma}\label{lem:Q-adjunctions-on-local-model-str}
  The left adjoint functors $\iota_t$, $\iota_b, \bS^{\cI}_{\cR}$, and
  $\pi_b$ introduced in~\eqref{eq:various-left-adjoints} are left
  Quillen functors with respect to the absolute and positive local model
  structures.
\end{lemma}
\begin{proof}
  For the functor $\bS^{\cI}_{\cR}$, we observe that
  $\bS^{\cI}_{\cR}[F_{\bld{n}}^{\cS^{\cI}_{\cR}}(\alpha_! Z) ]\iso
  F_{\bld{n}}^{\Spsym{\cR}}(\alpha_!(Z \barsm S^m))$ where the first
  $\alpha_!$ is part of the pseudofunctor defining $\cS^{\cI}_{\cR}$
  and the second is part of the pseudofunctor defining
  $\Spsym{\cR}$. Since $Z \barsm S^m$ is cofibrant in $\cS_{\cR}$ if
  $Z$ is, Corollary~\ref{cor:alpha-Z-weak-equivalence}
  and~\cite[Proposition 3.3.18(1)]{Hirschhorn_model} show that
  $\bS^{\cI}_{\cR}$ is left Quillen. The other cases are analogous
  (but easier).
\end{proof}

\begin{corollary}\label{cor:S-I-b-left-right-Quillen-locally}
  Both $(\bS^{\cI}_b,\pi_b)$ and $(\pi_b,\bS^{\cI}_b)$ are Quillen
  adjunctions with respect to the absolute and positive local model
  structures.\qed
\end{corollary}
The absolute and positive $\cI$-model structures on $\cS^{\cI}$ give rise to \emph{injective} model structures on  $\mathrm{Ar}(\cS^{\cI})$ where a map is a cofibration or weak equivalence if and only if its two components have this property in $\cS^{\cI}$.  
\begin{lemma}\label{lem:SIar-OmegaIar-Quillen-adjunction}
The adjunction $\iota_{\mathrm{ar}} \colon \mathrm{Ar}(\cS^{\cI}) \rightleftarrows 
 \cS^{\cI}_{\cR} \colon \pi_{\mathrm{ar}}$ from~\eqref{eq:SIar-OmegaIar-adjunction} is a Quillen adjunction with respect to the absolute or positive model structures on $ \mathrm{Ar}(\cS^{\cI})$ and the respective local model structures on~$\cS^{\cI}_{\cR}$. \qed
\end{lemma}

\subsection{The local model structures on $X$-relative symmetric
  spectra} Let $X$ be an $\cI$-space. Then
Proposition~\ref{prop:existence-local-model-structure} gives rise to
absolute and positive local model structures on the category
$\Spsym{X}$ of $X$-relative symmetric spectra. When $K$ is a space,
then these local model structures on
$\Spsym{K} = \Spsym{\const_{\cI}K}$ correspond to the absolute and
positive stable model structure on $\Spsym{}(\cS_K,S^1)$, and the
fibrant objects are the absolute (resp.\ positive) $\Omega$-spectra in
the latter category. For a general base $\cI$-space $X$, an object
$(E,X)$ is fibrant in the absolute (resp.\ positive) local model
structure on $\Spsym{X}$ if and only if it is absolute
(resp.\ positive) level fibrant and the
square~\eqref{eq:fibrant-in-SpsymR-testsquare} is homotopy cartesian
for all $\alpha$ in $\cI$ (resp.\ $\cI^+$). Although their base
$\cI$-space may not be constant, we think of the fibrant objects as
fiberwise (positive) $\Omega$-spectra.

\begin{remark}
  In lack of a symmetric monoidal structure on $\Spsym{X}$, we cannot
  directly apply~\cite[Theorem 10]{Gorchinskiy-G_positive} to show
  that the weak equivalences in the absolute and positive local model
  structures coincide. We will derive this from the corresponding result
  for $\SpsymR$ in Corollary~\ref{cor:absolute-positive-SpsymX} below. 
\end{remark}

\begin{lemma}\label{lem:base-change-local-Q-adjunction}
  If $f\colon X \to Y$ is a map of $\cI$-spaces, then
  $f_!\colon \Spsym{X} \rightleftarrows \Spsym{Y}\colon f^*$ is a
  Quillen adjunction with respect to the absolute and positive local
  model structures. If $f\colon X \to Y$ is an absolute
  (resp.\ positive) level equivalence, then $(f_!,f^*)$ is a Quillen
  equivalence with respect to the absolute (resp.\ positive) local
  model structures.
\end{lemma}
\begin{proof}
  An adjunction argument shows that the cobase changes
  $f_!\colon \Spsym{X} \to \Spsym{Y}$ and
  $f_!\colon \cS_{X(\bld{m})} \to \cS_{Y(\bld{m})}$ commute with the
  free functors. Since the standard generating cofibrations for
  $\cS_{X(\bld{m})}$ have cofibrant
  domains~\cite{Hirschhorn_over_under}, it follows from
  Corollary~\ref{cor:alpha-Z-weak-equivalence} and~\cite[Proposition
  3.3.18(1)]{Hirschhorn_model} that $f_!$ is left Quillen with respect
  to the local model structures. For the Quillen equivalence
  statement, it is by~\cite[Proposition 1.1.13]{Hovey_model}
  sufficient that the derived unit and counit of the adjunction
  $(f_!,f^*)$ are natural weak equivalences. For the derived counit,
  the claim follows because $(f_!,f^*)$ is a Quillen equivalence in
  all (resp.\ all positive) levels. For the derived unit, it is
  sufficient to show that $(E,X) \to f^*(f_!(E,X)^{\mathrm{loc-fib}})$
  is a weak equivalence when $(E,X)$ is both cofibrant and fibrant in
  the local model structure. The fibrancy condition implies that a
  level fibrant replacement of $f_!(E,X)$ is already a fibrant
  replacement in the local model structure, and so the map in question
  is an absolute (resp.\ a positive) level equivalence because
  $(f_!,f^*)$ is a Quillen equivalence in all (resp.\ all positive)
  levels.
\end{proof}

The less obvious result that $(f_!,f^*)$ is already a Quillen
equivalence if $f$ is an $\cI$-equivalence will be shown in
Corollary~\ref{cor:base-change-local-Q-adjunction} below.

\begin{lemma}
  The absolute and positive local model structures on $\Spsym{X}$ are
  simplicial.
\end{lemma} 
\begin{proof}
  This follows from Proposition~\ref{prop:SpsymX-simplicial}
  and~\cite[Theorem 4.1.1(4)]{Hirschhorn_model}.
\end{proof}
\begin{proposition}\label{prop:SpsymX-stable}
The absolute and positive local model structures on $\Spsym{X}$ are stable. 
\end{proposition}
\begin{proof}
  Since the positive case is analogous, we only discuss the absolute
  case.
 
  The suspension on $\Spsym{X}$ is the based tensor with $S^1$, which
  is isomorphic to the functor
  $F^{\Spsym{\cR}}_{\bld{0}}(S^1) \barsm - \colon \Spsym{X} \to
  \Spsym{X}$
  arising from restricting the $\barsm$-product on $\Spsym{\cR}$. The
  latter functor is a left adjoint since $\Spsym{\cR}$ is closed
  monoidal, and left Quillen by the previous lemma. We need to show
  that it induces an equivalence on homotopy categories. The inclusion
  $\iota\colon \bld{0} \to \bld{1}$ induces a map
  $i \colon F^{\cS^{\cI}}_{\bld{1}}(*) \boxtimes X \to
  F^{\cS^{\cI}}_{\bld{0}}(*)\boxtimes X \xrightarrow{\iso} X $. This $i$ is an
  $\cI$-equivalence by~\cite[Proposition 8.2]{Sagave-S_diagram}. We
  consider the composite
  \[D = i_!  (F^{\Spsym{\cR}}_{\bld{1}}(S^0) \barsm - ) \colon
  \Spsym{X} \to \Spsym{X} \ .\]
  Lemma~\ref{lem:cobase-change-spym-lax-monoidal} implies that both
  composites of $F^{\Spsym{\cR}}_{\bld{0}}(S^1) \barsm -$ and $D$ are
  isomorphic to the functor
  $I = i_!  (F^{\Spsym{\cR}}_{\bld{1}}(S^1) \barsm - ) \colon
  \Spsym{X} \to \Spsym{X} $.
  The functors $D$ and $I$ are left adjoint since the symmetric
  monoidal structure of $\Spsym{\cR}$ is closed. They are left Quillen
  with respect to the absolute level model structures by
  Lemma~\ref{lem:base-change-Q-adj-for-SpsymX-level},
  Corollary~\ref{cor:compatible-level-model-str}, and
  Proposition~\ref{prop:barsm-pushout-product-level}. To see that they
  are left Quillen with respect to the absolute local model structures, we
  notice that Lemmas~\ref{lem:free-cR-and-free-relative-X-pushout} and
  \ref{lem:cobase-change-spym-lax-monoidal} as well as the
  isomorphism~\eqref{eq:barsm-of-free} give rise to natural
  isomorphisms
  \begin{equation}\label{eq:ID-identification} \begin{split} & i_! (F_{\bld{1}}^{\Spsym{\cR}}(S^k) \barsm F^{\Spsym{X}}_{\bld{m}}(Z)) \, \iso\,  \left(F^{S^{\cI}}_{\bld{1}\concat \bld{m}}(X(\bld{m})) \to X\right)_!  F^{\Spsym{\cR}}_{\bld{1}\concat \bld{m}}(S^k \barsm Z) \\
    \iso \;& F^{\Spsym{X}}_{\bld{1}\concat \bld{m}}\left(
      ((\iota\concat \bld{m})_* \colon X(\bld{m}) \to X(\bld{1}\concat
      \bld{m}))_!(S^k\barsm Z)\right)\ .
\end{split}
\end{equation}
Hence the functor
$i_! (F_{\bld{1}}^{\Spsym{\cR}}(S^k) \barsm -)$
sends the maps $\alpha_Z$ used to form the local model structures
to
$(1\concat \alpha)_Y$ with $Y =  ((\iota\concat \bld{m})_* \colon X(\bld{m})
  \to X(\bld{1}\concat \bld{m}))_!(S^k\barsm Z)$, 
and $(1\concat \alpha)_Y$ is a local equivalence by
Corollary~\ref{cor:alpha-Z-weak-equivalence}. Combining this
identification for $k=0$ and $k=1$ with~\cite[Proposition
3.3.18(1)]{Hirschhorn_model} implies that $D$ and $I$ are left
Quillen with respect to the absolute local model structures.  

Since $F_{\bld{1}}^{\cS^{\cI}}\!\!(*) \to F_{\bld{0}}^{\cS^{\cI}}\!\!(*)$ is the map of base spaces underlying $\iota_{S^1}
\colon F^{\Spsym{\cR}}_{\bld{1}}(S^1) \to
F^{\Spsym{\cR}}_{\bld{0}}(S^0) = \bS$, 
the latter map and the identification~\eqref{eq:base-change-cobase-change} induce a natural transformation $I \to \id$ of endofunctors of $\Spsym{X}$. 
Since both
functors are left Quillen, a cell induction argument reduces the claim
to showing $I \to \id$ is a local equivalence when evaluated on the
domains and codomains of generating cofibrations. To see this, we note
that the isomorphism~\eqref{eq:ID-identification} implies the
evaluation of $I \to \id$ on $F_{\bld{m}}^{\Spsym{X}}(Z)$ is
isomorphic to the map $(\iota\concat \bld{m})_{Z}$, 
which is a local
equivalence by construction.
\end{proof}
We now consider the diagram 
\[
\xymatrix@-1pc{
\cS^{\cI}_X  \ar@<.4ex>[rrr]\ar@<-.4ex>[d] &&& \Spsym{X}  \ar@<.4ex>[lll]  \ar@<-.4ex>[d] \\ 
\Spsym{}(\cS^{\cI}_X,S^1)  \ar@<.4ex>[rrr]  \ar@<-.4ex>[u]&&& \Spsym{}(\Spsym{X},S^1)  \ar@<.4ex>[lll]\ar@<-.4ex>[u]
}
\]
where the vertical adjunctions are the stabilizations~\cite[Theorem
9.1]{Hovey_symmetric-general} and the horizontal left adjoints are
given by $\bS^{\cI}_{\cR}$ and its induced functor on symmetric spectrum
objects.  The left adjoints and the right adjoints commute up to
isomorphism. 
\begin{lemma}\label{lem:stabilization-quillen-equiv}
  With respect to the absolute and positive local model structures,
  the two adjunctions
  $\Spsym{X}\rightleftarrows \Spsym{}(\Spsym{X},S^1) $ and
  $\Spsym{}(\cS^{\cI}_X,S^1) \rightleftarrows \Spsym{}(\Spsym{X},S^1)$
  are Quillen equivalences. In particular, $\bS^{\cI}_{\cR}$ models
  the stabilization of $\cS^{\cI}_X$. 
\end{lemma}
\begin{proof}
  For $\Spsym{X}\rightleftarrows \Spsym{}(\Spsym{X},S^1) $, this
  follows from Proposition~\ref{prop:SpsymX-stable} and~\cite[Theorem
  9.1]{Hovey_symmetric-general}. For the second adjunction, we note
  that the category of symmetric spectrum objects in $\cS^{\cI}_X$ is
  equivalent to the section category of the Quillen $\cI$-category
  $\bld{m}\mapsto \Spsym{}(\cS_{X(\bld{m})},S^1)$ whose structure maps
  are induced by those discussed in
  Example~\ref{ex:pseudo-local}(iv). Inspecting the cofibrations and
  fibrant objects, it follows that the stable model structure on
  $\Spsym{}(\cS^{\cI}_X,S^1)$ corresponds to the local model structure
  associated with this Quillen $\cI$-category where the categories
  $\Spsym{}(\cS_{X(\bld{m})},S^1)$ are equipped with the stable model
  structure. Analogously, we can identify $\Spsym{}(\Spsym{X},S^1)$
  with the section category of
  $\bld{m}\mapsto \Spsym{}(\cS_{X(\bld{m})},S^1)$ where now the
  structure maps are the spectrifications of the structure
  maps~\eqref{eq:Quillen-I-cat-for-SpsymX} for
  $\Spsym{X}$. Again, the stable model structure corresponds to the
  local model structure on the section category. Under these
  identifications, the adjunction in question is induced in level
  $\bld{m}$ by the left adjoints
  $S^{\bld{m}} \barsm - \colon \Spsym{}(\cS_{X(\bld{m})},S^1) \to
  \Spsym{}(\cS_{X(\bld{m})},S^1)$
  in the way explained in Lemma~\ref{lem:adjunction-on-section-cat}.
  By stability, the latter functor participates in a Quillen
  equivalence, and the claim follows by a similar argument as in the
  proof of Lemma~\ref{lem:base-change-local-Q-adjunction}.
\end{proof}
\begin{remark}
The lemma implies that the model category $\Spsym{X}$ we are interested
in is also equivalent to $\Spsym{}(\cS^{\cI}_X,S^1)$. However, the latter
category is more complicated in that it has separate $\cI$- and spectrum directions, and it is less suited for the approach to Thom spectra in Section~\ref{subsec:T_R}
and the analysis of parametrized spectra carried out in \cite[Section~5]{HS-twisted}.
\end{remark}

\begin{corollary}\label{cor:base-change-local-Q-adjunction}
If $f\colon X \to Y$ is an $\cI$-equivalence, then
$f_!\colon \Spsym{X} \rightleftarrows \Spsym{Y}\colon f^*$ is a Quillen
equivalence with respect to the absolute and positive local model structures. 
\end{corollary}
\begin{proof}
  We know from Lemma~\ref{lem:base-change-local-Q-adjunction} that
  $(f_!,f^*)$ is a Quillen adjunction. By properness of the
  $\cI$-model structures on $\cS^{\cI}$ (see~\cite[Proposition
  3.2]{Sagave-S_diagram}) and the discussion in
  Example~\ref{ex:pseudo-local} (iv), it follows that $f$ induces a
  Quillen equivalence $\cS^{\cI}_X \rightleftarrows
  \cS^{\cI}_Y$.
  By~\cite[Theorem 9.3]{Hovey_symmetric-general}, this Quillen
  equivalence induces a Quillen equivalence on the stabilization. The
  claim follows by the last lemma and 2-out-of-3 for Quillen
  equivalences.
\end{proof}
Let again $X_{h\cI} = \hocolim_{\cI}X$ denote the Bousfield--Kan
homotopy colimit of an $\cI$-space $X\colon \cI \to \cS$ and let
$\overline{X}$ be the bar resolution of $X$, that is, the homotopy
left Kan extension of $X$ along $\mathrm{id}_{\cI}$.  Then the adjoint
of the isomorphism $\colim_{\cI}(\overline{X}) \iso X_{h\cI}$ and the
canonical map $\overline{X} \to X$ provide a zig-zag of
$\cI$-equivalences $\const_{\cI}X_{h\cI} \ot \overline{X} \to X$ (see
e.g.~\cite[\S 4]{Schlichtkrull_Thom-symmetric}).  Using this, the
previous corollary implies: 
\begin{corollary}\label{cor:SpsymX-vs-SpsymXhI}
Let $X$ be an $\cI$-space. Then there is a chain of Quillen equivalences
relating $\Spsym{X}$ and $\Spsym{X_{h\cI}}$ with the absolute local model structures. \qed
\end{corollary}

\section{Comparison with the local integral model structure}\label{compmodel}
Our next aim is to prove a version of
Proposition~\ref{prop:level-Grothendieck-construction} for the local
model structures, i.e., we show that the $\cI$-model structures on
$\cS^{\cI}$ discussed in Example~\ref{ex:pseudo-local}(i) and the
local model structures on the $\Spsym{X}$ assemble to the local model
structures on $\SpsymR$.

\begin{lemma}\label{lem:equivalent-fibrancy-conditions}
  Let $X$ be an absolute (resp.\ a positive) $\cI$-fibrant $\cI$-space
  and let $(E,X)$ be an object in $\Spsym{X}$. Then $(E,X)$ is fibrant
  in the absolute (resp.\ positive) local model structure on
  $\Spsym{\cR}$ if and only if it is fibrant in the absolute
  (resp.\ positive) local model structure on $\Spsym{X}$.
\end{lemma}
\begin{proof} An object $(E,X)$ is absolute local fibrant in
  $\Spsym{\cR}$ if and only if it is absolute level fibrant and the
  horizontal maps in the
  square~\eqref{eq:fibrant-in-SpsymR-testsquare} are weak equivalences
  for all $\alpha$ in $\cI$. Under the assumptions on $X$, this holds
  if and only if $(E,X)$ is absolute level fibrant in $\Spsym{X}$
  and~\eqref{eq:fibrant-in-SpsymR-testsquare} it homotopy cartesian
  for all $\alpha$ in $\cI$. The positive case is analogous.
\end{proof}

\begin{lemma}\label{lem:SpsymK-SpsymR-maps-between-free}
  Let $K$ be a cofibrant space, let $(U,K)$ be cofibrant in $\cS_K$, and
  let $\alpha \colon \bld{m} \to \bld{n}$ be a map in $\cI$. Then
  $\alpha_{(U,K)}\colon F^{\Spsym{K}}_{\bld{n}}((U,K)\barsm S^{\bld{n}-\alpha}) \to
  F^{\Spsym{K}}_{\bld{m}}(U,K)$
  is a local weak equivalence in
  $\Spsym{\cR}$.
\end{lemma}
\begin{proof}
We consider the commutative diagram 
\[\xymatrix@-1pc{ F^{\Spsym{\cR}}_{\bld{n}}((U,K)\barsm S^{\bld{n}-\alpha})  \ar[d]_-{\sim} && \bS^{\cI}_b[F^{\cS^{\cI}}_{\bld{n}}(K)] \ar[d]_-{\sim}\ar[rr] \ar@{ >->}[ll]  && \bS^{\cI}_b[\const_{\cI}(K)]\ar[d]^-{=} \\ 
 F^{\Spsym{\cR}}_{\bld{m}}(U,K)   && \bS^{\cI}_b[F^{\cS^{\cI}}_{\bld{m}}(K)]\ar[rr] \ar@{ >->}[ll] && \bS^{\cI}_b[\const_{\cI}(K)].}
\]
The left hand vertical map is a local weak equivalence in
$\Spsym{\cR}$ by Corollary~\ref{cor:alpha-Z-weak-equivalence}. The
middle vertical map is because
$F^{\cS^{\cI}}_{\bld{n}}(K) \to F^{\cS^{\cI}}_{\bld{m}}(K)$ is an
$\cI$-equivalence and $\bS^{\cI}_{b}$ is left Quillen by
Lemma~\ref{lem:Q-adjunctions-on-local-model-str}. Since $(U,K)$ and
$(U,K) \barsm S^{\bld{n}-\alpha}$ are cofibrant and $\bS^{\cI}_b[F^{\cS^{\cI}}_{\bld{m}}(K)] \iso F^{\Spsym{\cR}}_{\bld{m}}(K,K)$, the left hand
horizontal maps are cofibrations in $\Spsym{\cR}$. 
Since $\Spsym{\cR}$
is left proper by Proposition~\ref{prop:existence-local-model-structure}, 
Lemma~\ref{lem:free-cR-and-free-relative-X-pushout} implies the claim. 
\end{proof}

\begin{lemma}\label{lem:inclusion-preserves-fib-replacement}
  Let $K$ be cofibrant in $\cS$. With respect to the absolute or
  positive local model structures, the inclusion functor
  $\Spsym{K} \to \Spsym{\cR}$ preserves acyclic cofibrations with
  fibrant codomain.
\end{lemma}
\begin{proof}
  Given a map $\alpha_Z$ in the set of maps we use to form the local
  model structure on $\Spsym{K}$, we use the mapping cylinder
  construction resulting from the simplicial structure of $\Spsym{K}$
  to factor it into a cofibration $\alpha_Z^c$ followed by an absolute
  level equivalence. We let $J$ be the set of maps in $\Spsym{K}$ that
  is the union of the generating acyclic cofibrations for the absolute
  level model structure and the maps of the form
  $\alpha_Z^c \tensor i$ where $i = (\partial D^n \to \Delta^n)$ runs
  through the generating cofibrations $\sset$ and $\alpha_Z$ runs
  through the maps we are localizing at.  Writing $X = \const_{\cI}K$,
  an object $(E,X)$ is fibrant in $\Spsym{K}$ if and only if
  $(E,X) \to \bS^{\cI}_b[X]$ has the right lifting property with
  respect to $J$ (compare~\cite[Proposition~4.2.4]{Hirschhorn_model}
  for an analogous statement using cosimplicial resolutions). The
  domains of the maps in $J$ are small relative to $J$-cell complexes
  because this property is inherited from the cofibrantly generated
  model category $\cS_K$ (and preserved by forming the mapping
  cylinder). Hence we can apply the small object argument to see that
  the fibrant replacement in the local model structure on $\Spsym{K}$
  is the retract of a $J$-cell complex.

  By Lemma~\ref{lem:SpsymK-SpsymR-maps-between-free} and
  Proposition~\ref{prop:local-pushout-product}, the maps in $J$ are
  acyclic cofibrations in the absolute local model structure on
  $\Spsym{\cR}$. Since the inclusion $\Spsym{X} \to \Spsym{\cR}$
  preserves pushouts and filtered colimits, it follows that $J$-cell
  complexes are also acyclic cofibrations in $\Spsym{\cR}$. The
  claim follows because the fibrant objects in $\Spsym{K}$ and
  $\Spsym{\cR}$ coincide by
  Lemma~\ref{lem:equivalent-fibrancy-conditions}.
\end{proof}

\begin{proposition}\label{prop:model-str-spsymR-spsymX} A map in
  $\Spsym{X}$ is a weak equivalence in the absolute or positive local
  model structure if and only if it is so as a map in $\Spsym{\cR}$.
\end{proposition}
\begin{proof}
  We consider a map $\varphi\colon (E,X) \to (E',X)$ and prove the
  claim by gradually allowing more and more general cases. When
  $X = \const_{\cI}K$ is the constant $\cI$-diagram on a cofibrant
  space $K$ and both $(E,X)$ and $(E',X)$ are locally fibrant in
  $\Spsym{K}$, then they are also locally fibrant as objects in
  $\Spsym{\cR}$ by Lemma~\ref{lem:equivalent-fibrancy-conditions}, and
  the claim follows since in both categories weak equivalences between
  fibrant objects are level equivalences. When $(E,X)$ and $(E',X)$
  are not necessarily fibrant in $\Spsym{K}$, we apply the fibrant
  replacement in $\Spsym{K}$ to $\varphi$ and use
  Lemma~\ref{lem:inclusion-preserves-fib-replacement} to see that it
  is also a fibrant replacement in $\Spsym{\cR}$.  Hence the claim
  reduces to the previous case.

  In the next step, we assume that $X$ is absolute (resp.\ positive)
  cofibrant as an $\cI$-space. Setting $K = \colim_{\cI}X$, the
  adjunction counit provides an $\cI$-equivalence
  $f\colon X \to \const_{\cI}K$. Now given a map
  $\varphi\colon (E,X) \to (E',X)$ of cofibrant objects in
  $\Spsym{X}$, we apply
  Corollary~\ref{cor:base-change-local-Q-adjunction} to see that
  $(f_!,f^*)$ is a Quillen equivalence and deduce that $\varphi$ is a
  local weak equivalence in $\Spsym{X}$ if and only if $f_!(\varphi)$
  is a local weak equivalence in $\Spsym{K}$. Left properness of the
  level model structure on $\Spsym{\cR}$, the
  identification~\eqref{eq:base-change-cobase-change}, and
  Corollary~\ref{cor:S-I-b-left-right-Quillen-locally} imply that
  $\varphi$ is a local weak equivalence in $\Spsym{\cR}$ if and only
  if $f_!(\varphi)$ is. So we have reduced the claim to the previous
  step. Since the cofibrant replacement in $\Spsym{X}$ is a level
  equivalence, we may drop the cofibrancy assumption on $(E,X)$ and
  $(E',X)$ in the previous argument.

  In the last step, we consider a general $X$ and let
  $f\colon X^c \to X$ be an absolute (resp.\ positive) acyclic
  fibration with absolute (resp.\ positive) cofibrant domain.  Since
  $f$ is a level equivalence and $\cS$ is proper, $(f_!,f^*)$ is a
  Quillen equivalence with respect to the level model
  structures. Hence our test map $\varphi\colon (E,X) \to (E',X)$ is
  level equivalent to the image of a map of cofibrant objects
  $\varphi^c\colon (E^c,X^c) \to (E'^c,X^c)$ in $\Spsym{X^c}$ under
  $f_!$. Since $(f_!,f^*)$ is a Quillen equivalence with respect to
  the local model structures by
  Corollary~\ref{cor:S-I-b-left-right-Quillen-locally}, $\varphi$ is a
  local equivalence in $\Spsym{X}$ if and only $\varphi^c$ is a local
  equivalence in $\Spsym{X^c}$. Since the level model structures on
  $\SpsymR$ are right proper by
  Proposition~\ref{prop:level-model-on-section}, $\varphi$ and
  $\varphi^c$ are level equivalent in $\SpsymR$. This reduces the
  general claim to the previous case.
\end{proof}
\begin{corollary}\label{cor:absolute-positive-SpsymX}
The weak equivalences in the absolute and the positive local model structures on $\Spsym{X}$ coincide.
\end{corollary}
\begin{proof}
This follows by combining Propositions~\ref{prop:abs-pos-local-equiv-SpsymR} and~\ref{prop:model-str-spsymR-spsymX}.
\end{proof}
\begin{corollary}\label{cor:acy-cof-fib-on-SpsymX}
  Let $f\colon X \to Y$ be a map of $\cI$-spaces. If $f$ is an acyclic
  cofibration (resp.\ acyclic fibration) in the absolute $\cI$-model
  structure, then $f_!\colon \Spsym{X} \to \Spsym{Y}$
  (resp.\ $f^*\colon \Spsym{Y}\to \Spsym{X}$) preserves weak
  equivalences of the local model structures. An analogous statement
  holds in the positive case.
\end{corollary}
\begin{proof}
  If $f\colon X\to Y$ is an acyclic cofibration, then $\bS^{\cI}_b[f]$
  is an acyclic cofibration in the local model structure on
  $\Spsym{\cR}$ by
  Corollary~\ref{cor:S-I-b-left-right-Quillen-locally}.  The claim
  follows by the first isomorphism in~\eqref{eq:base-change-cobase-change}
  and Proposition~\ref{prop:model-str-spsymR-spsymX}. The statement
  about $f^*$ can be proved by arguing in a dual way.
\end{proof}

Applying Definition~\ref{def:integral-model-str} to the pseudofunctor
$X \mapsto \Spsym{X}$ from~\eqref{eq:SpsymX-pseudofunctor}, the
absolute (resp.\ positive) $\cI$-model structure and the absolute
(resp.\ positive) local model structure on the $\Spsym{X}$ give rise to
absolute (resp.\ positive) local \emph{integral} cofibrations,
fibrations, and weak equivalences on the Grothendieck construction.

\begin{theorem}\label{thm:local-Grothendieck-construction}
  These classes of maps form an absolute (resp.\ positive) integral local
  model structure on the Grothendieck construction.  Under the
  equivalence with $\Spsym{\cR}$, it coincides with the absolute (resp.\ positive)
  local model structures on $\Spsym{\cR}$. 
  \end{theorem}
\begin{proof}
  Combining Lemma~\ref{lem:base-change-local-Q-adjunction},
  Corollary~\ref{cor:base-change-local-Q-adjunction}, and
  Corollary~\ref{cor:acy-cof-fib-on-SpsymX}, the existence of the
  integral model structure follows from~\cite[Theorem
  3.0.12]{Harpaz-P_Grothendieck-construction}.   For the comparison, we note that the cofibrations and fibrant
  objects of the two model structure coincide by
  Proposition~\ref{prop:level-Grothendieck-construction} and
  Lemma~\ref{lem:equivalent-fibrancy-conditions}. Hence the claim
  follows from~\cite[Proposition E.1.10]{Joyal-quasi-categories}.
\end{proof}
The last theorem and the definition of the integral model structure imply
the next two statements. 
\begin{corollary}\label{cor:local-model-structure-inherited}
A map in $\Spsym{X}$ is a cofibration, fibration, or weak equivalence in the absolute or positive local model structure if and only if it is so as a map in $\Spsym{\cR}$. 
\qed
\end{corollary}

\begin{corollary}\label{cor:characterization-local-equivalences}
Let $(E,X) \to (F,Y)$ be a map in $\Spsym{\cR}$ with $f \colon X \to Y$ as map of base $\cI$-spaces. Then the following are equivalent:
\begin{enumerate}[(i)]
\item The map $(E,X) \to (F,Y)$ is a local weak equivalence in $\Spsym{\cR}$.
\item $f$ is an $\cI$-equivalence and a cofibrant replacement $(E^c,X) \to (E,X)$ in $\Spsym{X}$ induces a local weak equivalence $(f_!(E^c),Y) \to (f_!(E),Y) \to (F,Y)$ in $\Spsym{Y}$.
\item $f$ is an $\cI$-equivalence and a fibrant replacement $(F,Y) \to (F^f,Y)$ in $\Spsym{Y}$ induces a local weak equivalence $(E,X) \to (f^*(F),X) \to (f^*(F^f),X)$ in $\Spsym{X}$. \qed
\end{enumerate}
\end{corollary}

We have now proved the main results about the local model structures stated in the introduction:
\begin{proof}[Proof of Theorems~\ref{thm:model-str-introduction} and \ref{thm:integral-introduction}]
Theorem~\ref{thm:model-str-introduction} is a combination of 
Corollary~\ref{cor:local-model-structure-inherited}, Lemma~\ref{lem:base-change-local-Q-adjunction}, Corollaries~\ref{cor:base-change-local-Q-adjunction} and~\ref{cor:SpsymX-vs-SpsymXhI}. Theorem~\ref{thm:integral-introduction} is Theorem \ref{thm:local-Grothendieck-construction}.
\end{proof}
\begin{remark}
Using \cite[Theorem 4.2]{Cagne-M_bifibrations}, Theorem~\ref{thm:local-Grothendieck-construction} also implies that the functors $f_!$ and $f^*$ satisfy the \emph{homotopical Beck--Chevalley condition} formulated in~\cite[Definition 4.1]{Cagne-M_bifibrations}.
\end{remark}

Let $M$ be a commutative $\cI$-space monoid. It is now easy to see
that the symmetric monoidal product on $\Spsym{M}$ discussed
in~\eqref{eq:spsymrel-M-product} is also compatible with the local
model structures:
\begin{proposition}\label{prop:SpsyM-pushout-product} The category 
  $\Spsym{M}$ satisfies the pushout product axiom with respect to the
  absolute and positive local model structures.
\end{proposition}
\begin{proof}
Since $\mu_!\colon \Spsym{M\boxtimes M} \to \Spsym{M}$ is left Quillen, Corollary~\ref{cor:local-model-structure-inherited} and the pushout product axiom in $\Spsym{\cR}$ provide the pushout product axiom for $\Spsym{M}$. 
\end{proof}
By the discussion following Theorem~\ref{thm:SpsymM-infty-identification} below, the previous proposition provides a symmetric monoidal model for the stabilization of the category of spaces over and under a given $E_{\infty}$ space. 
\begin{remark}
  In view of Corollary~\ref{cor:local-model-structure-inherited}, one
  may wonder if one can simply use the local model structure on
  $\Spsym{\cR}$ to define the local model structures on the
  subcategories $\Spsym{X}$ and avoid many of the intermediate steps
  in our construction. The problem with this approach is that the
  factorizations in $\Spsym{\cR}$ do not necessarily give rise to
  factorizations in $\Spsym{X}$. Moreover, the important property that
  an $\cI$-equivalence $f\colon X \to Y$ induces a Quillen equivalence
  $f_!\colon \Spsym{X} \rightleftarrows \Spsym{Y} \colon f^*$ does not
  appear to be a consequence of the local model structure on $\SpsymR$
  since this would require a form of right properness of
  $\Spsym{\cR}$.
\end{remark}
\subsection{Comparison of simplicial and topological variants} When developing our model structures, we allowed the underlying category of spaces $\cS$ to be either the category $\sset$ of simplicial sets or  the category of compactly generated weak Hausdorff spaces $\tp$. The Quillen adjunction \[|\!-\!| \colon \sset \rightleftarrows \tp \colon \Sing\] relating them induces an adjunction \[|\!-\!| \colon \SpsymR(\sset) \rightleftarrows \SpsymR(\tp)\colon \Sing\] on the associated categories of symmetric spectra in retractive spaces with  $|\!-\!|$ strong symmetric monoidal and $\Sing$ lax  symmetric monoidal. 
\begin{proposition}\label{prop:simp-top-spsymR}
The adjunction $\SpsymR(\sset) \rightleftarrows \SpsymR(\tp)$ is a Quillen equivalence with respect to the absolute and positive level and local model structures. 
\end{proposition}
\begin{proof}
This can be checked from the definitions or deduced from~\cite[Theorem 9.3]{Hovey_symmetric-general}.
\end{proof}
Now let $X$ be an $\cI$-diagram of simplicial sets, $Y$ an $\cI$-diagram of topological spaces, and $|X| \to Y$ a map with adjoint $X \to \Sing(Y)$. Then the two composites 
\begin{align*}
&\Spsym{X}(\sset) \xrightarrow{|-|}  \Spsym{|X|}(\tp) \xrightarrow{(|X| \to Y)_!} \Spsym{Y}(\tp) \qquad \text{and} \\ & \Spsym{Y}(\tp) \xrightarrow{\Sing} \Spsym{\mathrm{Sing(Y)}}(\sset) \xrightarrow{(X \to \mathrm{Sing(Y)})^*} \Spsym{X}(\sset)
\end{align*}
define an adjunction $\Spsym{X}(\sset) \rightleftarrows \Spsym{Y}(\tp)$. Taking $|X| \to Y$ or its adjoint to be the identity gives adjunctions 
$\Spsym{X}(\sset) \rightleftarrows \Spsym{|X|}(\tp)$ and $\Spsym{\Sing(Y)}(\sset) \rightleftarrows \Spsym{Y}(\tp)$. 
\begin{proposition}\label{prop:simp-top-spsymX}
The last two adjunctions are  Quillen equivalences with respect to the absolute and positive level and local model structures.
\end{proposition}
\begin{proof}
This follows from Proposition~\ref{prop:simp-top-spsymR} 
and Corollaries~\ref{cor:compatible-level-model-str} and~\ref{cor:local-model-structure-inherited}.
\end{proof}
It is also easy to check that these adjunctions respect the convolution product~\eqref{eq:spsymrel-M-product} if the base is a commutative $\cI$-space monoid. 

\subsection{Model structures on parametrized commutative ring spectra}\label{subsec:commutative-monoids}
Next we explain how to lift the previously constructed
local model structures to commutative ring spectra and for this we wish to apply the general theory from \cite{Pavlov-S_symmetric-operads}. Since this theory is only applicable in the simplicial setting, we shall limit ourselves to working simplicially when discussing model structures on commutative ring spectra. Thus, for the rest of this section we specify that the underlying category $\cS$ of spaces be the category $\sset$ of simplicial sets. We briefly comment on the topological setting in Remark~\ref{rem:top-problems-commutative-parametrized-model}.

We write $\cC\SpsymR$ for the category of commutative ring spectra in 
$\SpsymR$, i.e., for commutative monoid objects in $(\SpsymR,\barsm, \bS)$. 
\begin{theorem}\label{thm:positive-local-on-CSpsymR}
The category $\cC\SpsymR$ admits a positive local model structure where
a map is a fibration or weak equivalence if and only if the underlying
map in $\SpsymR$ is. 
\end{theorem}
\begin{proof}
  We first notice that the absolute and positive model structures can
  also be constructed using~\cite[Theorem
  3.2.1]{Pavlov-S_symmetric-operads}. For this we have to show that
  the category $\sset_{\cR}$ of retractive simplicial sets satisfies
  the requirements of~\cite[Definition~2.1]{Pavlov-S_symmetric-operads}. This holds since $\sset_{\cR}$ is
  locally presentable, all objects are cofibrant, the domains and
  codomains of the generating cofibrations are finitely presentable,
  and $\sset_{\cR}$ satisfies the pushout product axiom. The theorem
  then follows from~\cite[Theorem~4.1]{Pavlov-S_symmetric-operads}.
\end{proof}

\begin{remark}\label{rem:associative-fibrant-repl} In fact, the result
  in~\cite{Pavlov-S_symmetric-operads} shows that the positive local
  model structure on $\SpsymR$ has favorable monoidal
  properties~\cite[Proposition 3.5.1]{Pavlov-S_symmetric-operads} that
  allow it to be lifted to algebras over general colored symmetric
  operads. In particular, there is also a lifted model structure on
  associative parametrized ring spectra. 
\end{remark}

Now let $M$ be a commutative $\cI$-space monoid and consider the
category $\Spsym{M}$ with the symmetric monoidal product~\eqref{eq:spsymrel-M-product}. We write $\cC\Spsym{M}$ for the
category of commutative $M$-relative symmetric ring spectra, i.e., the
commutative monoid objects in $(\Spsym{M},\barsm,\bS_M)$.
\begin{theorem}\label{thm:positive-local-on-CSpsymM}
  The category $\cC\Spsym{M}$ admits a \emph{positive local model
    structure} where a map is a fibration or weak equivalence if and
  only if the underlying map in $\Spsym{M}$~is.
\end{theorem}
The proof of this statement is more difficult because $\Spsym{M}$ not
being equivalent to symmetric spectrum objects in some category
prevents us from applying the results
of~\cite{Pavlov-S_symmetric-operads} directly.  Instead, we rely on
the following lemma. To formulate it, we let
$\bC^{\cR} \colon \SpsymR \to \cC\SpsymR$,
$\bC^{M}\colon \Spsym{M} \to \cC\Spsym{M}$, and
$\bC^{\cI}\colon \cS^{\cI} \to \cC\cS^{\cI}$ be the free functors
which are left adjoint to the respective forgetful functors. We also
note that there is a canonical inclusion functor
$\cC\Spsym{M} \to \cC\SpsymR$ that identifies $\cC\Spsym{M}$ with the
fiber of the projection functor  $\pi_b\colon \cC\SpsymR \to \cC\cS^{\cI}$
and that $\bS^{\cI}_b$ induces a
functor $\bS^{\cI}_b\colon \cC\cS^{\cI} \to \cC\SpsymR$.
\begin{lemma}\label{lem:comparison-free-commutative}
  Let $(A,M)$ be an object in $\cC\Spsym{M}$ and let $f\colon (D,M)
  \to (E,M)$ and $g\colon (D,M) \to (A,M)$ be a maps in
  $\Spsym{M}$. Let $\widetilde g \colon \bC^{\cR}(D,M) \to (A,M)$ and
  $\hat g \colon \bC^{M}(D,M) \to (A,M)$ be the adjoints of $g$ with
  respect to the above adjunctions. Then the cobase change of
  $\bC^M(f)$ along $\hat g$ in $\cC\Spsym{M}$ is isomorphic to the
  cobase change of $\bC^{\cR}(f)$ along $\widetilde{g}$ in
  $\cC\SpsymR$.
\end{lemma}
\begin{proof}
The underlying commutative $\cI$-space monoid of $\bC^{\cR}(D,M)$ is
$\bC^{\cI}(M)$. Inspecting the universal properties of the free functors
shows that $\bC^{M}(D,M)$ is isomorphic to $\bC^{\cR}(D,M) \barsm_{\bS^{\cI}_{b}[\bC^{\cI}(M)]} \bS^{\cI}_{b}[M]$, the cobase change of $\bC^{\cR}(D,M)$ along the map
given by applying $\bS^{\cI}_{b}$ to the adjoint $\bC^{\cI}(M) \to M$ of $\id_{M}$. Commuting pushouts in $\cC\SpsymR$, we see that 
\begin{align*} &(A,M) \barsm_{ \bC^{M}(D,M)}  \bC^{M}(E,M)\\ \iso &\left((A,M) \barsm_{\bS^{\cI}_{b}[M]}  \bS^{\cI}_{b}[M] \right)\barsm_{\left(\bC^{\cR}(D,M) \barsm_{\bS^{\cI}_{b}[\bC^{\cI}(M)] } \bS^{\cI}_{b}[M]\right)} \!\! \left(\bC^{\cR}(E,M)  \barsm_{\bS^{\cI}_{b}[\bC^{\cI}(M)] } \bS^{\cI}_{b}[M]\right)\\
\iso & (A,M) \barsm_{ \bC^{\cR}(D,M)}\bC^{\cR}(E,M)\ .
\end{align*}
As the inclusion functor $\cC\Spsym{M} \to \cC\SpsymR$ preserves pushouts, 
the claim follows. 
\end{proof}

\begin{proof}[Proof of Theorem~\ref{thm:positive-local-on-CSpsymM}]
We apply~\cite[Theorem 11.3.2]{Hirschhorn_model} to the free/forgetful adjunction $\bC^{M} \colon \Spsym{M} \rightleftarrows \cC\Spsym{M} \colon U$. Let $J$ be a set of generating acyclic cofibrations for the positive local model structure on $\Spsym{M}$ and let $\bC^{M}(J)$ be its image under $\bC^{M}$. The non-trivial part is to show that relative $\bC^{M}(J)$-cell complexes are local equivalences. Lemma~\ref{lem:comparison-free-commutative} and the fact that filtered colimits in $\cC\Spsym{M}$ and $\cC\SpsymR$ are both created in $\SpsymR$ imply that this follows from the corresponding property for $\cC\SpsymR$ resulting from Theorem~\ref{thm:positive-local-on-CSpsymR}.
\end{proof}
\begin{remark}\label{rem:top-problems-commutative-parametrized-model}
We expect that there are analogous model structures on associative and commutative ring spectra in $\SpsymR$ and $\Spsym{M}$ in the topological setting. However, the construction of such model structures will most likely require an elaborate analysis of $h$-cofibrations that we wish to avoid in the present paper. (Even the associative case is not an immediate consequence of~\cite[Theorem 4.1(3)]{Schwede-S_algebras_modules} since we do not know if the topological  $\cS_{\cR}$, $\SpsymR$ or $\Spsym{M}$ satisfy the monoid axiom.)

Nonetheless, we note that our results suffice to fibrantly replace associative or commutative parametrized ring spectra in the topological $\SpsymR$: combining Lemma \ref{lem:fibrant-in-local} with the fact that the geometric realization $|\!-\!|\colon \sset \to \tp$ preserves fibrations and weak equivalences, it follows that $|\!-\!| \colon  \SpsymR(\sset) \to \SpsymR(\tp)$ preserves locally fibrant objects. Thus applying the singular complex, forming a fibrant replacement, and then passing to the realization gives a topological fibrant replacement functor for associative or commutative parametrized ring spectra that is related to the identity functor by a zig-zag of local equivalences. 
\end{remark}

Since the left adjoint functors $\iota_t \colon \cS^{\cI} \to
\cS^{\cI}_{\cR}$ and $\bS^{\cI}_{\cR} \colon \cS^{\cI}_{\cR} \to \SpsymR$
from Construction~\ref{const:various-adjunctions} are strong symmetric
monoidal, they induce adjunctions
\[\iota_t\colon \cC\cS^{\cI}\rightleftarrows
\cC\cS^{\cI}_{\cR}\colon\pi_t \quad \text{ and } \quad
\bS^{\cI}_{\cR}\colon \cC\cS^{\cI}_{\cR} \rightleftarrows \cC\SpsymR
\colon \Omega^{\cI}_{\cR}.\]
\begin{lemma}\label{lem:various-com-adjunctions}
  These adjunctions and their composite $\bS^{\cI}_{t} \colon
  \cC\cS^{\cI}\rightleftarrows \cC\SpsymR \colon \Omega^{\cI}_{t}$ are
  Quillen adjunctions with respect to the positive local model structures.
\end{lemma}
\begin{proof}
Arguing with the right adjoints, the claim follows from Lemma~\ref{lem:Q-adjunctions-on-local-model-str}. 
\end{proof}

\section{Parametrized homology and cohomology}\label{sec:tw-coho}
In this section we define the parametrized (co)homology theories associated to a parametrized spectrum that were outlined in the introduction. Concrete examples arise from the universal line bundle (see Section~\ref{sec:univ-bundle} and Proposition~\ref{cohom infty comp}) and the twisted $K$-theory spectra studied in~\cite{HS-twisted}.

Key ingredients for the definition of parametrized (co)homology are the adjoints of the derived restriction that we discuss now. If $f\colon Y \to X$ is a map of $\cI$-spaces, then
$f^*\colon \Spsym{X} \to \Spsym{Y}$ is right Quillen with respect to
the absolute local model structures and thus induces a right derived
functor
$\mathbb Rf^* \colon \mathrm{Ho}(\Spsym{X}) \to
\mathrm{Ho}(\Spsym{Y})$
with left adjoint
$\mathbb Lf_!\colon \mathrm{Ho}(\Spsym{Y}) \to
\mathrm{Ho}(\Spsym{X})$.  

\begin{proposition}\label{prop:Rf-upper-star-has-right-adjoint}
The functor $\mathbb Rf^* \colon \mathrm{Ho}(\Spsym{X}) \to \mathrm{Ho}(\Spsym{Y})$ is also a left adjoint. 
\end{proposition}
The proposition will be proved at the end of this section. 
\begin{definition}\label{def:Rf_ast_RGamma}
We write $\mathbb R f_* \colon \mathrm{Ho}(\Spsym{Y}) \to \mathrm{Ho}(\Spsym{X})$ for the right adjoint of $\mathbb Rf^*$ that results from the previous proposition. When $X = *$, we use the notation $\mathbb R \Gamma$ for the functor $\mathbb R (Y\to *)_*\colon \mathrm{Ho}(\Spsym{Y}) \to \mathrm{Ho}(\Spsym{})$ and the notation $\mathbb L\Theta$ for $\mathbb L (Y\to *)_! \colon \mathrm{Ho}(\Spsym{Y}) \to \mathrm{Ho}(\Spsym{})$. 
\end{definition}
We stress that since $f^*\colon \Spsym{X} \to \Spsym{Y}$ is in general
not left Quillen, the functor $\mathbb R f_*$ is not the right derived
functor of a right Quillen functor.  In the context of topological
spaces, an explicit description of $\mathbb R \Gamma$ in a useful special
case is given in Lemma~\ref{gamma-top} below.  We also point out that when
working over simplicial sets, deriving the left adjoint
$\Theta = (Y \to *)_!$ is not really necessary since it preserves
level equivalences and thus sends local equivalences to stable
equivalences.

The following statement will also be shown at the end of this section. 
\begin{proposition}\label{prop:monoidal_pull}
  Given maps of $\cI$-spaces $f\colon X' \to X$ and
  $g\colon Y' \to Y$, the lax monoidal structure map
  $f^*(E,X) \barsm g^*(F,Y) \to (f\boxtimes g)^*((E,X) \barsm (F,Y))$
  from~\eqref{eq:upper-star-lax-monoidal} induces the following
  natural isomorphism in $\mathrm{Ho}(\SpsymR)$:
  \begin{equation}\label{eq:Rf-upper-star-strong-monoidal} (\mathbb R
    f^*) (E,X) \barsm_{\mathbb L} (\mathbb R g^*) (F,Y)
    \xrightarrow{\iso} (\mathbb R (f\boxtimes_{\mathbb L} g)^*) (
    (E,X)\barsm_{\mathbb L}(F,Y))
  \end{equation}
\end{proposition}
For objects $(E',X')$ and $(F',Y')$ in $\SpsymR$, the isomorphism~\eqref{eq:Rf-upper-star-strong-monoidal} and the  units and counits for the adjunctions resulting from Proposition~\ref{prop:Rf-upper-star-has-right-adjoint} applied to $f,g$ and $f\boxtimes g$  induce natural maps 
\begin{equation}\label{eq:R-Gamma-monoidal-str}
\mathbb R f_* (E',X') \barsm_{\mathbb L} \mathbb R g_* (F',Y') \to  \mathbb R (f\boxtimes_{\mathbb L} g)_* ((E',X') \barsm_{\mathbb L} (F',Y')) 
\end{equation}
in $\mathrm{Ho}(\SpsymR)$ that are associative, commutative, and unital.

\subsection{\texorpdfstring{$\cI$}{I}-spacification}\label{subsec:I-spacification} 
To be able to define parametrized (co-)homology and Thom spectra from space level data, we now
recall from \cite[\S 4.2]{Schlichtkrull_Thom-symmetric} and \cite[\S
4.1]{Basu_SS_Thom} how one can pass from spaces to $\cI$-spaces. For any $\cI$-space $X$, there is an
\emph{$\cI$-spacification} functor
\begin{equation}\label{eq:I-spacificaton}
P_X\colon \cS/X_{h\cI} \to \cS^{\cI}/X, \quad (\tau \colon K \to X_{h\cI}) \mapsto ( P_X(\tau) \colon P_\tau(K) \to X) 
\end{equation}
that is a homotopy inverse of the homotopy colimit functor. We briefly recall its definition. Writing $\overline{X}$ for the  homotopy left Kan extension of $X$ along $\id_{\cI}$, the canonical map $t\colon \overline{X} \to X$ is a natural level equivalence that we refer to as the bar resolution. There is a natural isomorphism $\colim_{\cI}\overline{X} \iso X_{h\cI}$ with adjoint $\pi\colon \overline{X} \to \const_{\cI}X_{h\cI}$. A map of spaces $\tau\colon K \to X_{h\cI}$ gives rise to a map of $\cI$-spaces
\[ \const_{\cI}K \times_{\const_{\cI}X_{h\cI}} \overline{X} \xrightarrow{\mathrm{pr}} \overline{X} \xrightarrow{t} X. \]
This construction can be viewed as a functor $\cS/X_{h\cI} \to \cS^{\cI}/X$. In the topological case, the homotopy invariant $\cI$-spacification functor~\eqref{eq:I-spacificaton} is defined by precomposing it with the standard Hurewicz fibrant replacement $\Gamma(\tau)\colon \Gamma_\tau(K) \to X_{h\cI}$ of $\tau$ (not to be confused with the meaning of $\Gamma$ in Definition~\ref{def:Rf_ast_RGamma}). In the simplicial case, we replace the $\Gamma$ by the functor sending a map of simplicial sets $\tau \colon K \to X_{h\cI}$ to the map $\Gamma(\tau)$ defined by the right hand pullback square in the diagram 
\[ \xymatrix@-1pc{
\Sing |K| \ar[rr]^-{\sim} && \Sing \Gamma_{|\tau|}(|K|) \ar@{->>}[rr] && \Sing |X_{h\cI}| \\ 
K \ar[u]^{\sim} \ar[rr]^{\sim} && \Gamma_\tau(K) \ar[u]^{\sim}  \ar@{->>}^{\Gamma(\tau)}[rr] && X_{h\cI} \ar[u]^{\sim}. 
}
\] 
The lower left hand map arises from the universal property of the pullback.
Compared to a replacement by fibration obtained from the small object argument, this functor $\Gamma$ has the advantage of being lax monoidal and preserving operad actions. Both in the simplicial and the topological case, the resulting $\cI$-spacification functor $P_X$ then sends weak equivalences to $\cI$-equivalences. When $M$ is a commutative $\cI$-space monoid, $M_{h\cI}$ is an algebra over the Barratt--Eccles operad, and $P_M$ preserves actions of operads augmented over the Barratt--Eccles operad and is lax monoidal.

We also note the following naturality statement for later use. 
\begin{lemma}\label{lem:PM-natural}
If $\rho \colon M \to N$ is a map of commutative $\cI$-space monoids, then there is a natural map $\rho\circ P_M(\tau) \to P_N(\rho_{h\cI} \circ \tau)$ of spaces over $N$ that is an $\cI$-equivalence if $\rho$ is.  \qed
\end{lemma}

\subsection{Parametrized homology and cohomology}\label{sec: coho}
To define the parametrized (co)ho\-mology groups associated with a parametrized spectrum $(E,X)$, we use the $\cI$-spacification discussed in~\eqref{eq:I-spacificaton}.
 Given a map $\tau \colon K \to X_{h\cI}$, we use the shorthand notation $\tau_{\cI}= P_X(\tau)\colon P_{\tau}(K) \to X$
  and write
 $\mathbb R \tau_{\cI}^*: \mathrm{Ho}(\Spsym{X}) \rightarrow \mathrm{Ho}(\Spsym{P_{\tau}(K)})$ for the induced functor. Moreover, for any  $\cI$-space $Y$ the functors $\mathbb L\Theta, \mathbb R\Gamma \colon \mathrm{Ho}(\Spsym{Y}) \rightarrow \mathrm{Ho}(\Spsym{})$ denote the left and right adjoint, respectively, of $\mathbb R (Y\to *)^*$, the derived pullback functor along the unique map $Y \rightarrow *$.

\begin{definition}\label{def:local-cohomology}
For a parametrized spectrum $(E,X) \in \Spsym{\cR}$ the associated parametrized (co)homology theories are given by 
\begin{align*}(E,X)_n \colon \cS/X_{h\cI} \longrightarrow \mathrm{Ab}, \quad & (\tau \colon K \rightarrow X_{h\cI}) \longmapsto \pi_{n}(\mathbb L\Theta)(\mathbb R \tau_\cI^*) (E,X) \quad\text{ and}\\
(E,X)^n \colon \cS/X_{h\cI} \longrightarrow \mathrm{Ab}, \quad & (\tau \colon K \rightarrow X_{h\cI}) \longmapsto \pi_{-n}(\mathbb R\Gamma)(\mathbb R \tau_\cI^*) (E,X). 
\end{align*}
\end{definition}
The functoriality of parametrized homology (resp.\ cohomology) results from the  adjunction counit of $(\mathbb L f_!, \mathbb R f^*)$ (resp.\ the adjunction unit of $(\mathbb R f^*, \mathbb R f_*)$). 

Instead of directly verifying the usual properties of a (co)homology theory (including the construction of relative terms and boundary maps), let us proceed by comparing these definitions with those of May and Sigurdsson \cite[Definition 20.2.4]{May-S_parametrized}, which they show satisfy a version of the usual axioms for a (co)homology theory. In Proposition~\ref{cohom infty comp} we will also compare Definition~\ref{def:local-cohomology} with the $\infty$-categorical counterparts from \cites{Ando-B-G-H-R_infinity-Thom, Ando-B-G_parametrized}. 
\begin{proposition}\label{prop:comparison-twisted-homology}
For a constant $\cI$-space $X = \const_{\cI} B$, a fiberwise orthogonal spectrum $E \in \Sp_{B}^\mathrm{O}$ (in the sense of 
\cite[Chapter~11]{May-S_parametrized}), and $\tau \colon K \rightarrow X_{h\cI} = B \times B\cI$, there is a canonical isomorphism between our $(E,X)_*(K,\tau)$ and $(E,X)^*(K,\tau)$ on the one hand and the definitions from \cite[Definition 20.2.4]{May-S_parametrized} applied to $K \xrightarrow{\tau} X_{h\cI} \xrightarrow{\mathrm{pr}} B$ on the other.
\end{proposition}

\begin{remark} Since in the situation of the proposition the forgetful functor $\Sp^{\mathrm O}_B \rightarrow \Spsym{X}$ induces an equivalence on homotopy categories by the conjunction of \cite[Theorem B.2]{Ando-B-G_parametrized}, our Corollary \ref{cor:SpsymX-vs-SpsymXhI}, and Lemma \ref{lem:SpsymX-infty-identification}, we can find a weakly equivalent orthogonal spectrum to an arbitrary $(E,X) \in \Spsym{X}$. Thus we can deduce the (co)homological consequences of the proposition without the orthogonality assumption. Investing Corollary \ref{cor:SpsymX-vs-SpsymXhI} also for non-constant $X$ we can then also deduce them for arbitrary $(E,X) \in \SpsymR$. We leave the details to the reader.
\end{remark}

\begin{proof}[Proof of Proposition~\ref{prop:comparison-twisted-homology}]
Let us first recall the definitions: For a fiberwise orthogonal parametrized spectrum $E$ over a space $B$ as in \cite[Definition 11.2.3]{May-S_parametrized} and a map $\sigma\colon K \rightarrow B$, May and Sigurdsson set 
\[E_n(K,\sigma) = \pi_n(\mathbb L\Theta)(\mathbb L\Sigma^\infty_BK_+ \sm_B^{\mathbb L} E)\]
and
\[ E^n(K,\sigma) = \pi_{-n}(\mathbb R\Gamma)(\mathbb RF_B(\mathbb L\Sigma^\infty_BK_+,E)),\]
where we have adapted those functors to our notation that have occurred in our presentation. The remaining ones are $\mathbb L\Sigma^\infty_B(-)_+ \colon \mathrm{Ho}(\cS/B) \rightarrow \mathrm{Ho}(\Sp^\mathrm{O}_B)$ which adds a disjoint base section and then takes the suspension spectrum, the functor $\sm_B^{\mathbb L} \colon \mathrm{Ho}(\Sp^\mathrm{O}_B) \times \mathrm{Ho}(\Sp^\mathrm{O}_B) \rightarrow \mathrm{Ho}(\Sp^\mathrm{O}_B)$, which is the fiberwise smash product obtained by internalizing the external smash product by pullback along the diagonal,  and $\mathbb RF_B \colon \mathrm{Ho}(\Sp^\mathrm{O}_B)^\mathrm{op} \times \mathrm{Ho}(\Sp^\mathrm{O}_B) \rightarrow \mathrm{Ho}(\Sp^\mathrm{O}_B)$, which takes fiberwise function spectra (and has no direct counterpart in our setup; compare Remark~\ref{rem:internal-hom}). 

Now, from \cite[Proposition 13.7.4]{May-S_parametrized} we find $\mathbb L\Sigma^\infty_BK_+ \cong \mathbb L\sigma_!(\mathbb S_K)$, where $\bS_K$ denotes the trivially parametrized sphere spectrum over $K$ that is the unit for $\sm_K$. Then the projection formulas \cite[(11.4.5) and (11.4.6)]{May-S_parametrized} (verified for the derived functors in \cite[Proposition 13.7.5]{May-S_parametrized} or investing the comparison theorem \cite[Theorem B.2]{Ando-B-G_parametrized} also in \cite[Proposition 6.8]{Ando-B-G_parametrized}) show that the formulas of May and Sigurdsson can be rewritten as 
\[E_n(K,\sigma) = \pi_{n}(\mathbb L\Theta)(\mathbb R \sigma^*) (E,B) \quad \text{and} \quad E^n(K,\sigma) = \pi_{-n}(\mathbb R\Gamma)(\mathbb R \sigma^*) (E,B).\]

But then the conjunction of our comparison in Lemma~\ref{lem:SpsymX-infty-functorial} with \cite[Theorem B.2]{Ando-B-G_parametrized}, imply that for $X = \mathrm{const}_\cI B$ we may interpret the above formulas in our categories $\Spsym{X}, \Spsym{\const_\cI K}$ and $\Spsym{}$. The commutative diagram
\[\xymatrix@-1pc{ \const_\cI K \ar[rr] \ar[drr]_{\const_\cI \tau}&& \const_\cI \Gamma_{\tau}(K) \ar@{->>}[d]&& P_{\tau}(K) \ar[ll] \ar[d]^{\tau_{\cI}}\\
&&  \const_\cI X_{h\cI}  \ar[rr]^-{\const_\cI\mathrm{pr}} && \const_\cI B = X
}\]
with $\mathrm{pr} \colon X_{h\cI} = B \times B\cI \rightarrow B$ the projection then provides the desired isomorphisms
\[(E,X)_n(K,\tau) \iso E_n(K,\mathrm{pr} \circ \tau) \quad \text{and} \quad (E,X)^n(K,\tau) \iso E^n(K,\mathrm{pr} \circ \tau).\qedhere\]
\end{proof}
Let $(R,M)$ be a parametrized ring spectrum in $\Spsym{\cR}$ with multiplication on base $\cI$-spaces $\mu\colon {M \boxtimes M} \to M$.  Our next aim is to define pairings 
\begin{equation}\label{eq:pairings}\begin{split}\times \colon (R,M)_n(K,\tau) \otimes (R,M)_m(L,\sigma) & \longrightarrow (R,M)_{n+m}(K \times L, \tau \times_\mu \sigma)\\
\times \colon (R,M)^n(K,\tau) \otimes (R,M)^m(L,\sigma) & \longrightarrow (R,M)^{n+m}(K \times L, \tau \times_\mu \sigma)
\end{split}
\end{equation}
where $\tau \times_\mu \sigma$ refers to the composite $K \times L \xrightarrow{\tau \times \sigma} M_{h\cI} \times M_{h\cI} \xrightarrow{\mu_{h\cI}} M_{h\cI}$ in which the second map is the multiplication of the monoid $M_{h\cI}$ in spaces. It follows from Remarks~\ref{rem:associative-fibrant-repl} and~\ref{rem:top-problems-commutative-parametrized-model} that we may assume without loss of generality that $(R,M)$ is fibrant. Furthermore, if $(R,M)$ is commutative, we may assume that it is fibrant as a commutative parametrized ring spectrum by Theorem~\ref{thm:positive-local-on-CSpsymR} and Remark~\ref{rem:top-problems-commutative-parametrized-model}. 
\begin{construction}\label{constr:pairings}
We observe that there is a natural chain of maps
\begin{multline}\label{eq:multiplication-I-spacificaton}
\tau_{\cI}^*(R,M) \barsm \sigma_{\cI}^*(R,M) \to (\tau_{\cI} \boxtimes \sigma_{\cI})^*((R,M) \barsm (R,M))\\ \to (\tau_{\cI} \boxtimes \sigma_{\cI})^* \mu^* (R,M) \to (\tau \times_{\mu} \sigma)^*(R,M)
\end{multline}
in $\SpsymR$ where the first map is an instance of~\eqref{eq:upper-star-lax-monoidal}, the second map is the canonical map induced by $\mu$, and the last map is induced by the monoidal structure map of the $\cI$-spacification (see \cite[Proposition 4.17]{Schlichtkrull_Thom-symmetric} or \cite[Lemma 4.5]{Basu_SS_Thom}). 
Precomposing this chain with cofibrant replacements of the $\barsm$-factors in the source and using that $(R,M)$ is assumed to be fibrant gives a map 
\[ \mathbb R \tau_{\cI}^*(R,M) \barsm_{\mathbb L} \mathbb R \sigma_{\cI}^*(R,M) \to  \mathbb R(\tau \times_{\mu} \sigma)^*_{\cI}(R,M) \] 
on the homotopy category level. Precomposing it with the lax monoidal structure of $\mathbb L \Theta$ resulting from Lemma~\ref{lem:cobase-change-spym-lax-monoidal} and passing to homotopy groups induces the first pairing in~\eqref{eq:pairings}. Using the monoidal structure map for  $\mathbb R \Gamma$ resulting from~\eqref{eq:R-Gamma-monoidal-str} instead of that for $\mathbb L \Theta$ provides the analogous pairing in cohomology. Independence from the choices made during the construction, associativity and the fact that the unit of $(R,M)$ gives the unit $1 \in E_0(*,u)$ for the above product are now readily checked. 
\end{construction} 
\begin{remark}Proposition~\ref{prop:comparison-pairings} compares these pairings with the $\infty$-categorical variants from~\cites{Ando-B-G-H-R_infinity-Thom, Ando-B-G_parametrized}. 
\end{remark}

Now we assume in addition that $(R,M)$ is commutative and check that these products are graded commutative. In order to give meaning to this, we first define an explicit twist homomorphism
\begin{equation}\label{eq:tw-homomorphism}
\mathrm{tw}\colon (R,M)_*(K\times L,\tau\times_{\mu}\sigma)\to (R,M)_*(L \times K, \sigma \times_\mu \tau).
\end{equation}
Since $M$ is supposed to be commutative, $M_{h\cI}$ inherits the structure of an $E_{\infty}$ space with a canonical action of the Barratt-Eccles operad. Hence there is an essentially unique homotopy $M_{h\cI}\times M_{h\cI}\times I\to M_{h\cI}$ starting at the multiplication $\mu_{h\cI}$ and ending at $\mu_{h\cI}\circ \mathrm{tw}$. After precomposing with $\tau\times\sigma$, we get a homotopy $H$
from $\tau\times_{\mu}\sigma$ to $\sigma \times_\mu \tau\circ \mathrm{tw}$. Now we pull back $(R,M)$ via the $\cI$-spacification of $H$ to obtain a chain of local equivalences
\begin{equation}\label{eq:tw-chain}
(\tau\times_{\mu}\sigma)_{\cI}^*(R,M)\xrightarrow{i_0}H_{\cI}^*(R,M) \xleftarrow{i_1} (\sigma \times_\mu \tau\circ \mathrm{tw})_{\cI}^*(R,M)
\xrightarrow{\mathrm{tw}}(\sigma \times_\mu \tau)_{\cI}^*(R,M)
\end{equation}
in which $i_0$ and $i_1$ denote the endpoint inclusions. Applying $\mathbb L\Theta$ we get a diagram of stable equivalences and \eqref{eq:tw-homomorphism} is the induced map of homotopy groups. Clearly the latter does not depend on the choice of $H$. 

\begin{proposition}
The square
\[\xymatrix@-1pc{
(R,M)_n(K,\tau) \otimes (R,M)_m(L,\sigma) \ar[r]^-\times \ar[d]_{\mathrm{tw}} & (R,M)_{n+m}(K \times L, \tau \times_\mu \sigma) \ar[d]_{\mathrm{tw}} \\
            (R,M)_m(L,\sigma) \otimes (R,M)_n(K,\tau) \ar[r]^-\times &        (R,M)_{m+n}(L \times K, \sigma \times_\mu \tau),
}\]
commutes up to the sign $(-1)^{nm}$. An analogous statement holds for parametrized cohomology groups.
\end{proposition}
\begin{proof}
It suffices to consider the topological setting. Let us write $\tau_{\cI}\boxtimes_{\mu}\!\sigma_{\cI}$ for the composition of $\tau_{\cI}\boxtimes\sigma_{\cI}$ with the multiplication $\mu\colon M\boxtimes M\to M$. The commutativity assumption on $(R,M)$ implies that the first square in the diagram 
\[
\xymatrix@-1pc{
\tau_\cI^*(R,M) \barsm \sigma_\cI^*(R,M)\ar[r] \ar[d]_{\mathrm{tw}}&(\tau_\cI \boxtimes_{\mu}\! \sigma_\cI)^*(R,M) \ar[r] \ar[d]_{\mathrm{tw}}& (\tau \times_{\mu} \sigma)_{\cI}^*(R,M)\\
\sigma_\cI^*(R,M) \barsm\tau_\cI^*(R,M)\ar[r] & (\sigma_\cI \boxtimes_{\mu}\! \tau_\cI)^*(R,M) \ar[r] & (\sigma \times_{\mu} \tau)_{\cI}^*(R,M) 
}
\]
is commutative. Here the horizontal maps are defined as in \eqref{eq:multiplication-I-spacificaton}. It follows from the proof of \cite[Lemma~6.7]{Schlichtkrull_Thom-symmetric} that the maps $\pi\colon\overline M\to \const_{\cI}M_{h\cI}$ and $t\colon \overline M\to M$ going into the definition of the $\cI$-spacification functor are compatible with the actions of the Barratt-Eccles operad on these $\cI$-spaces. Hence there is a commutative diagram of homotopies 
\[
\xymatrix@-1pc{
\overline M\boxtimes \overline M \times I\ar[r] \ar[d]_{\pi\boxtimes \pi\times I}& \overline M\ar[d]_{\pi}\\
\const_{\cI}(M_{h\cI}\times M_{h\cI}\times I)\ar[r] & \const_{\cI} M_{h\cI}
}
\]
where the bottom homotopy is the one used to define the homotopy $H$ in \eqref{eq:tw-chain} and the upper homotopy starts at $\bar\mu$ and ends at $\bar\mu\circ \mathrm{tw}$. Furthermore, the composition of  
the upper homotopy with $t$ is the constant homotopy on $t\circ\bar\mu$. Using both of these homotopies, we get a natural map of $\cI$-spaces
\[
P_{\tau}(K)\boxtimes P_{\sigma}(L)\times I\to P_H(K\times L\times I),
\]
where the notation $P_{\tau}(K)$ etc.\ denote the domains for the $\cI$-spacifications as in \eqref{eq:I-spacificaton}. This is in fact a map of $\cI$-spaces over $M$ when we augment the left hand side via the constant homotopy on $\tau_{\cI}\boxtimes_{\mu}\!\sigma_{\cI}$. Pulling back $(R,M)$ along these augmentations, we end up with the commutative diagram
\[
\xymatrix@-1pc{
(\tau_{\cI}\boxtimes_{\mu}\!\sigma_{\cI})^*(R,M) \ar[r] \ar[d]_{i_0}& (\tau\times_{\mu}\sigma)_{\cI}^*(R,M)\ar[d]_{i_0}\\
(\tau_{\cI}\boxtimes_{\mu}\!\sigma_{\cI})^*(R,M) \barsm \bS^{\cI}_t[\const_{\cI}\!I] \ar[r] & H_{\cI}^*(R,M)\\
(\tau_{\cI}\boxtimes_{\mu}\!\sigma_{\cI})^*(R,M) \ar[r] \ar[u]^{i_1} \ar[d]_{\mathrm{tw}} & (\sigma \times_\mu \tau\circ \mathrm{tw})_{\cI}^*(R,M)\ar[u]^{i_1}\ar[d]_{\mathrm{tw}}\\
(\sigma_\cI \boxtimes_{\mu}\! \tau_\cI)^*(R,M) \ar[r] & (\sigma \times_\mu \tau)_{\cI}^*(R,M).
}
\]
Applying $\mathbb L\Theta$ and identifying $\Theta(\bS^{\cI}_t[\const_{\cI}\!I])$ with $\bS\wedge I_+$, we get a homotopy commutative diagram from which we deduce the statement in the proposition. The cohomological statement follows by applying $\mathbb R \Gamma$ instead of $\mathbb L \Theta$. 
\end{proof}

\subsection{Derived restriction as a left adjoint}\label{subsec:derived-restriction} We now begin to prepare for the proofs of Propositions~\ref{prop:Rf-upper-star-has-right-adjoint} and~\ref{prop:monoidal_pull}. These proofs will rely on the following three lemmas which require us to work over simplicial sets and do not have direct topological counterparts. This will not lead to limitations for the propositions since they make statements about the homotopy category. 

\begin{lemma}\label{lem:f-upper-star-left-Quillen-SR}
  If $f \colon L \to K$ is a Kan fibration in $\sset$, then the restriction functor   $f^*\colon \sset_{\cR}/\iota_b(K) \to
  \sset_{\cR}/\iota_b(L)$
  preserves weak equivalences and is left Quillen.
\end{lemma}
\begin{proof}
  Since we are working over simplicial sets, $f^*$ has a right adjoint
  by the corresponding statement for the category of sets. Since base
  change preserves colimits and monomorphisms of sets, $f^*$ preserves
  cofibrations and colimits of simplicial sets and hence cofibrations
  in $\sset_{\cR}$ by their definition. Since $\sset$ is right proper, base change along the Kan
  fibration $f$ preserves weak equivalences. Thus $f^*$ is  left Quillen. 
\end{proof}

\begin{lemma}\label{lem:f-upper-star-left-Quillen-over-constant}
  The functor
  $(\const_{\cI}f)^*\colon \SpsymR/\bS^{\cI}_b[\const_{\cI}\!K] \to
  \SpsymR/\bS^{\cI}_b[\const_{\cI}\!L]$
  is left Quillen with respect to the absolute local model structures
  provided that $f \colon L \to K$ is a Kan fibration of Kan
  complexes.
\end{lemma}
\begin{proof}
The functor $(\const_{\cI}f)^*$ is a left adjoint by the corresponding statement
for $\cS_{\cR}$ established in Lemma~\ref{lem:f-upper-star-left-Quillen-SR}. 
Let $\bld{m},\bld{n}$ be objects of $\cI$, let $(W,P) \to \iota_b(K)$ be a map in $\cS_{\cR}$, 
and let $F_{\bld{m}}^{\SpsymR}(W,P)\to \bS^{\cI}_b[\const_{\cI}\!K]$ be the resulting object in $\SpsymR/\bS^{\cI}_b[\const_{\cI}\!K] $. Then there is a natural isomorphism 
\begin{multline}\label{eq:base-change-of-free}
\left((\const_{\cI}f)^* F_{\bld{m}}^{\SpsymR}(W,P)\right)(\bld{n}) = (\const_{\cI}f)^* \left( \textstyle \coprod_{\alpha \in \cI(\bld{m},\bld{n})} (W,P) \barsm S^{\bld{n}-\alpha}\right)\\ \xrightarrow{\iso}  \textstyle \coprod_{\alpha \in \cI(\bld{m},\bld{n})} f^* (W,P)\barsm S^{\bld{n}-\alpha}  =  
F_{\bld{m}}^{\SpsymR}(f^*(W,P)) (\bld{n}) 
\end{multline}
where the coproducts are taken in $\cS_{\cR}$ and the base change on the right
hand side is formed along $(W,P) \to \iota_b(K)$. Since the
cofibrations and generating acyclic cofibrations of the absolute level
model structure on $\SpsymR/\bS^{\cI}_b[\const_{\cI}\!K]$ are obtained from those of
$\SpsymR$ by allowing all possible augmentations~\cite{Hirschhorn_over_under}, the claim for the
level model structure follows from the isomorphism~\eqref{eq:base-change-of-free} and
Lemma~\ref{lem:f-upper-star-left-Quillen-SR}. Since we assume $K$ to be fibrant, $\bS^{\cI}_b[\const_{\cI}\!K]$ is fibrant in $\SpsymR$ so that we can use \cite[Proposition~3.4]{schulz-logarithmic} to deduce that the local model structure on $ \SpsymR/\bS^{\cI}_b[\const_{\cI}\!K] $ can be viewed as the left Bousfield localization at a set of maps whose domains and codomains are of the form $F_{\bld{m}}^{\SpsymR}(W,P)\to \bS^{\cI}_b[\const_{\cI}\!K]$. So~\eqref{eq:base-change-of-free} implies  that $(\const_{\cI}f)^*$ is also left Quillen with respect to the local model structure. 
\end{proof}
\begin{remark}
  The preceding lemma does not hold in general if we consider the base change
  along $\bS^{\cI}_b[g]$ for an arbitrary map of $\cI$-spaces $g$ since
  in this case, the different levels of
  $g^* F_{\bld{m}}^{\SpsymR}(W,P)$ are coproducts over
  $g(\bld{m})^*(W,P)$ which may vary in $\bld{m}$.
\end{remark}

Recall that if $K$ is a space, $\Spsym{K} = \Spsym{\const_{\cI}K}$ is the stabilization of $\cS_K$. 
\begin{lemma}\label{lem:const_I_f_left-Quillen-SpsymK}
  The functor $(\const_{\cI}f)^*\colon \Spsym{K} \to \Spsym{L}$ is
  left Quillen if $f \colon L \to K$ is a fibration of Kan complexes.
\end{lemma}
\begin{proof}
The functor $(\const_{\cI}f)^*$ is a left adjoint since we are working with simplicial sets. 
The homotopical statement follows from Lemma~\ref{lem:f-upper-star-left-Quillen-over-constant}
and Corollary~\ref{cor:local-model-structure-inherited} (or by adapting the argument in Lemma~\ref{lem:f-upper-star-left-Quillen-over-constant} to $\Spsym{K}$ and $\Spsym{L}$). 
\end{proof}
We have now developed enough tools to verify the statements about $f^*\colon \Spsym{X} \to \Spsym{Y}$ and its monoidal behavior made in the beginning of this section. 
\begin{proof}[Proof of Proposition~\ref{prop:Rf-upper-star-has-right-adjoint}]
  Using Proposition~\ref{prop:simp-top-spsymX}, it suffices to verify
  the claim in the simplicial case.  Since $X$ is $\cI$-equivalent to
  $X' = \const_{\cI}(X^{\mathrm{fib}}(\bld{0}))$, the Quillen
  equivalences relating $\Spsym{X}$ to $\Spsym{X'}$ allow us to assume
  that $f$ is of the form $\const_{\cI}g$ for a map of Kan complexes
  $g\colon L \to K$.  We factor $g$ as an acyclic cofibration
  $k \colon L \to P$ followed by a fibration $h \colon P \to X$. Then
  $(\const_{\cI}h)^*$ is left Quillen by
  Lemma~\ref{lem:f-upper-star-left-Quillen-over-constant}, and
  applying $(\const_{\cI}h)^*$ to objects that are both cofibrant and
  fibrant shows that
  $\mathbb L (\const_{\cI}h)^* = \mathbb R (\const_{\cI}h)^*$.  Since
  $(\const_{\cI}k)^*$ participates in a Quillen equivalence by
  Lemma~\ref{lem:base-change-local-Q-adjunction},
  $\mathbb R (\const_{\cI}k)^*$ is an equivalence of categories.
  Hence
  \[\mathbb R (\const_{\cI}g)^* = \mathbb R(\const_{\cI}k)^* \circ \mathbb R (\const_{\cI}h)^* =  \mathbb R(\const_{\cI}k)^* \circ \mathbb L (\const_{\cI}h)^*\] is a left adjoint.
\end{proof}

Given maps of cofibrant $\cI$-spaces $f\colon X' \to X$ and $g\colon Y' \to Y$ as well as cofibrant and fibrant objects $(E,X)$ in $\Spsym{X}$ and $(F,Y)$ in $\Spsym{Y}$, there is a chain of maps 
\begin{multline}\label{eq:upper-star-weak-monoidal} 
f^*(E,X)^{\mathrm{cof}} \barsm g^*(F,Y)^{\mathrm{cof}} \to f^*(E,X) \barsm g^*(F,Y)  \\ \to (f\boxtimes g)^*((E,X) \barsm (F,Y)) \to (f\boxtimes g)^*((E,X) \barsm (F,Y))^{\mathrm{fib}}
\end{multline}
induced by cofibrant replacements in $\Spsym{X'}$ and $\Spsym{Y'}$, the map~\eqref{eq:upper-star-lax-monoidal},  and  a fibrant replacement in $\Spsym{X\boxtimes Y}$.

\begin{proof}[Proof of Proposition~\ref{prop:monoidal_pull}]
For the statement of the proposition, it is sufficient to show that the map~\eqref{eq:upper-star-weak-monoidal} is a local equivalence. Arguing in $\SpsymR$ with arguments analogous to those in the proof of Proposition~\ref{prop:Rf-upper-star-has-right-adjoint}, we may assume that $f = \const_{\cI}\tilde f$ and $g = \const_{\cI} \tilde g$ where $\tilde f \colon K' \to K$ and $\tilde g \colon L' \to L$ are Kan fibrations of Kan complexes. Then $f\boxtimes g \iso \const_{\cI}(\tilde f \times \tilde g)$, and it follows from Lemma~\ref{lem:const_I_f_left-Quillen-SpsymK} that the first and the last map in~\eqref{eq:upper-star-weak-monoidal} are local equivalences. Working over a constant base, Lemma~\ref{lem:fiberwise-smash-pullback} and the fact that base change preserves pullbacks show that the map~\eqref{eq:upper-star-lax-monoidal} is even an isomorphism.  
\end{proof}

We conclude with an explicit description of the topological version of
$\mathbb R\Gamma = \mathbb R (Y\to *)_*$ and its monoidal structure in
an important special case.  This will become relevant
in~\cite{HS-twisted}.

By \cite[Proposition 1.5]{Lewis_fibre-spaces} the functor
$f^*\colon \Spsym{\const_{\cI}K} \rightarrow \Spsym{\const_{\cI}L}$
admits a right adjoint if and only if the map
$f\colon L \rightarrow K$ is open. This is certainly the case for the
map $r \colon K \rightarrow *$ and the adjoint $\Gamma$ is given by
sending a parametrized spectrum $(E,\const_{\cI}K)$ to the spectrum
whose $m$th level is given by the section space of the projection
$E(\bld m) \rightarrow K$.

\begin{lemma}\label{gamma-top}
Suppose that $K$ is a cell complex or more generally cofibrant. Then $\Gamma \colon \Spsym{\const_{\cI}K} \rightarrow \Spsym{}$ preserves (positive) level equivalences between parametrized spectra whose projections are Serre-fibrations. In particular, $\Gamma$ preserves local equivalences between (positive) locally fibrant spectra and carries these to (positive) fibrant spectra.
\end{lemma}

\begin{proof}
This is immediate from the fiber sequence $\Gamma(B,E) \rightarrow E^B \rightarrow B^B$, whenever $E \rightarrow B$ is a Serre-fibration and $B$ is a cell complex, and the fact that $(-)^B$ preserves weak-equivalences. The last claim follows since $\Gamma$ commutes with taking (fiberwise) loops by adjunction.
\end{proof}

Thus $\mathbb R\Gamma(E,\const_{\cI}K)$ is represented by $\Gamma(E^\mathrm{fib},\const_{\cI}K)$, the value of $\Gamma$ on a locally fibrant replacement of $E$. Furthermore, by construction, the map 
\[\mathbb R \Gamma (E,\const_{\cI}K) \sm_{\mathbb L} \mathbb R \Gamma (F,\const_{\cI}L) \to  \mathbb R \Gamma (E\barsm_{\mathbb L} F,\const_{\cI}K \times L) \]
is represented by the natural map
\begin{equation}\label{eq:Gamma-monoidal}\Gamma (E,\const_{\cI}K) \sm \Gamma (F,\const_{\cI}L) \longrightarrow \Gamma (E\barsm F,\const_{\cI}K \times L)
\end{equation}
taking products of sections, whenever $(E,\const_{\cI}K)$ and $(F,\const_{\cI}L)$ are bifibrant.

A similar description still applies when we are presented with a (positive) level equivalence $f \colon X \rightarrow \const_{\cI}K$. For $(E,X) \in \Spsym{X}$ we then find
\begin{equation}\label{eq:RGamma-topological} \mathbb R\Gamma (E,X) \iso \mathbb R\Gamma \mathbb Rf_* (E,X) \iso \Gamma (\mathbb Lf_! (E,X))^{\mathrm{fib}},
\end{equation}
an observation which is made use of in the comparison of operator algebraic and homotopical twisted $K$-theory in~\cite[Proposition~6.2]{HS-twisted}. 

\section{The universal line bundle}\label{sec:univ-bundle}
In this section, we construct an important example of a commutative para\-metrized ring spectrum, namely the universal line bundle $\gamma_R$ associated with a commutative symmetric ring spectrum $R$. We are interested in $\gamma_R$ for several reasons. In Section~\ref{subsec:infty-thom}, we show that it represents its $\infty$-categorical counterpart studied in \cites{Ando-B-G-H-R_infinity-Thom, Ando-B-G_parametrized}. This leads to a multiplicative comparison of the parametrized (co)homology groups from Section~\ref{sec:tw-coho} with the $\infty$-categorical ones from~\cite{Ando-B-G_parametrized}. The universal line bundle also allows us to relate multiplicative point set level Thom spectrum functors to $\infty$-categorical ones (see Theorem~\ref{thm:Thom-spectrum-comparison}). Lastly, it also plays a prominent role in the multiplicative comparison of twisted $K$-theory spectra in~\cite{HS-twisted}.  

\subsection{The construction of the universal line bundle}\label{unibund}
In the following, we use the notion of commutative $\cI$-space monoids from Definition~\ref{def:commutative-I-space-monoid}, the positive $\cI$-model structure on the resulting category of commutative $\cI$-space monoids $\cC\cS^{\cI}$~\cite[\S 3]{Sagave-S_diagram}, and the adjunction $\bS^{\cI}\colon \cC\cS^{\cI}\rightleftarrows \cC\Spsym{}\colon \Omega^{\cI}$ relating them to commutative symmetric ring spectra~\cite[(3.9)]{Sagave-S_diagram}. Moreover, we say that a commutative $\cI$-space monoid $M$ is \emph{grouplike} if the monoid $\pi_0(M_{h\cI})$ is a group~\cite[\S 3.17]{Sagave-S_diagram}.

Let $R$ be a positive fibrant commutative symmetric ring spectrum. Its multiplicative $E_{\infty}$ space is modeled by the commutative $\cI$-space monoid $\Omega^{\cI}(R)$, and its units $\GLoneIof{R}$ are given by the sub commutative $\cI$-space monoid of invertible path components of $\Omega^{\cI}(R)$. The fibrancy condition on $R$ is needed to ensure that $\Omega^{\cI}(R) $ and $\GLoneIof{R}$ capture a well-defined homotopy type. It can be enforced by applying a fibrant replacement to $R$ (and could be relaxed to only asking $R$ to be positive level fibrant and semistable \cite[Remark 2.6]{Basu_SS_Thom}). 

We let $G = (\GLoneIof{R})^{\mathrm{cof}}$ be a cofibrant replacement in the positive $\cI$-model structure on $\cC\cS^{\cI}$. The adjoint of $G \to \GLoneIof{R} \to \Omega^{\cI}(R)$ is a map of
commutative symmetric ring spectra $\bS^{\cI}[G] \to R$. Via the strong symmetric monoidal functor $\bS^{\cI}_t$ from
Lemma~\ref{lem:various-com-adjunctions}, $G$ also gives rise to a
commutative monoid $\bS^{\cI}_t[G] = \bS^{\cI}_{\cR}[\iota_t(G)] =
\bS^{\cI}_{\cR}[G\amalg G,G]$ in $\SpsymR$ whose base commutative
$\cI$-space monoid is $G$. The unique map $G \to \ucI$ induces a
commutative monoid map $(G\amalg G,G) \to (\ucI\amalg G,\ucI)$
in $\cS^{\cI}_{\cR}$, and the composite
\[ \bS^{\cI}_t[G] \to \bS^{\cI}_{\cR}[\ucI\amalg G,\ucI]
\xrightarrow{\iso} \bS^{\cI}[G] \to R\]
allows us to view $R$ as a commutative $ \bS^{\cI}_t[G]$-algebra in
$\SpsymR$. We may also view $\bS = \bS^{\cI}_{t}[\ucI]$ as a
commutative $ \bS^{\cI}_t[G]$-algebra via the map induced by
$G \to \ucI$. Altogether, this allows us to form the two-sided bar construction
\[ B^{\barsm}(\bS, \bS^{\cI}_t[G],R) = | [q] \mapsto \bS \barsm
\bS^{\cI}_t[G]^{\barsm q} \barsm R|.\]
Being the realization of a simplicial object in $\cC\SpsymR$, it is
itself a commutative parametrized ring spectrum. Its underlying
commutative $\cI$-space monoid is $BG = B(\ucI,G,\ucI)$, the bar
construction of $G$ with respect to $\boxtimes$. As explained
in~\cite[\S 2.9]{Basu_SS_Thom}, $BG$ classifies $G$-modules. Its underlying $E_{\infty}$ space $(BG)_{h\cI} \simeq B(G_{h\cI})$ models the usual classifying space $ B \GLoneof{R}$ of the units of $R$. 

\begin{definition}\label{def:gamma_R}
  Let $R$ be a positive fibrant commutative symmetric ring spectrum in
  simplicial sets. Its \emph{universal line bundle} is defined to be
  $\gamma_R = \!(B^{\barsm}(\bS, \bS^{\cI}_t[G],R))^{\mathrm{fib}}$, a
  fibrant replacement of $B^{\barsm}(\bS, \bS^{\cI}_t[G],R)$ in the
  positive local model structure on $\cC\Spsym{BG}$.
\end{definition}
It follows from Lemmas~\ref{lem:realization} and~\ref{lem:product-with-cofibrant-in-SpsymR} below and the fact that $G$ is flat as an $\cI$-space~\cite[Proposition 3.15(i)]{Sagave-S_diagram} that a stable equivalence $R \to R'$ of positive fibrant commutative symmetric ring spectra induces a local equivalence $\gamma_R \to \gamma_R'$. 
\begin{remark}\label{rem:associative-gamma-R}The above construction can also be carried out for not-necessarily commutative ring spectra $R$, by using associative cofibrant and fibrant replacements instead of commutative ones. In this case, $\gamma_R$ is only an $R$-module and no longer a parametrized ring spectrum. The constructions from Section \ref{sec: coho} still produce twisted $R$-(co)homology functors, but these are no longer equipped with products.
\end{remark}

\begin{remark}
  For a positive fibrant commutative symmetric ring spectrum in
  topological spaces, we cannot directly implement
  Definition~\ref{def:gamma_R} because we have not established the
  topological version of the model structure on $\cC\Spsym{BG}$ (and
  the topological counterparts of Lemmas~\ref{lem:realization}
  and~\ref{lem:product-with-cofibrant-in-SpsymR} below). Rather than
  going through this, we content ourselves with the following
  construction: Given a positive fibrant commutative symmetric ring
  spectrum $R$ in topological spaces, we apply the above construction
  to $\Sing R$ and define $\gamma_R$ to be $| \gamma_{\Sing R}|$. Then
  the realization of the simplicial $BG$ models the topological one by
  the discussion in Section~\ref{subsec:Thom-via-classifying} below,
  and $| \gamma_{\Sing R}|$ is locally fibrant by
  Remark~\ref{rem:top-problems-commutative-parametrized-model}.
\end{remark}

We again work over simplicial sets and let $E$ be an $R$-module spectrum. Then  we can view $E$ as an
$\bS^{\cI}_t[G]$-module by restriction along $\bS^{\cI}_t[G] \to R$
and generalize $\gamma_R$ by considering
$\gamma_E = (B^{\barsm}(\bS, \bS^{\cI}_t[G],E))^{\mathrm{fib}}$. Here
the fibrant replacement is taken in a lifted model structure on
$B^{\barsm}(\bS, \bS^{\cI}_t[G],R)$-module spectra that exists
by~\cite[Proposition 3.4.2]{Pavlov-S_symmetric-operads}. Based on
this notion, we now describe the behavior of universal bundles under pullback. On the one hand this is crucial for the applications in~\cite{HS-twisted}, and on the other it shows that the fiber of $\gamma_E$ over the basepoint of $BG$ is just $E$ itself, as should be expected. 

\begin{proposition}\label{prop:Thom-of-BH-BG} We work in simplicial sets and let $H \to G$ be a map of $\cI$-space monoids with $H$ flat and
  grouplike. Then the canonical map
\[ B^{\barsm}(\bS, \bS^{\cI}_t[H],E) \to (BH \to BG)^*(\gamma_E) \] 
is a local equivalence of parametrized spectra.  When
$E = R$ and $H$ is commutative, it is a local equivalence of commutative $BH$-relative
parametrized ring spectra.
\end{proposition}
The proof requires some preparation and will be given at the end of this section.

\subsection{Homotopy invariance properties} We now establish a series of lemmas needed for the homotopy invariance of $\gamma_R$, the proof of Proposition~\ref{prop:Thom-of-BH-BG}, and the next section. For this we work again only over $\cI$-spaces and (parametrized) symmetric spectra of simplicial sets.

The realization of simplicial objects in $\SpsymR$ can 
be defined by diagonalizing along the two simplicial directions and immediately lifts to a realization
functor $\mathrm{Fun}(\Delta^{op},\cC\SpsymR) \to  \cC\SpsymR$. 
\begin{lemma}\label{lem:realization}
Let $\varphi\colon (E,X)_{\bullet} \to (F,Y)_{\bullet}$ be a natural transformation between simplicial objects
in $\SpsymR$ with each $(E,X)_{q} \to (F,Y)_{q}$ a local equivalence. Then the realization of $\varphi$
is a local equivalence. 
\end{lemma}
\begin{proof}
  We consider the Reedy model structure on  $\mathrm{Fun}(\Delta^{op},\SpsymR)$ induced by the absolute local
  model structure. The realization of a Reedy cofibrant replacement is a level equivalence by applying the
  realization lemma for simplicial sets. The claim follows because realization
  preserves weak equivalences between Reedy cofibrant objects.
\end{proof}

We say that an ordinary symmetric spectrum $E$ is \emph{flat} if it is cofibrant in the flat (or $S$-) model structure on symmetric spectra (see~\cite{Shipley_convenient} and~\cite{Schwede_SymSp}). This notion is useful because $E \sm -$ preserves stable equivalences if $E$ is flat and the underlying symmetric spectra of cofibrant objects in the positive stable model structure on $\cC\Spsym{}$ are flat. Analogously, there is the notion of a flat $\cI$-space such that $X \boxtimes -$ preserves $\cI$-equivalences if $X$ is flat and underlying $\cI$-spaces of cofibrant commutative $\cI$-space monoids are flat~\cite[\S 3.8]{Sagave-S_diagram}.

\begin{lemma}\label{lem:product-with-cofibrant-in-SpsymR} 
Let $(E,X)$ be cofibrant in $\SpsymR$, let $F$ be a flat symmetric spectrum,
and let $Z$ be a flat $\cI$-space. Then  $(E,X)\barsm - $, $F\barsm -$, and $\bS^{\cI}_t[Z]\barsm - $
preserve local equivalences as functors $\SpsymR \to \SpsymR $.
\end{lemma}
\begin{proof}
  By~\cite[Propositions 2.3.10 and 3.3.6]{Pavlov-S_symmetric-operads}
  the category $\SpsymR$ also has a \emph{flat} absolute local model
  structure with more cofibrations and with weak equivalences the
  local equivalences. We call the cofibrant objects in this model
  structure \emph{flat} and notice that the $\barsm$-product with flat
  objects preserves local equivalences by the flatness statement
  subsumed in~\cite[Proposition
  3.4.2]{Pavlov-S_symmetric-operads}. Hence $(E,X) \barsm -$ preserves
  local equivalences. One can check on the generating cofibrations
  that both $\bS^{\cI}_t\colon \cS^{\cI} \to \SpsymR$ and the
  inclusion functor $\Spsym{} \to \SpsymR$ preserve the cofibrations
  of the flat model structures and thus flat objects.
\end{proof}
\begin{corollary}\label{cor:I-equivalence-between-flat-to-local-equiv}
  If $(E,X)$ is an object in $\SpsymR$ and $f\colon Y \to Y'$ is an
  $\cI$-equivalence between flat $\cI$-spaces, then
  $\bS^{\cI}_t[f]\barsm (E,X)$ is a local equivalence.
\end{corollary}
\begin{proof}
Taking a cofibrant replacement of $(E,X)$, this follows from the previous lemma by two out of three for local equivalences.  
\end{proof}

Our next aim is to obtain homotopy invariance results for restriction functors beyond what can be deduced directly from Lemma~\ref{lem:f-upper-star-left-Quillen-over-constant}.
\begin{lemma}\label{lem:f-upper-star-preserves-we}
  If $f\colon Y \to X$ is a fibration between fibrant objects in the
  absolute $\cI$-model structure on $\cS^{\cI}$, then
  $f^*\colon \SpsymR/\bS^{\cI}_b[X] \to \SpsymR/\bS^{\cI}_b[Y]$ preserves local equivalences.
\end{lemma}
\begin{proof}
Since absolute $\cI$-fibrant $\cI$-spaces are naturally level equivalent to constant $\cI$-spaces, 
we may assume that $f$ is a fibration of fibrant and constant $\cI$-spaces. Then $f^*$ is left Quillen
by Lemma~\ref{lem:f-upper-star-left-Quillen-over-constant} and right Quillen by general model category theory. Hence $f^*$ preserves weak equivalences.
\end{proof}

We now consider the following commutative diagram in $\SpsymR$ where the right hand
horizontal maps are the identity on the base: 
\[\xymatrix@-1pc{ \bS^{\cI}_b[Y] \ar[rr]^{\bS^{\cI}_b[f]} \ar[d]_{\bS^{\cI}_{b}[p]}&& \bS^{\cI}_b[X]  \ar[d]_{\bS^{\cI}_{b}[q]} && (E,X)\ar[ll] \ar[d]\\
 \bS^{\cI}_b[Y']  \ar[rr]^{\bS^{\cI}_b[g]} && \bS^{\cI}_b[X'] && (E',X')\ar[ll]
}\]

The next proposition uses the description of $f^*$ from~\eqref{eq:base-change-cobase-change} and essentially states that the local model structure satisfies a weak form of right properness where the fibrations are only allowed to be in the image of $\bS^{\cI}_b$.  Its proof is based on Bousfield's observation that it is sufficient to check right properness of model categories on fibrations between fibrant objects~\cite[Lemma~9.4]{Bousfield_telescopic}. 
\begin{proposition}\label{prop:weak-cogluing}
If $f$ and $g$ are absolute $\cI$-fibrations, $p$ and $q$ are $\cI$-equivalences, and $(E,X) \to (E',X')$ is a local equivalence, then the induced map of pullbacks  $f^*(E,X) \to g^*(E',X')$ is a local equivalence. The same statement holds when working over topological spaces. 
\end{proposition}
Setting $p=\id$ and $q=\id$ in the proposition implies that the statements of Lemmas~\ref{lem:f-upper-star-left-Quillen-over-constant} and~\ref{lem:f-upper-star-preserves-we} hold without the fibrancy conditions on the objects. 
\begin{proof}
Since the topological statement follows from the simplicial one by applying the singular complex, it suffices to verify the latter. By choosing a replacement of $f$ by an $\cI$-fibration between $\cI$-fibrant objects $\tilde f \colon \tilde Y \to \tilde X$, Lemma~\ref{lem:iota-b-left-right-adjoint-I} provides  the left hand square in the following commutative diagram: 

\[\xymatrix@-1pc{
\bS^{\cI}_b[Y] \ar@{->>}[rrr]^{\bS^{\cI}_b[f]}  \ar[d]_{\bS^{\cI}_b[j]}^{\sim}&&& \bS^{\cI}_b[X]  \ar[d]_{\bS^{\cI}_b[i]}^{\sim}&&& (E,X) \ar[lll]_{u} \ar[d]_{k}^{\sim}\\
\bS^{\cI}_b[\tilde Y] \ar@{->>}[rrr]^{\bS^{\cI}_b[\tilde f]} &&& \bS^{\cI}_b[\tilde X]  &&& (\hat E,\hat X) \ar@{->>}[lll]_{v}
}\]
The right hand square is obtained by factoring $(E,X) \to  \bS^{\cI}_b[\tilde X]$ as a local equivalence followed by a fibration. We get the following sequence of maps where $(-)^*$ denotes the base change along the respective map in $\SpsymR$: 
\[ \bS^{\cI}_b[f]^*(E,X) \iso u^*(\bS^{\cI}_b[Y]) \xrightarrow{\sim} u^*(\bS^{\cI}_b[i]^*(\bS^{\cI}_b[\tilde Y])) \xrightarrow{\iso} \bS^{\cI}_b[\tilde f]^*(E,X) \xrightarrow{\sim} (\hat E,\hat X) \]
Here the first map is a local equivalence since $u^*$ is right Quillen when viewed as a  functor $\Spsym{R}/\bS^{\cI}_b[X] \to \Spsym{R}/(E,X)$ and $\bS^{\cI}_b[Y] \to \bS^{\cI}_b[i]^*(\bS^{\cI}_b[\tilde Y])$ is a weak equivalence between fibrant objects in $\Spsym{R}/\bS^{\cI}_b[X]$ because the $\cI$-model structure on $\cS^{\cI}$ is right proper. The last map is a local equivalence by Lemma~\ref{lem:f-upper-star-preserves-we}. 
Hence we have shown that the pullback of the top row is locally equivalent to the pullback of the bottom row, 
and the latter is homotopy invariant since both maps are fibrations with fibrant codomain. Since this construction can be arranged to be natural with respect to $(E,X) \to (E',X')$ and $Y \to Y'$, the claim follows. 
\end{proof}

\begin{lemma}\label{lem:bSIR-full-hty-inv}
Let $f \colon (Z,Y) \to (Z',Y')$ be a map in $\cS^{\cI}_{\cR}$ such that both
$Z\to Z'$ and $Y \to Y'$ are absolute level (resp.\ $\cI$-) equivalences in $\cS^{\cI}$. Then
$\bS^{\cI}_{\cR}[f]$ is an  absolute level (resp.\ local) equivalence in $\SpsymR$. 
\end{lemma}
\begin{proof}
  Since $-\barsm S^m \colon \cS_{\cR} \to \cS_{\cR}$ preserves weak
  equivalences as we work over $\sset$, $\bS^{\cI}_{\cR}$ preserves level equivalences. Since
  $\bS^{\cI}_{\cR}$ is left Quillen with respect to the local model
  structures, arguing with a cofibrant replacement shows the second
  claim.s
\end{proof}

\begin{lemma}\label{lem:commuting-f-upper-star-barsm}
Let $f\colon Y \to X$ be an absolute $\cI$-fibration in $\cS^{\cI}$, let $Z \to X$ be a map of $\cI$-spaces, and let $E$ be a flat symmetric spectrum. Then the canonical map 
\[ \bS^{\cI}_{\cR}[Y \amalg f^*(Z),Y] \barsm E \to f^*(\bS^{\cI}_{\cR}[X \amalg Z,X] \barsm E) \]
is a local equivalence. 
\end{lemma}
\begin{proof}
Arguing with the absolute $\cI$-model structure on $\cS^{\cI}$ and the Quillen equivalence $\colim_{\cI}\colon \cS^{\cI}\rightleftarrows \cS \colon \const_{\cI}$, we can construct a commutative diagram 
\[\xymatrix@-1pc{
Y \ar@{->>}[d] && Y^{c} \ar@{->>}[d] \ar[ll]_-{\sim} \ar[rr]^-{\sim} && \const_{\cI}(L) \ar@{->>}[d] \\
X  && X^{c}  \ar[ll]_-{\sim} \ar[rr]^-{\sim} && \const_{\cI}(K)  \\
Z \ar[u] && Z^{c} \ar[u] \ar[ll]_-{\sim} \ar[rr]^-{\sim} && \const_{\cI}(P) \ar[u] 
}\]
with the $\cI$-fibrations and $\cI$-equivalences as indicated. Arguing with this diagram, Lemma~\ref{lem:product-with-cofibrant-in-SpsymR}, Lemma~\ref{lem:bSIR-full-hty-inv}, and Proposition~\ref{prop:weak-cogluing} reduce the claim to the case where all $\cI$-spaces are constant. In this situation, the map in question is an isomorphism by Lemma~\ref{lem:fiberwise-smash-pullback} and the explicit description of $\bS^{\cI}_{\cR}$ in Construction~\ref{const:various-adjunctions}.
\end{proof}

\begin{proof}[Proof of Proposition~\ref{prop:Thom-of-BH-BG}]
  By Lemmas~\ref{lem:realization}
  and~\ref{lem:product-with-cofibrant-in-SpsymR}, both sides send
  cofibrant replacements of $R$ and $E$ to local equivalences. Thus we
  may assume $E$ to be a cofibrant module over a cofibrant commutative
  ring spectrum and therefore to be flat as a symmetric spectrum. Next
  we choose a factorization of $f\colon BH \to BG$ into an acyclic
  cofibration $g\colon BH \to Y$ followed by a fibration
  $h \colon Y \to BG$ in the absolute $\cI$-model structure and
  consider the following diagram explained below:
\[\xymatrix@-1pc{
    B^{\barsm}(\bS^{\cI}_{\cR}[BH \amalg B(\ucI,H,G), BH], \bS^{\cI}[G], E) \ar[d]_-{\sim} \ar[rr]^-{\sim} && B^{\barsm}(\bS, \bS^{\cI}_t[H],E) \ar[dd]\\
    g^*((B^{\barsm}(\bS^{\cI}_{\cR}[Y \amalg h^*B(\ucI,G,G), Y], \bS^{\cI}[G], E))^{\mathrm{fib}}) \ar[d]_-{\sim}&& \\
    f^* ((B^{\barsm}(\bS^{\cI}_{\cR}[BG \amalg B(\ucI,G,G),BG],
    \bS^{\cI}[G], E))^{\mathrm{fib}}) \ar[rr]^-{\sim} && f^*(\gamma_E)
  }\] The top horizontal map arises by identifying
$\bS^{\cI}_{\cR}[BH \amalg B(\ucI,H,G)]$ with the bar construction
$B^{\barsm}(\bS, \bS^{\cI}_t[H], \bS^{\cI}[G])$, commuting bar
constructions, and using the map induced by the canonical stable
equivalence $B(\bS^{\cI}[G],\bS^{\cI}[G],E) \to E$.  The resulting map
is a local equivalence since $B^{\barsm}(\bS,\bS^{\cI}_t[H],-)$
preserves local equivalences by Lemmas~\ref{lem:realization}
and~\ref{lem:product-with-cofibrant-in-SpsymR}. The lower horizontal
map arises in the same way by setting $H = G$ and taking fibrant
replacements and base change along the right Quillen functor $f^*$ in
addition. The map
\[ B^{\barsm}(\bS^{\cI}_{\cR}[Y \amalg h^*B(\ucI,G,G), Y], \bS^{\cI}[G], E) \to B^{\barsm}(\bS^{\cI}_{\cR}[BG \amalg B(\ucI,G,G), BG], \bS^{\cI}[G], E)\]
is a local equivalence by Lemmas~\ref{lem:realization} and~\ref{lem:commuting-f-upper-star-barsm}. Since
$h^*(\gamma_E)$ is fibrant, we can extend the resulting local equivalence to $h^*(\gamma_E)$ over a fibrant replacement of the domain and apply $g^*$ to get the lower left hand vertical local equivalence. The upper left hand vertical equivalence arises from the fact that $  B(\ucI,H,G) \to h^*B(\ucI,G,G)$ is an $\cI$-equivalence since $G$ and $H$ are grouplike~\cite[Proof of Proposition 3.15]{Basu_SS_Thom}, the homotopy invariance of $\bS^{\cI}_{\cR}$ established in Lemma~\ref{lem:bSIR-full-hty-inv} and that of $B(-, \bS^{\cI}[G], E)$ resulting from Lemmas~\ref{lem:realization} and~\ref{lem:product-with-cofibrant-in-SpsymR}, and from Corollary~\ref{cor:characterization-local-equivalences}. It follows that the right hand vertical map is a local equivalence. 
\end{proof}

\section{Point-set level Thom spectrum functors}\label{sec:thom-comparison}
We now explain how our approach to parametrized spectra gives rise to a multiplicative $R$-module Thom spectrum functor. As an application, we compare it to various other  approaches to generalized Thom spectra. 
\subsection{Generalized Thom spectra via universal bundles}\label{subsec:T_R} Let $R$ be a commutative ring spectrum in simplicial sets that is positive fibrant or, more generally, level fibrant and semistable (cf. \cite[Remark 2.6]{Basu_SS_Thom}). We now write $(\Spsym{})_R$ for the category of (right) $R$-modules in $\Spsym{}$ and $(\SpsymR)_R$ for the category of (right) $R$-modules in $\SpsymR$. Via the composite $R \to B^{\barsm}(\bS, \bS^{\cI}_t[G],R) \to \gamma_R$, we can view the universal line bundle $\gamma_R$ as a commutative $R$-algebra, i.e., a commutative monoid with respect to the resulting product $\barsm_R$ in $(\SpsymR)_R$. We obtain a Thom spectrum functor 
\[ T_{\cR}^{\cI} \colon \cS^{\cI}/BG \to (\SpsymR)_R \to \Mod_R,\quad (f\colon X \to BG)
  \mapsto \Theta(f^*\gamma_R) = (X \to \ucI)_! (f^*\gamma_R).\]
This functor takes values in right $R$-modules since $\gamma_R$ is a right
$R$-module and both base change and the collapse of base space functor
$\Theta$ preserve right $R$-module structures as follows from the monoidality in Lemma~\ref{lem:cobase-change-spym-lax-monoidal} and \eqref{eq:upper-star-lax-monoidal}.
 Precomposing $T_{\cR}^{\cI}$ with the $\cI$-spacification $P_{BG}$ provides a space level Thom spectrum functor
\begin{equation}\label{eq:TIR} T_{\cR} \colon \cS/(BG)_{h\cI} \to \Mod_R, \quad(\tau \colon K \to (BG)_{h\cI}) \mapsto T^{\cI}_{\cR}(\tau_{\cI}) \end{equation}
that sends weak equivalences to stable equivalences and preserves actions of operads augmented over the Barratt--Eccles operad. Since $\gamma_R$ is fibrant and $\Theta$ coincides with its left derived functors, the homotopy groups $\pi_n(T_{\cR}(\tau))$ are just the parametrized homology groups associated with the universal line bundle $\gamma_R$ and the map $\tau$. 

\begin{proposition}
  The functor $T_{\cR}^{\cI}$ is lax symmetric monoidal and sends
  $\cI$-equiv\-alences over $BG$ to stable equivalences of
  $R$-modules. It preserves colimits, tensors with simplicial sets,
  and actions of operads in simplicial sets.
\end{proposition}
\begin{proof}
We get a natural map $f^*(\gamma_R) \barsm_R g^*(\gamma_R) \to (f\boxtimes g)^*(\gamma_R)$ since $\gamma_R$ is a commutative $R$-algebra. This exhibits $\cS^{\cI}/BG \to  (\SpsymR)_R,f\mapsto  f^*(\gamma_R)$ as a lax symmetric monoidal functor. Since $\Theta$ is strong symmetric monoidal by Lemma~\ref{lem:cobase-change-spym-lax-monoidal}, it follows that $T_{\cR}^{\cI}$ is lax symmetric monoidal. For the homotopy invariance, we note that an $\cI$-equivalence $g\colon Y \to X$ and a map $f\colon X \to BG$ give rise to a map 
$(fg)^*(\gamma_R) = g^* f^*(\gamma_R) \to f^*(\gamma_R)$ that is a local equivalence by the fibrancy assertion on $\gamma_R$ and Corollary~\ref{cor:characterization-local-equivalences}. The functor $\Theta$ maps this local equivalence to a stable equivalence $T_{\cR}^{\cI}(fg) \to T_{\cR}^{\cI}(f)$. Compatibility with the tensor and colimits follows since the individual functors have this property. 
\end{proof}
If $R \to R'$ is a stable equivalence between positive fibrant objects, we get a natural stable equivalence between the resulting Thom spectrum functors that is induced by the above local equivalence $\gamma_R \to \gamma_R'$. 

\subsection{Generalized Thom spectra via classifying spaces for \texorpdfstring{$G$}{G}-modules}\label{subsec:Thom-via-classifying} We begin by reviewing the Thom spectrum functor introduced in~\cite{Basu_SS_Thom}. In the latter paper the focus is on the topological setting, but the analogous construction works equally well in the simplicial setting, cf.\ \cite[Remark 3.7]{Basu_SS_Thom}. Thus, in the following discussion, the underlying category of spaces $\cS$ can be either $\sset$ or $\tp$.

Let $R$ be a positive fibrant and flat commutative symmetric ring spectrum, write $\GLoneIof{R}$ for its $\cI$-space
units, and let $G \to \GLoneIof{R}$ be a cofibrant replacement in
$\cC\cS^{\cI}$. We define $EG$ by choosing a factorization of the form
\begin{equation}\label{eq:def-EG-as-cof-rep} \xymatrix@1{B(\ucI,G,G) \ar@{ >->}[r]^-{\sim} & EG \ar@{->>}[r] &BG }
\end{equation}
in the positive model structure of $\cC\cS^{\cI}$. Now let
$U\colon \cS^{\cI}/BG \to \Mod_G$ be the functor to $G$-modules in
$\cS^{\cI}$ sending a map $f\colon X \to BG$ to the pullback $U(f)$
of the diagram $X \xrightarrow{f} BG \ot EG$ where both $X$ and $BG$
carry the trivial $G$-module structure. The fibrant replacement
in~\eqref{eq:def-EG-as-cof-rep} ensures that $U$ preserves
$\cI$-equivalences.

The $\cI$-space version of the Thom spectrum functor \cite[Definition 3.6]{Basu_SS_Thom} is the
composite
\begin{equation}\label{eq:def-Uf} T^{\cI}_{EG}\colon \cS^{\cI}/BG \xrightarrow{U} \Mod_G \xrightarrow{\bS^{\cI}} \Mod_{\bS^{\cI}[G]} \xrightarrow{B(-,\bS^{\cI}[G],R)} \Mod_R 
\end{equation}
where we use the subscript $EG$ to distinguish it from
$T^{\cI}_{\cR}$. Precomposing
$T^{\cI}_{EG}$ with the $\cI$-spacification~\eqref{eq:I-spacificaton}
defines a space level Thom spectrum functor
\[T_{EG}\colon \cS/{BG}_{h\cI} \to \Mod_R\] with favorable properties; see~\cite[\S 4.6]{Basu_SS_Thom}.
It is proved in \cite[Proposition~4.6]{Basu_SS_Thom} that $T_{EG}$ is homotopy invariant by our assumption that $R$ is flat.
\begin{remark} In~\cite{Basu_SS_Thom} the commutative $\cI$-space monoids $G$ and $EG$ were defined using the so-called 
\emph{flat} $\cI$-model structure on $\cC\cS^{\cI}$. For the definition of the Thom spectrum functors
 $T^{\cI}_{EG}$ and $T^{\cI}_{EG}$ we may equally well work with the \emph{projective}
  $\cI$-model structure used in the present paper since the latter model structure has fewer cofibrations.
\end{remark}

We now explain why the simplicial and topological versions are equivalent.
Firstly, geometric realization and singular complex induce Quillen equivalences between the simplicial and topological versions of commutative symmetric ring spectra and commutative $\cI$-space monoids. Up to isomorphism, geometric realization commutes with $\bS^{\cI}$ and thus $\Sing$ commutes with $\Omega^{\cI}$. Moreover, geometric realization preserves positive fibrant objects. When $R$ is a topological positive fibrant commutative symmetric ring spectrum, then $\Sing(\GLoneIof{R}) \iso \GLoneIof{\Sing(R)}$. If $R$ is a positive fibrant commutative symmetric ring spectrum in simplicial sets and $G \to \GLoneIof{R}$ is a cofibrant replacement of its units, then the adjoint of $\bS^{\cI}[|G|] \xrightarrow{\iso} | \bS^{\cI}[G] | \to |R|$ exhibits $|G|$ as a cofibrant replacement of $\GLoneIof{|R|}$ since its image under $\Sing$ participates as the upper left hand horizontal arrow in the commutative diagram
\[\xymatrix@-1pc{
    \Sing|G| \ar[rr] &&  \Sing \Omega^{\cI}(|R|) \ar[rr]^-{\iso}&&  \Omega^{\cI} \Sing|R| \\
    G \ar[u]^{\sim} \ar[rrrr] &&&& \Omega^{\cI}(R).\ar[u]_{\sim} 
  }
\]
Hence $|BG|$ models $B (\GLoneIof{|R|})^{\mathrm{cof}}$. Since realization also preserves positive $\cI$-fibrations, $|EG|$ models its topological counterpart for $|R|$. 
\begin{proposition}\label{prop:TEG-top-simp}
Let $R$ be a positive fibrant commutative symmetric ring spectrum in simplicial sets and let
$f \colon X \to BG$ be a map of $\cI$-spaces, also in simplicial sets. Defining the topological Thom spectrum functor $T_{|EG|}^{\cI}$ for $|R|$ using $|G|$ and $|EG|$ as explained above, there is a natural isomorphism $|T_{EG}^{\cI}(f)| \iso T_{|EG|}^{\cI}(|f|)$. It induces a monoidal natural stable equivalence $|T_{EG}(\tau)| \to T_{|EG|}(|\tau|)$ of space level Thom spectra preserving actions of operads augmented over the Barratt--Eccles operad. 
\end{proposition}
\begin{proof}
The statement for $T^{\cI}_{EG}$ follows since geometric realization preserves pullback and is strong symmetric monoidal both for $\cI$-spaces and symmetric spectra. The space level version results from the natural $\cI$-equivalence $|P_{BG}(\tau)| \to P_{|BG|}(|\tau|)$ induced by the adjunction $(|-|,\Sing)$. 
\end{proof}
Conversely, let $R$ be a topological positive fibrant commutative symmetric ring spectrum and $G \to \GLoneIof{(\Sing R)^{\mathrm{cof}}}$ be  a cofibrant replacement of the units of a cofibrant replacement of $\Sing R$ in commutative ring spectra. Given any map $f\colon X \to |BG|$, a homotopy pullback construction provides a map $f'\colon X' \to BG$ such that $|X'|$ and $X$ are weakly equivalent over $|BG|$ so that
$|T^{\cI}_{EG}(f')| \simeq T^{\cI}_{EG}(f) \simeq T^{\cI}(f)$ as modules over $|(\Sing R)^{\mathrm{cof}}| \simeq R$. This shows that the topological Thom spectrum functor can be expressed in terms of the simplicial one. 

\subsection{Comparing $R$-module Thom spectra}
Our next aim is to compare the simplicial version of $T^{\cI}_{EG}$ to the Thom spectrum functor
$T^{\cI}_{\cR}$ of Section~\ref{subsec:T_R}. For this we consider the following commutative diagram in $\Spsym{BG}$ explained below: 
\begin{equation}\label{ref:comparison-bundles}\xymatrix@-1.2pc{
 B^{\barsm}(\bS^{\cI}_\cR[BG\! \amalg\! B(\ucI,G,G), BG], \bS^{\cI}[G], R) \ar[d]_{\sim} \ar[rr]^-{\sim}  &&  B^{\barsm}(\bS^{\cI}_\cR[BG\! \amalg\! EG, BG], \bS^{\cI}[G], R) \ar[d]^v_{\sim}\\ 
 \bS^{\cI}_\cR[BG\! \amalg\! B(\ucI,G,G), BG] \barsm_{\bS^{\cI}[G]} R\ar[d]_{\iso} \ar@{ >->}[rr]^-{\sim} &&  \bS^{\cI}_\cR[BG\! \amalg\! EG, BG] \barsm_{\bS^{\cI}[G]} R \ar@{..>}[d]_{\sim}^u\\ 
B^{\barsm}(\bS, \bS^{\cI}_t[G],R) \ar[rr]^-{\sim} && \gamma_R.
}
\end{equation}
The two upper horizontal maps are induced by  $B(\ucI,G,G) \to EG$. The upper one is a local equivalence by Lemmas~\ref{lem:realization} and~\ref{lem:bSIR-full-hty-inv}.  To analyze the second, we again use the functor $\bS^{\cI}_{\mathrm{ar}} = \bS^{\cI}_{\cR}\circ \iota_{\mathrm{ar}}\colon  \mathrm{Ar}(\cS^{\cI}) \to \SpsymR$ from~\eqref{eq:SIar-OmegaIar-adjunction}. It induces a functor 
\[ (G \to *) \downarrow \mathrm{Ar}(\cC\cS^{\cI}) \to \bS^{\cI}[G]\downarrow \cC\SpsymR \]
from commutative $ (G \to *)$-algebras in $\mathrm{Ar}(\cS^{\cI})$ to commutative $\bS^{\cI}[G]$-algebras in $\SpsymR$. With respect to the injective model structure on $\mathrm{Ar}(\cC\cS^{\cI})$ (cf. Lemma~\ref{lem:SIar-OmegaIar-Quillen-adjunction}), it sends the acyclic cofibration $B(*,G,G) \to EG$ over $BG$ to an acyclic cofibration of commutative $\bS^{\cI}[G]$-algebras in $\SpsymR$. Extending the latter map along $\bS^{\cI}[G] \to R$ shows that the middle horizontal map in the diagram is an acyclic cofibration in $\cC\SpsymR$. The upper vertical maps are instances
of the natural map from the two sided bar construction to the
relative $\barsm$-product. The lower left hand isomorphism
results from commuting $\bS^{\cI}_{\cR}$ with the bar construction (compare the argument in the proof of Proposition~\ref{prop:Thom-of-BH-BG}). The left hand vertical composite can be identified with the map
\[ B^{\barsm}(\bS, \bS^{\cI}_t[G], B^{\barsm}( \bS^{\cI}[G],
  \bS^{\cI}[G], R)) \to B^{\barsm}(\bS, \bS^{\cI}_t[G], R)\] and is
thus a weak equivalence by Lemmas~\ref{lem:realization}
and~\ref{lem:product-with-cofibrant-in-SpsymR}. So $v$ is a local
equivalence by two out of three. The lower horizontal map is the
fibrant replacement defining~$\gamma_R$. Lastly, $u$ arises as a lift
in the positive local model structure on $\cC\SpsymR$.

Given a map of $\cI$-spaces $f\colon X \to BG$, we get a natural 
map 
\begin{multline}\label{eq:EG-R-comparison} B^{\barsm}(\bS^{\cI}_\cR[X \amalg f^*(EG), X], \bS^{\cI}[G], R)\\ \to  f^*  B^{\barsm}(\bS^{\cI}_\cR[BG \amalg EG, BG], \bS^{\cI}[G], R)\to f^*(\gamma_R)
\end{multline}
where the first map results from the universal property of $f^*$ and the second map is induced by the composite $uv$ in~\eqref{ref:comparison-bundles}.
\begin{lemma}\label{lem:local-equiv-for-T_I_EG-T_I_R}
The composite map in~\eqref{eq:EG-R-comparison} is a local equivalence in $\Spsym{X}$.
\end{lemma}
\begin{proof}
  We first suppose that $f\colon X \to BG$ is an absolute
  $\cI$-fibration. Then the first map in~\eqref{eq:EG-R-comparison} is
  a local equivalence by Lemmas~\ref{lem:realization}
  and~\ref{lem:commuting-f-upper-star-barsm} since both $\bS^{\cI}[G]$
  and $R$ are flat. The second map in~\eqref{eq:EG-R-comparison} is a
  local equivalence by Proposition~\ref{prop:weak-cogluing} and the
  above observation that $uv$ is a local equivalence.  Since
  $U(f) = f^*(EG)$ is a homotopy pullback, both sides send
  $\cI$-equivalences to local equivalences, and the map is a local
  equivalence for all $f$.\end{proof}
  
Applying the collapse of base space functor $\Theta$ to~\eqref{eq:EG-R-comparison} provides maps
\begin{equation}\label{eq:TEG-TR-maps}
T^{\cI}_{EG}(f) \to T^{\cI}_{\cR}(f) \quad \text{ and } \quad T_{EG}(\tau) \to T_{\cR}(\tau)
\end{equation} 
where the second is obtained from the first by precomposing with the $\cI$-spacification.
\begin{proposition}\label{prop:TEG-vs-TR} The first map $T^{\cI}_{EG} \to T^{\cI}_{\cR}$ in~\eqref{eq:TEG-TR-maps} is a natural lax symmetric monoidal stable equivalence of functors $\bS^{\cI}/BG \to \Mod_R$, and the second is a natural lax monoidal stable equivalence $T_{EG} \to T_{\cR}$ of functors $\cS/(BG)_{h\cI} \to \Mod_R$ that respects actions of operads augmented over the Barratt--Eccles operad. \end{proposition}
\begin{proof}
  Since the composite $uv$ in~\eqref{ref:comparison-bundles} is a map of parametrized commutative ring spectra, $T^{\cI}_{EG}(-) \to T^{\cI}_{\cR}(-)$ is a lax symmetric monoidal transformation. By the homotopy invariance of $\Theta$ and Lemma~\ref{lem:local-equiv-for-T_I_EG-T_I_R}, it is a stable equivalence. The second statement then follows from the properties of the $\cI$-spacification discussed in Section~\ref{subsec:I-spacification}.
\end{proof}

\subsection{Thom spectra over the sphere spectrum}\label{subsec:Thom-over-S} We now consider the case of the sphere spectrum $\mathbb S$, work over topological spaces, and write $F = \GLoneIof(\bS)$ for the units of~$\bS$. (Since $\bS$ is semistable,~\cite[Lemma 2.5]{Basu_SS_Thom} implies that we do not need to replace it fibrantly before forming $\GLoneIof(\bS)$ and applying our Thom spectrum functor constructions.)
In this case, $F(\bld{m}) \subseteq \Omega^{\cI}(\bS)(\bld{m}) = \Omega^m(S^m)$ is the space of self-homotopy equivalences of $S^m$ which is a monoid under composition. The multiplications of the $F(\bld{m})$ assemble to an associative and unital multiplication map $F \times F \to F$ in $\cS^{\cI}$. The canonical $F(\bld{m})$-action on the $S^{m}$ assemble to an action $F \times \mathbb S \to \mathbb S$ in $\SpsymR$ where $\times$ now denotes the action introduced in~\eqref{eq:SI-tensor-SpsymR}. This action and the multiplication of $F$ allow us to form the two-sided bar construction \[ B^{\times}(\ucI,F,\mathbb S) = \big|[q] \mapsto \ucI\times F^{\times q} \times \mathbb S \big|\] in $\SpsymR$. Its evaluation at $\bld{m}$ is the classifying space for sectioned fibrations with fiber equivalent to $S^m$ which was considered in \cite[Section~IX]{LMS} (see also \cite[\S 2]{Schlichtkrull_Thom-symmetric}). Writing $T_{BF}^{\cI} \colon \tp^{\cI}/BF \to \Spsym{}$ for the Thom spectrum functor introduced in \cite[Definition 3.3]{Schlichtkrull_Thom-symmetric}, we thus obtain a natural isomorphism $T_{BF}^{\cI}(f) \iso \Theta(f^* B^{\times}(\ucI,F,\mathbb S))$. It follows from \cite[Corollaries~4.13 and 6.9]{Schlichtkrull_Thom-symmetric} that the space level counterpart $T_{BF} = T^{\cI}_{BF}\circ P_{BF}$ is a monoidal homotopy functor on $\tp^{\cI}/BF_{h\cI}$ that respects actions by operads augmented over the Barratt-Eccles operad.

Our goal is to compare $T_{BF}$ to the Thom spectrum functor $T_{\cR}$ in~\eqref{eq:TIR} with $R=\Sing\bS$. We first
observe that the restriction of the multiplication $F \times F \to F$
of $F$ along the map $F \boxtimes F \to {F \times F}$ arising from
Construction~\ref{const:rhoZYEX} (or from~\cite[Section~2.24]{Sagave-S_group-compl}) provides the commutative $\cI$-space
monoid structure of $F$. A cofibrant replacement
$G \to F = \GLoneIof(\bS)$ in commutative $\cI$-space monoids and the
maps from Construction~\ref{const:rhoZYEX} induce comparison maps
$\bS^{\cI}_t[G]^{\barsm q} \barsm \bS \to F^{\times q} \times \bS$ for
every $q \geq 0$. Using~\eqref{eq:distributivity}, one can check that
these are compatible with the simplicial structure maps and induce a well-defined map of commutative parametrized ring spectra
\begin{equation}\label{eq:barsm-times-F-comparison}
B^{\barsm}(\bS,\bS^{\cI}_t[G],\bS) \to  B^{\times}(\ucI,F,\bS).
\end{equation}
We write $\rho  \colon BG \to B^{\times}(F)$ for the underlying map of commutative $\cI$-space monoids. 
\begin{proposition}\label{prop:barsm-times-F-comparison}
The map $B^{\barsm}(\bS,\bS^{\cI}_t[G],\bS) \to  B^{\times}(\ucI,F,\bS)$ is a local equivalence. 
\end{proposition}
\begin{proof}
  As a first step, we show that the degeneracy maps in the underlying
  simplicial objects are levelwise $h$-cofibrations on the base and
  the total spaces.  For $B^{\barsm}(\bS,\bS^{\cI}_t[G],\bS)$, this
  follows from~\cite[Proposition 12.7 and Lemma 7.7]{Sagave-S_diagram}
  and the explicit description of $\bS^{\cI}_t$
  in~\eqref{eq:bsIb-and-bSIt-explicit}. For $B^{\times}(\ucI,F,\bS)$
  this holds because $F$ is
  well-based~\cite[Theorem~2.1]{Lewis_when-cofibration}.

  Next we show that $\rho$ is an $\cI$-equivalence. Since we checked
  that the underlying simplicial object of base $\cI$-spaces is good,
  it is sufficient to show that for fixed $q$, the map of $\cI$-spaces
  $G^{\boxtimes q} \to F^{\times q}$ is an $\cI$-equivalence. For
  $q=2$, it can be identified as the composite of the morphism
  $\rho_{G,G}\colon G\boxtimes G \to {G\times G}$ considered
  in~\cite[Section~2.24]{Sagave-S_group-compl} and the level
  equivalence $G\times G \to F\times F$ induced by $\rho$. Since $G$
  is semistable and flat by construction, $\rho_{G,G}$ is an
  $\cI$-equivalence by~\cite[Proposition
  2.27]{Sagave-S_group-compl}. The assertion for $q>2$ follows by an
  inductive argument based on~\cite[Proposition~2.27]{Sagave-S_group-compl}.

  To check that~\eqref{eq:barsm-times-F-comparison} is a local
  equivalence, we form a commutative diagram
  \[\xymatrix@-1pc{
    B^{\barsm}(\bS,\bS^{\cI}_t[G],\bS) \ar@{ >->}[d]^{\sim} \ar[rrrr] && &&   B^{\times}(\ucI,F,\bS) \ar@{ >->}[d]^{\sim} \\
    (E,BG) \ar[rr] && \rho^*(E',B^{\times}(F)) \ar[rr] && (E',B^{\times}(F)) }\]
  where the vertical maps are fibrant replacements in the absolute
  local model structures on $\Spsym{BG}$ and $\Spsym{B^{\times}(F)}$ and the
  lower left hand horizontal map arises by extending the resulting map 
  $B^{\barsm}(\bS,\bS^{\cI}_t[G],\bS) \to \rho^*(E',B^{\times}(F))$ in
  $\Spsym{BG}$ over the left hand acyclic cofibration. By
  Corollary~\ref{cor:characterization-local-equivalences}, it is
  sufficient to show that the map of fibrant objects
  $ (E,BG) \to \rho^*(E',B^{\times}(F))$ is a local equivalence in~$\Spsym{BG}$.

  Let $\iota \colon * \to BG$ be the unit. As in the discussion
  preceding Proposition~\ref{prop:TEG-top-simp}, we may assume that
  $G$ is the realization of a cofibrant replacement of
  $\GLoneIof{(\Sing \bS)}$ and that $\bS^{\cI}[G] \to \bS$ factors
  through $|\Sing \bS| \to \bS$. Together with the above statement
  about $h$-cofibrations, this implies that the map
  $B^{\barsm}(\bS,\bS^{\cI}_t[G],|\Sing \bS|) \to
  B^{\barsm}(\bS,\bS^{\cI}_t[G],\bS)$
  is an absolute level equivalence. Using this, it follows from
  Proposition~\ref{prop:Thom-of-BH-BG} that the canonical map
  $\bS \to \iota^*(E,BG)$ is a stable equivalence. To get an
  analogous statement for $\rho^*(E',B^{\times}(F))$, we use the absolute
  $\cI$-model structure and the standard levelwise replacement by a
  Hurewicz fibration to factor $\rho\iota$ as an $\cI$-equivalence
  $j\colon * \to X$ followed by a map $q\colon X \to B^{\times}(F)$ that is both
  an absolute $\cI$-fibration and a levelwise Hurewicz fibration. This
  factorization give rise to a commutative square
\[\xymatrix@-1pc{
(\rho\iota)^*(B^{\times}(\ucI,F,\bS)) \ar[d] \ar[rr] && q^*(B^{\times}(\ucI,F,\bS)) \ar[d] \\
(\rho\iota)^*(E',B^{\times}(F)) \ar[rr] && q^*(E',B^{\times}(F)).
}\]
The right hand vertical map is a local equivalence by Proposition~\ref{prop:weak-cogluing} and thus a stable equivalence after applying $\mathbb L \Theta$. The top horizontal map is a stable equivalence after applying $\mathbb L \Theta$ by~\cite[Theorem 1.4]{Schlichtkrull_Thom-symmetric} (where the \emph{$T$-goodness} assumption is taken care of by~\cite[Lemmas 2.2 and 2.3]{Schlichtkrull_Thom-symmetric}). Since $j$ is an $\cI$-equivalence and $(E',B^{\times}(F))$ is fibrant, Corollary~\ref{cor:characterization-local-equivalences} implies that the bottom horizontal map is a local equivalence. Hence $\bS \iso (\rho\iota)^* (B^{\times}(\ucI,F,\bS)) \to (\rho\iota)^*(E',B^{\times}(F))$ is a stable equivalence. 

It follows that  $ (E,BG) \to \rho^*(E',B^{\times}(F))$ induces a stable equivalence after pullback along $\iota$. Since the fibers are $\Omega$-spectra, the induced map of fibers is even a level equivalence. By inspecting the simplicial object defining $BG$, each $BG(\bld{n})$ is connected. The long exact sequences of homotopy groups of the Serre fibrations $E_n \to BG(\bld{n})$ and $\rho(\bld{n})^*E_n' \to BG(\bld{n})$ show that $(E,G) \to \rho^*(E',B^{\times}(F))$ is a level equivalence and hence a local equivalence.  
\end{proof}
We now explain how the proposition leads to a comparison of Thom spectrum functors. Let $R = \Sing \bS$ and let $G \to \GLoneIof{R}$ be a cofibrant replacement of its $\cI$-space units. (This is homotopically meaningful since $R$ is level fibrant and semistable.) By the discussion before Proposition~\ref{prop:TEG-top-simp}, we get a map $\bS^{\cI}[|G|] \to \bS$ and note that $|G|$ can be used as a model for the cofibrant replacement of the units of the sphere spectrum in topological spaces. 

Next we consider $\overline\gamma_{\bS}$, the bar resolution of $B^{\barsm}(\bS,\bS^{\cI}_t[G],R)$. It is the object in $\SpsymR$ defined by 
\[ (\overline\gamma_{\bS})_{n} = \hocolim_{\alpha\colon \bld{k} \to \bld{n}} B^{\barsm}(\bS,\bS^{\cI}_t[G],R)_k \barsm S^{\bld{n}-\alpha} \]
where the homotopy colimit is taken over the category $\cI\downarrow \bld{n}$.  The base $\cI$-space of $\overline\gamma_{\bS}$ is the bar resolution of $BG$ used in the $\cI$-spacification. As in the case of symmetric spectra (see~\cite[Section 7.3]{Blumberg-C-S_THH-Thom}), there is a canonical level equivalence $\overline\gamma_{\bS} \to B^{\barsm}(\bS,\bS^{\cI}_t[G],R)$. Applying realization, composing this map with the fibrant replacement defining $\gamma_R$, and using Proposition~\ref{prop:barsm-times-F-comparison} gives a zig-zag of local equivalences 
\begin{equation}\label{eq:gamma_R-equivalent-to-BF}
|\gamma_R| \ot |\overline\gamma_{\bS}| \to B^{\times}(*,F,\bS). 
\end{equation}
Let $\tau \colon K \to BG_{h\cI}$ be a map of simplicial sets. We note 
that its $\cI$-spacification $\tau_{\cI} \colon P_{\tau}(K) \to BG$ factors by
definition as an absolute $\cI$-fibration $\overline\tau_{\cI}\colon  P_{\tau}(K) \to \overline{BG}$ and the canonical map $t\colon \overline{BG} \to BG$. This factorization, the maps~\eqref{eq:gamma_R-equivalent-to-BF}, and the map from Lemma~\ref{lem:PM-natural} give rise to a zig-zag
\begin{equation}\label{eq:T_R-vs_T_BF}
|T_{\cR}(\tau)| = |T^{\cI}_{\cR}(\tau_{\cI})| \ot |\Theta(\overline\tau_I^* (\overline\gamma_{\bS}))|
\to {T}^{\cI}_{BF}((\rho_{h\cI}\circ |\tau|)_{\cI}) =  T_{BF}(\rho_{h\cI} \circ |\tau|).
\end{equation}
Here we use the bar resolution in the middle term since $\overline{\tau}_{\cI}$ being a fibration ensures that $\overline\tau_I^* (\overline\gamma_{\bS})$ captures a well-defined homotopy type without $\overline\gamma_{\bS}$ being locally fibrant. 
\begin{proposition}\label{prop:T_R-vs_T_BF}
The maps in~\eqref{eq:T_R-vs_T_BF} are natural monoidal stable equivalences that respect actions of operads augmented over the Barratt--Eccles operad. 
\end{proposition}
\begin{proof}
The claim about monoidality and operad actions is clear since all constructions involved preserve these structures. To see that the first map in~\eqref{eq:T_R-vs_T_BF} is a stable equivalence, we factor $\tau_{\cI}$ as an $\cI$-equivalence followed by an absolute fibration and use Corollary~\ref{cor:characterization-local-equivalences}, Proposition~\ref{prop:weak-cogluing}, and the fact that the $\Theta$ in simplicial sets sends local equivalences to stable equivalences. For the second map, we argue as in the proof of the previous proposition, factor $(\rho_{h\cI}\circ |\tau|)_{\cI}$ as an $\cI$-equivalence followed by a map that is both an absolute $\cI$-fibration and a levelwise Hurewicz fibration, and apply Proposition~\ref{prop:barsm-times-F-comparison}, the $\cI$-equivalence part of Lemma~\ref{lem:PM-natural}, Proposition~\ref{prop:weak-cogluing}, and~\cite[Theorem~1.4]{Schlichtkrull_Thom-symmetric}. 
\end{proof} 
Together with the $\infty$-categorical comparison to be proved in Section~\ref{subsec:infty-thom}, these results now combine to the statement of  Theorem~\ref{thm:Thom-spectrum-comparison}:
\begin{proof}[Proof of Theorem~\ref{thm:Thom-spectrum-comparison}]
This is a combination of Propositions~\ref{prop:TEG-top-simp},~\ref{prop:TEG-vs-TR}, and~\ref{prop:T_R-vs_T_BF}, the $\infty$-categorical comparison in Lemma~\ref{lem:SpsymX-infty-functorial} and Proposition~\ref{prop:gamma_comp} below.
\end{proof}

\section{Comparison to \texorpdfstring{$\infty$}{infinity}-categorical parametrized spectra}\label{sec:infty-categorical}
If $\cC$ is a model category, we write $\cC_{\infty}$ for the underlying $\infty$-category of $\cC$. In particular, we write $\cS_{\infty}$ (resp.\ $\Sp_{\infty}$) for the $\infty$-category of spaces (resp.\ spectra). When $\cC$ is a symmetric monoidal model category, then $\cC_{\infty}$ inherits the structure of a symmetric monoidal $\infty$-category from $\cC$~\cite[Example 4.1.7.6]{Lurie_HA}. More specifically, the localization functor $\cC \rightarrow \cC_\infty$ is lax symmetric monoidal and strong symmetric monoidal when restricted to cofibrant objects, see \cite[Proposition~3.2.2]{Hinich_Dwyer-Kan} and also \cite[Appendix A]{Nikolaus-S_tc}.

\subsection{Comparison of categories}
The next lemma is closely related to \cite[Proposition B.1 and Theorem~B.4]{Ando-B-G_parametrized}. 
\begin{lemma}\label{lem:SpsymX-infty-identification}
  Let $X$ be an $\cI$-space. Then the $\infty$-category
  $(\Spsym{X})_{\infty}$ resulting from the local model structure on
  $\Spsym{X}$ is equivalent to $\mathrm{Fun}(X_{h\cI}, \Sp_{\infty})$,
  the $\infty$-category of functors from the underlying
  $\infty$-groupoid of $X_{h\cI}$ to the $\infty$-category of spectra.
\end{lemma}
\begin{proof}
  By~\cite[Remark 1.4.2.9]{Lurie_HA}, stabilization commutes with the
  passage to presheaf categories. Hence $\mathrm{Fun}(X_{h\cI},
  \Sp_{\infty})$ is equivalent to $\Sp(\mathrm{Fun}(X_{h\cI},
  \cS_{\infty}))$, the stabilization of the category of space valued
  functors on $X_{h\cI}$. The $\infty$-category
  $\mathrm{Fun}(X_{h\cI}, \cS_{\infty})$ is equivalent to
  $\cS_\infty/X_{h\cI} \simeq (\cS/X_{h\cI})_{\infty}$ by \cite[Theorem 2.2.1.2]{Lurie_HTT}. The fact that stabilization commutes with the passage to underlying $\infty$-categories and \cite[Corollary 10.4]{Hovey_symmetric-general} imply that $\Sp((\cS/X_{h\cI})_{\infty})$ is equivalent to   $(\Spsym{X_{h\cI}})_{\infty}$.
  The claim follows from the Quillen equivalence between $\Spsym{X}$ and $\Spsym{X_{h\cI}}$ established in Corollary~\ref{cor:SpsymX-vs-SpsymXhI}.
\end{proof}
In the next step, we show that the equivalence in
the last lemma is natural so that we can identify 
the functors between $\Spsym{X}$ and $\Spsym{Y}$ induced by
$f\colon X \to Y$ with their $\infty$-categorical counterparts including
coherences between compositions.

\begin{lemma}\label{lem:SpsymX-infty-functorial} The functors
  $(\cS^{\cI})_\infty\! \to\!\mathrm{Cat}_{\infty}$ given by $(\Spsym{(-)}\!)_{\infty}$  and  
  $\mathrm{Fun}((-)_{h\cI}, \Sp_{\infty})$ are equivalent.
\end{lemma}

\begin{proof}
  Our strategy is to show that they classify equivalent bicartesian
  fibrations. For this we first show that (up to replacement by
  categorical fibrations) the functors
  \[ (\SpsymR)_{\infty} \xrightarrow{u} \Ar(\cS^\cI)_{\infty}
  \xrightarrow{p} (\cS^{\cI})_{\infty}\]
  induced by the right Quillen functors
  $\Omega^\cI_\mathrm{ar} = \pi^\cI_{\mathrm{ar}} \circ
  \Omega^{\cI}_{\cR}$
  and $\pi_b$ (see
  Corollary~\ref{cor:S-I-b-left-right-Quillen-locally}
  and Lemmas~\ref{lem:Q-adjunctions-on-local-model-str} and~\ref{lem:SIar-OmegaIar-Quillen-adjunction}) exhibit $u$ as the
  stable envelope of $p$ in the sense of \cite[Definition~7.3.1.1]{Lurie_HA}. For \cite[Definition~7.3.1.1(1)]{Lurie_HA}, we
  note that both $p$ and~$pu$ are presentable fibrations because the
  fibers are the underlying $\infty$-categories of combinatorial model
  categories (using the simplicial version of our categories) and they
  are Cartesian and coCartesian by~\cite[Proposition
  3.1.2]{Harpaz-P_Grothendieck-construction} and (the dual
  of)~\cite[Corollary 5.2.2.5]{Lurie_HTT}. For the condition
  \cite[Definition 7.3.1.1(2)]{Lurie_HA} we need the following
  observation: By \cite[Proposition 2.4.4.3]{Lurie_HTT} a morphism
  $(f,g)\colon (E,X) \rightarrow (E',X')$ in $(\SpsymR)_{\infty}$ is
  $pu$-cartesian if (after (co)fibrant replacement of source and
  target, respectively) the induced map
  $(f,\mathrm{id}) \colon (E,X) \rightarrow g^*(E',X')$ is a local
  equivalence, and similarly for the case of $p$-cartesian
  morphisms. But this condition is preserved by $u$ since
  $g^* \pi^\cI_{\mathrm{ar}}\Omega^\cI_{X'} \cong \Omega^\cI_{X}
  \pi^\cI_{\mathrm{ar}} g^*$
  for any map $g \colon X \rightarrow X'$ and all functors involved
  are right Quillen. For \cite[Definition 7.3.1.1(3)]{Lurie_HA}, we
  note that \cite[Example 7.3.1.4]{Lurie_HA} and
  Lemma~\ref{lem:stabilization-quillen-equiv} imply that $u$ restricts
  to a stable envelope on fibers.

   Now, \[ \int_{\cS_{\infty}}\!\!\!\mathrm{Fun}(-,\Sp_{\infty})
\xrightarrow{\int \Omega^\infty} \int_{\cS_{\infty}}\!\!\!\mathrm{Fun}(-,(\cS_{\infty})_\ast) \xrightarrow{\mathrm{pr}} \cS_{\infty}\]
is also well-known to be a stable envelope. Thus the uniqueness theorem  \cite[Proposition 7.3.1.7(3)]{Lurie_HA} reduces the claim to the corresponding space level statement. The latter is obtained from the Quillen equivalence $\colim_{\cI} \colon \cS \rightleftarrows \cS^\cI\colon \const_\cI$ and the equivalences 
\[\xymatrix@-1pc{\mathrm{Ar}(\cS_{\infty}) \ar[drr]_{\mathrm{ev}_1}\ar[rr]^-\simeq && \int_{\cS_{\infty}}\!\!\!\cS_{\infty}/-  \ar[rr]^-\simeq \ar[d]^{\pi_b} && \int_{\cS_{\infty}}\!\!\!\mathrm{Fun}(-,\cS_{\infty}) \ar[lld]\\
  && \cS_\infty && }\]
of bicartesian fibrations where the right equivalence follows from the unstraightening equivalence \cite[Theorem 2.2.1.2]{Lurie_HTT}, and the left is immediate from the main theorem of~\cite{Harpaz-P_Grothendieck-construction} together with the compatibility between diagram categories of model and $\infty$-categories.
\end{proof}
Next we give an $\infty$-categorical interpretation of the category of symmetric spectra in retractive spaces $\SpsymR$. To do so let us denote 
\[T\cS_\infty = \int_{\cS_{\infty}}\!\!\!\mathrm{Fun}(-,\Sp_{\infty}),\]
the $\infty$-categorical Grothendieck construction of $\mathrm{Fun}(-,\Sp_{\infty}) \colon \cS_{\infty} \to \mathrm{Cat}_{\infty}$, which is a model for the tangent bundle of the category $\cS_\infty$ as in \cite[Section 7.3.1]{Lurie_HA}.

\begin{proposition}\label{prop:infty comp}
  The $\infty$-category $(\SpsymR)_{\infty}$ resulting from the
  local model structure on $\SpsymR$ is canonically equivalent to $T\cS_\infty$. 
\end{proposition}
\begin{proof} 
  Combining~\cite[Proposition
  3.1.2]{Harpaz-P_Grothendieck-construction},
  Theorem~\ref{thm:local-Grothendieck-construction}, and
  Lemma~\ref{lem:SpsymX-infty-identification}, we see that the
  $\infty$-category $(\SpsymR)_{\infty}$ is equivalent to the
  $\infty$-categorical Grothendieck construction of
  $(\Spsym{(-)})_{\infty}\colon (\cS^{\cI})_{\infty} \to
  \mathrm{Cat}_{\infty}$.
  By Lemma~\ref{lem:SpsymX-infty-functorial}, we can identify it with the
  $\infty$-categorical Grothendieck construction of
  $\mathrm{Fun}(-,\Sp_{\infty})\colon (\cS^{\cI})_{\infty} \to
  \mathrm{Cat}_{\infty}$.
  The claim follows because $\cS^{\cI}$ and $\cS$ are Quillen
  equivalent~\cite[Theorem~3.3]{Sagave-S_diagram}.
\end{proof}

\subsection{Comparison of symmetric monoidal categories}
When $\cC$ is a symmetric monoidal $\infty$-category, then the Day
convolution product gives rise to a symmetric monoidal structure on
$\mathrm{Fun}(\cC,\cS_{\infty})$, see~\cite{Gla}.
Using~\cite[Theorem 5.1]{Gepner-G-N_universality} and the equivalence
$\mathrm{Sp}(\mathrm{Fun}(\cC,\cS_{\infty})) \simeq
\mathrm{Fun}(\cC,\Sp_{\infty})$, we thus get a uniquely determined
symmetric monoidal structure on $\mathrm{Fun}(\cC,\Sp_{\infty})$.

If $M$ is a commutative $\cI$-space monoid, then $M_{h\cI}$ inherits
an action of the Barratt--Eccles operad~\cite[Proposition
6.5]{Schlichtkrull_Thom-symmetric}.  By~\cite[Proposition
4.1]{Nikolaus_S-presentably}, the underlying $\infty$-groupoid of
$M_{h\cI}$ represents a symmetric monoidal $\infty$-groupoid, and
$\mathrm{Fun}(M_{h\cI}, \Sp_{\infty})$ is a symmetric monoidal
$\infty$-category by the above discussion.

\begin{theorem}\label{thm:SpsymM-infty-identification}
  Let $M$ be a commutative $\cI$-space monoid. Then the symmetric
  monoidal $\infty$-category $(\Spsym{M})_{\infty}$ resulting from the
  absolute or positive local model structures on $\Spsym{M}$ and
  $\mathrm{Fun}(M_{h\cI}, \Sp_{\infty})$ are equivalent as symmetric
  monoidal $\infty$-categories.
\end{theorem}

\begin{remark}
There is a functorial rigidification $(-)^\mathrm{rig}$ of $E_{\infty}$ spaces (in their incarnation as spaces with an action of the Barratt-Eccles operad) to
commutative $\cI$-space monoids, so that $M
\simeq (M^{\mathrm{rig}})_{h\cI}$ as $E_{\infty}$
spaces~\cite[Corollary 3.7]{Sagave-S_diagram}. In this situation, the
previous theorem implies that $\Spsym{M^{\mathrm{rig}}}$ represents
the symmetric monoidal $\infty$-category $\mathrm{Fun}(M,
\Sp_{\infty})$. 

The existence of a rigidification of $\mathrm{Fun}(M,
\Sp_{\infty})$ is also implied by ~\cite[Theorem
1.1]{Nikolaus_S-presentably}, but the above construction provides a smaller model.
\end{remark}

\begin{proof}[Proof of Theorem~\ref{thm:SpsymM-infty-identification}]
  We show that
  the assignments $M \mapsto \mathrm{Fun}(M_{h\cI},\Sp_{\infty})$ and
  $M \mapsto (\Spsym{M})_{\infty}$ are equivalent as functors
  $(\cC\cS^{\cI})_{\infty} \to \mathrm{SymMonCat}_{\infty}$. By
  Lemma \ref{lem:SpsymX-infty-functorial}, they are equivalent as
  functors to $\mathrm{Cat}_{\infty}$. So it remains to compare the
  symmetric monoidal structures. For this, we adapt the argument given
  in the proof of \cite[Proposition 2.4]{Nikolaus_S-presentably}. The
  only change that is necessary to apply it in the case at hand is
  that we have to argue with the forgetful functor
  $U \colon (\cC\cS^{\cI})_{\infty} \to \cS_{\infty}$ rather than with the one
  $\mathrm{SymMonCat}_{\infty} \to \mathrm{Cat}_{\infty}$.
\end{proof}

In order to compare our construction of the universal bundle $\gamma_R$ with that of \cite{Ando-B-G-H-R_infinity-Thom}, we also have to compare the monoidal structures on $\SpsymR$ and $T\cS_\infty$. We are, however, not aware of a construction of the requisite symmetric monoidal structures on tangent categories in the literature (\cite[Example 7.3.1.15]{Lurie_HA} only gives the cartesian monoidal structure on $T\cS_\infty$). To describe it,
consider therefore the following general situation: 
Suppose we are given cocartesian fibrations of $\infty$-operads
\[\xymatrix@-1pc{\cC^\otimes \ar[r]^-p\ar[rd] & \cD^\otimes \ar[d]^\chi \\
                               & \NFin.}\]
admitting finite operadic limits. In particular, by definition $\cC$ and $\cD$ are then symmetric monoidal $\infty$-categories (see \cite[Example 2.1.2.18]{Lurie_HA}). The example of relevance for us is $\cD^\otimes = \cS_\infty^\times$ and $\cC^{\tensor} = \Ar(\cS_\infty)^\times$ with $p$ the projection to the base space. We are then looking for a left exact map of $\infty$-operads $u \colon \cE^\otimes \rightarrow \cC^\otimes$, i.e., a lax symmetric monoidal functor, that exhibits $\cE^\otimes$ as a fiberwise stabilization of $p$ in the following sense. We require that $\cE$ is a stable $\cD$-monoidal $\infty$-category in the sense of \cite[Definition 7.3.4.1]{Lurie_HA} and for all other stable $\cD$-monoidal $\infty$-categories $\cB$, postcomposition with $u$ induces an equivalence
\[\mathrm{Op}^\mathrm{lex}_{/\cD}(\cB,\cE) \longrightarrow \mathrm{Op}^\mathrm{lex}_{/\cD}(\cB,\cC)\]
where we have used $\mathrm{Op}^\mathrm{lex}$ to denote the category of operads with finite operadic limits and operad maps preserving these (as in \cite[Definition 2.3]{nik-stable}). The case $\cD = pt$ (i.e., $\cD^\otimes = \NFin, \chi = \id$) is extensively discussed in both \cite[Sections 6.2.4 - 6.2.6]{Lurie_HA} and \cite[Section 4]{nik-stable}. The construction in \cite[Section 6.2.5]{Lurie_HA} provides a candidate $u \colon \mathrm{St}_\chi(p) \rightarrow \cC^\otimes$ for such a fiberwise stabilization also in our case. We will review this construction in the next proof.  

Already for $\chi = \id$, however, the construction in general does not provide a map $u$ such that $pu$ is cocartesian, but rather only locally so, for example if $\cC$ is the category of pointed spaces under the cartesian product, see \cite[Example~6.2.4.17]{Lurie_HA}. 

A useful criterion for $pu$ to be cocartesian is established in \cite[Proposition~4.9]{nik-stable}. This result readily generalizes to give the following:

\begin{proposition}\label{fibremon}
  Let $p \colon \cC^\otimes \rightarrow \cD^\otimes$ be a cocartesian fibration of $\infty$-operads between symmetric monoidal $\infty$-categories. Assume further that $\cC$ is differentiable, and that for any multimorphism $(d_1, \dots, d_n) \rightarrow d'$ in $\cD^\otimes$ the induced functor $\cC_{d_1} \times \dots \times \cC_{d_n} \rightarrow \cC_{d'}$ commutes with colimits in each variable.

Then $pu \colon \mathrm{St}_\chi(p) \rightarrow \cD^\otimes$ makes $\mathrm{St}_\chi(p)$ a cocartesian fibration of $\infty$-operads. In particular, the composite $\chi pu$ makes $\mathrm{St}_\chi(p)$ into a symmetric monoidal $\infty$-category and $u \colon \mathrm{St}_\chi(p) \rightarrow \cC^\otimes$ is a map of $\infty$-operads, i.e., it is lax symmetric monoidal. Furthermore, the $\cD$-algebra structure on $\mathrm{St}_\chi(p)$ satisfies the same commutation with colimits as that of $\cC$, in particular the symmetric monoidal structure commutes with colimits in each variable. Finally, $u$ exhibits $\mathrm{St}_\chi(p)$ as a fiberwise stabilization of $p$ in the sense above.
\end{proposition}

Note that as cocartesian fibrations, both $p$ and $pu$ automatically preserve cocartesian lift of morphisms in $\NFin$, and are therefore strong symmetric monoidal.

\begin{remark}
Contrary to another claim in \cite[Example~6.2.4.17]{Lurie_HA}, \cite[Proposition 4.9]{nik-stable} or \cite[Examples~6.2.1.5 or 6.2.3.28]{Lurie_HA} imply that the composite $pu$ is a cocartesian fibration for $\cC^\otimes = (\cS_*)_\infty^\wedge$, the category of pointed spaces with the smash product.
\end{remark} 
\begin{proof}[Proof of Proposition~\ref{fibremon}]
Let us recall the definition of $\mathrm{St}_\chi(p)$ from \cite[Construction 6.2.5.20]{Lurie_HA}. Consider first the simplicial set $q \colon P\mathrm{St}_\chi(p) \rightarrow \cD^\otimes$ given by the universal property that for all simplicial sets $K$ over $\cD^\otimes$, we have
\[\Hom_{\cD^\otimes}(K,P\mathrm{St}_\chi(p)) \iso \Hom_{\cD^\otimes}(K \times_\NFin (\cS^{\mathrm{fin}}_*)^\wedge, \cC^\otimes).\]
Then $D\mathrm{St}_\chi(p)$ is given as the full simplicial subset of $P\mathrm{St}_\chi(p)$ spanned by those vertices $v$ that correspond to product functors
\[\prod_{\chi q(v) \setminus \{\ast\}} \cS^{\mathrm{fin}}_* \simeq (\cS^{\mathrm{fin}}_*)^\wedge_{\chi q(v)} \longrightarrow \cC^\otimes_{q(v)} \simeq \prod_{i \in \chi q(v) \setminus \{\ast\}} \cC_{q_i(v)},\]
and $\mathrm{St}_\chi(p)$ is spanned by those vertices corresponding to product functors whose factors are reduced and excisive. The structure map $u \colon D\mathrm{St}_\chi(p) \subseteq P\mathrm{St}_\chi(p) \rightarrow \cC^\otimes$ is given by taking $K = P\mathrm{St}_\chi(p) \rightarrow \cD^\otimes$ in the universal property above and precomposing the map corresponding to the identity with the section $\NFin \rightarrow (\cS^{\mathrm{fin}}_*)^\wedge$ that witnesses the $E_\infty$-structure of the unit $S^0$. By construction we have $q = pu$.

Now the restriction of $q$ to $D\mathrm{St}_\chi(p)$ is a cocartesian fibration just as in \cite[Lemma 2.10]{Gla} (which treats the case $\cD = pt$). It follows that in order to recognize $D\mathrm{St}_\chi(p)$ as a symmetric monoidal category, we only have to verify the Segal condition for $\chi q$, but this is immediate (compare \cite[Proposition 2.11]{Gla}). By definition $q$ is then a cocartesian fibration of operads. Part (2) of \cite[Proposition 2.2.1.9]{Lurie_HA} (which should have $\mathcal O^\otimes$ instead of $\mathrm{NFin}_*$ as the target of $p|\cD^\otimes$) applied to the localization functor provided by \cite[Theorem 6.1.1.10]{Lurie_HA} (precomposed with reduction) therefore implies the same for the restriction of $q$ to $\mathrm{St}_\chi(p)$, as desired; the assumptions of \cite[Theorem 2.1.1.9]{Lurie_HA} follow from the chain rule for the first derivative \cite[Theorem 6.2.1.22]{Lurie_HA}. That $u$ is lax symmetric monoidal follows just as in \cite[Corollary 3.8]{nik-stable}, and for the restriction to $\mathrm{St}_\chi(p)$ it then follows from \cite[Proposition 2.2.1.9 (3)]{Lurie_HA}. Commutation with colimits for the operad structure on $D\mathrm{St}_\chi(p)$ can be verified as in \cite[Lemma 2.13]{Gla} and for $\mathrm{St}_\chi(p)$ it follows from \cite[Proposition 2.2.1.9 (3)]{Lurie_HA}.

Finally, to see that $u$ indeed exhibits $\mathrm{St}_\chi(p)$ as a stabilization of $p$, we note that $\mathrm{St}_\chi(p)$ is stable by part (3) of \cite[Theorem 6.2.5.25]{Lurie_HA}, and in particular $u$ is left exact. To finish the proof we invoke part (3) of \cite[Theorem 6.2.6.6]{Lurie_HA} and need to check that a left exact decomposition-stable functor $F \colon \cE^\otimes \rightarrow \mathrm{St}_\chi(p)$ with fiberwise stable source is a map of $\infty$-operads if and only $uF \colon \cE^\otimes \rightarrow \cC^\otimes$ is. But this can be shown almost verbatim as in the proof of \cite[Theorem 6.2.6.2]{Lurie_HA} on the same page.
\end{proof}

We apply the above to the cocartesian fibration of operads $\pi_c \colon \Ar(\cS_\infty)^\times \longrightarrow \cS_\infty^\times$
projecting to the codomain. It produces a symmetric monoidal structure on $T\cS_\infty$ that models the exterior smash product $\barsm$.

\begin{theorem}\label{thm: comp-mult-total}
The categories $(\SpsymR)_{\infty}$ and $T\cS_\infty$ are canonically equivalent as stable $(\cS_\infty)^\times$-monoidal categories, and thus, in particular, as symmetric monoidal $\infty$-categories.
\end{theorem}

\begin{proof}
To avoid confusion we shall denote the monoidal structure on $T\cS_\infty$ by the generic $\otimes$ with unit $\mathbf{1}$ in the present proof.

Let $(\cS^\cI)^\boxtimes_\infty$ denote the $\infty$-operad underlying the symmetric monoidal structure on $\cS^\cI$ given by the $\boxtimes$-product. By Propositions \ref{prop:SpsymX-stable} and \ref{prop:local-pushout-product}, $\pi_b\colon (\SpsymR)_\infty \to (\cS^\cI)^\boxtimes_\infty $ exhibits $(\SpsymR)_\infty^{\barsm}$ as a stable $(\cS^\cI)^\boxtimes_\infty$-monoidal category.  By Lemmas~\ref{lem:SIar-OmegaIar-adjunction-monoidal}, \ref{lem:Q-adjunctions-on-local-model-str}, and~\ref{lem:SIar-OmegaIar-Quillen-adjunction}, the functor $(\Omega^\cI_{\mathrm{ar}})_\infty \colon (\SpsymR)^\barsm_\infty \rightarrow \Ar(\cS^\cI)^\boxtimes_\infty$ is left exact and a map of $\infty$-operads, i.e., lax symmetric monoidal. From the equivalence $(\cS^\cI)^\boxtimes_\infty \simeq \cS^\times_\infty$ and Proposition \ref{fibremon}, we therefore obtain a left exact map of $\infty$-operads
\[c \colon (\SpsymR)^\barsm_\infty \longrightarrow T\cS_\infty^\otimes,\]
i.e., a lax symmetric monoidal functor. By the universal property of stabilizations it agrees with that from Proposition \ref{prop:infty comp} when restricted to $(\SpsymR)_\infty$. Once we show that $c$ is in fact strong symmetric monoidal, then \cite[Remark 2.1.3.8]{Lurie_HA} implies that it is an equivalence of symmetric monoidal $\infty$-categories. 

To see the latter we need to verify that the canonical maps
\[\mathbf{1} \rightarrow c(\mathbb S) \quad \quad \text{and} \quad\quad c(X) \otimes c(Y) \rightarrow c(X \barsm Y)\]
are equivalences, where $\mathbf{1}$ denotes the unit of $T\cS_\infty$. To do so we note that the functor $\Omega^\infty \colon T\cS_\infty \rightarrow \Ar(\cS_\infty)$ admits a strong symmetric monoidal adjoint $\Sigma^\infty_+$. To construct it recall the map of $\infty$-operads $u \colon D\mathrm{St}_\chi(t) \rightarrow \mathrm{Ar}(\cS_\infty)^\times$ from the previous proof. We then invoke \cite[Corollary 7.3.2.12]{Lurie_HA} for $u$ and \cite[Proposition 2.2.1.9]{Lurie_HA} for the restriction to the localization $T\cS_\infty = \mathrm{St}_\chi(\mathrm{ev}_1)$ of $D\mathrm{St}_\chi(\mathrm{ev}_1)$. Their assumptions are verified just as in \cite[Corollary 3.8 and Proposition 4.9]{nik-stable}. This argument immediately shows that the map between unit objects is an equivalence as the sphere $\bS \in \SpsymR$ is also given by the left adjoint $\bS^\cI_\cR \circ \iota^\cI_{\mathrm{ar}} \colon \mathrm{Ar}(\cS^\cI) \rightarrow \SpsymR$ evaluated on the unit $* \rightarrow *$ of $\mathrm{Ar}(\cS^\cI)$. The second claim follows since the class of pairs of spectra, for which the map in question is an equivalence is closed under colimits in either variable and contains pairs of suspension spectra by the monoidality of $\Sigma^\infty_+$.
\end{proof}

\begin{remark}
When applied to the example of the target fibration 
\[\Ar(\mathrm{Alg}_{E_\infty}(\mathrm{Sp}_\infty))^\otimes \longrightarrow \mathrm{Alg}_{E_\infty}(\mathrm{Sp}_\infty)^\otimes\]
considered in detail in \cite[Section 7.3]{Lurie_HA}, the symmetric monoidal structure from Proposition \ref{fibremon} on $T(\mathrm{Alg}_{E_\infty}(\mathrm{Sp}_\infty))$ is also readily identified: Under the equivalence of $T(\mathrm{Alg}_{E_\infty}(\mathrm{Sp}_\infty))$ with the category of all modules over $E_\infty$-ring spectra \cite[Theorem 7.3.4.18]{Lurie_HA} it corresponds to the smash product on both the rings and the modules.
\end{remark}

We have now verified all parts of the $\infty$-categorical theorem
from the introduction:
\begin{proof}[Proof of Theorem~\ref{thm:infty-categorical-products-intro}]
This is a combination of Lemma~\ref{lem:SpsymX-infty-functorial}, Proposition~\ref{prop:infty comp}, and Theorems~\ref{thm:SpsymM-infty-identification} and~\ref{thm: comp-mult-total}. 
\end{proof}

\subsection{Comparison of universal bundles}\label{subsec:infty-thom}
In order to formulate the comparison we will denote objects and functors corresponding on the infinity categorical side to objects of the same nature as those introduced in the previous sections by the same name without the decoration $\cI$. 

Let us then first recall the definition of $B\GLoneof{R}$ in the $\infty$-categorical setting. For an $E_\infty$-ring spectrum $R$ the invertible $R$-module spectra and their equivalences span a sub-$\infty$-groupoid $\mathrm{Pic}(R) \in \cS_\infty$ in the $\infty$-category of all $R$-module spectra. This groupoid inherits a symmetric monoidal structure from the tensor product of $R$-modules, making it an $E_\infty$-space with unit $R$. The component of the unit $R$ is usually denoted by $\mathrm{Pic}_0(R)$ and as a mere space is clearly also defined for an $E_1$-ring spectrum $R$. 

\begin{lemma}
Let $R$ be an $E_1$ (resp.\ $E_{\infty}$) ring spectrum, and let $(\Omega^\infty R)^\times$ be the subspace of $\Omega^\infty R$ corresponding to the  units in the multiplicative $E_1$ (resp.\ $E_{\infty}$) structure. Then there are canonical equivalences 
\[\mathrm{Pic}_0(R) \simeq B\mathrm{Aut}_R(R) \simeq B(\Omega^\infty R)^\times\]
in $\cS_\infty$ (resp.  $\mathrm{Alg}_{E_\infty}(\cS_\infty)$). 
\end{lemma}

We shall refer to any of these equivalent spaces as $B\GLoneof{R}$.

\begin{proof}
By definition of mapping space there is, for every pair of objects $x,y$ of an $\infty$-category $\cC$, a cartesian diagram in $\mathrm{Cat}_\infty$
\[\xymatrix{\mathrm{Hom}_\cC(x,y) \ar[r]\ar[d] & \cC/y \ar[d]\\
                \ast \ar[r]^-x & \cC.}\]
If now $\cC \in \cS_\infty$ is an $\infty$-groupoid and $x = y$, $\cC/y$ is a contractible $\infty$-groupoid and we obtain an equivalence $\Omega_x \cC \simeq \mathrm{Aut}_\cC(x)$ and consequently $\cC_x \simeq B\mathrm{Aut}_\cC(x)$, where $\cC_x$ denotes the path component of $x$ in $\cC$. Furthermore, if $\cC$ is symmetric monoidal and $x$ its unit, then $\cC/x$ inherits a symmetric monoidal structure from $\cC$, the diagram defines an $E_\infty$-structure on $\mathrm{Aut}_\cC(x)$ and then becomes cartesian in $\mathrm{Alg}_{E_\infty}(\cS_\infty)$. Applied to $\cC = \mathrm{Pic}_0(R)$ and $x = R$, we obtain the first desired equivalence. 

For the second we have to distinguish the two cases: If $R$ is an $E_\infty$-ring spectrum, it arises from the adjunction equivalence
\[\Omega^\infty(R) \simeq \mathrm{Hom}_{\mathbb S}(\mathbb S,R) \xrightarrow{-\wedge R}  \mathrm{Hom}_{R}(R,R \sm R)\xrightarrow{\mu}\mathrm{End}_{R}(R) \]
since the functor $- \wedge R$ is symmetric monoidal and the multiplication of $R$ is an $E_\infty$-map. 

In the case of an $E_1$-ring spectrum $R$ the middle term does not carry an evident multiplication so we have to argue differently. We can obtain a map of $E_1$-ring spectra $R \rightarrow \mathrm{end}_R(R)$ from \cite[Corollary 4.7.1.41 and  Remark 7.1.2.2]{Lurie_HA}; here we apply the corollary to $R$ considered as a left $R$-module spectrum in the $\infty$-category of right $R$-module spectra, which is tensored over and consequently enriched in the $\infty$-category of spectra, see \cite[Proposition 4.2.1.33 and Remark 4.8.2.20]{Lurie_HA}. Applying the (symmetric) monoidal functor $\Omega^\infty \simeq \mathrm{Hom}_{\mathbb S}(\mathbb S,-)$ to the above arrow produces a map $\Omega^\infty R \rightarrow \mathrm{End}_R(R)$ of $E_1$-spaces, which equals the above composite, and is therefore an equivalence. 

It remains to check that the $E_1$-structure constructed on $\mathrm{Aut}_R(R)$ above via the identification with $\Omega\mathrm{Pic}_0(R)$ agrees with the restriction of that just constructed on $\mathrm{End}_R(R)$ to its units. But this is clear, since lifts of the functor $\Omega \colon */\cS^{\geq 1}_\infty \rightarrow \cS_\infty$ to $E_1$-monoids correspond to comonoid structures on $S^1$. Under the equivalence of $*/\cS_\infty^{\geq 1}$ with $E_1$-groups, these correspond to cogroup structures on the integers, of which there are only two, corresponding to forwards and backwards concatenation of loops (and the above constructions evidently give the same concatenation map up to homotopy).
% This is well-known, but we do not know of a reference where it is spelled out. We sketch an argument: By naturality, the restriction of the latter is given by the endomorphism object of $R$ when regarded as the object of $\mathrm{Pic}_0(R)$. But for an arbitrary object $x$ in an $\infty$-groupoid $\cC \in \cS_\infty$ the $E_1$ structure on its endomorphisms is given as follows: One considers the category $\Fun(\cC^{\mathrm{op}},\cS_\infty)$, which is tensored over $\cS_\infty$ by \cite[Corollary 4.8.1.4]{Lurie_HA}, since $\cS_\infty$ is the unit for the tensor product of cocomplete categories by \cite[Remark 4.8.1.8]{Lurie_HA}. Then the endomorphisms give a final object in the monoidal $\infty$-category $\cS_\infty[x]$ (formed with respect to this tensoring, see \cite[Definition 4.7.1.1 and Proposition 4.7.1.30]{Lurie_HA}), and noting that final objects in monoidal categories admit unique $E_1$-structures \cite[Corollary 3.2.2.4]{Lurie_HA}. But the category $\cS_\infty[x]$ is equivalent to $(\cS_\infty)/\Omega_x\cC$ via unstraightening, with its monoidal structure induced by that of the loop space, and in the latter category, a final object is given the identity of $\Omega_x\cC$.
\end{proof}

\begin{remark}
For an $E_n$-ring spectrum with $n \geq 2$, the category of left $R$-modules is $E_{n-1}$-monoidal so $\mathrm{Pic}(R)$ is an $E_{n-1}$-space in this case, as is $B\Omega^\infty R$. The method we gave for the $E_\infty$-case identifies these as $E_{n-2}$-spaces and one can check that the argument we gave for $E_1$-ring spectra identifies the remaining $E_1$-structures in a compatible fashion. In total one thus obtains an identification as $E_{n-2} \otimes E_1 \simeq E_{n-1}$-spaces. 
\end{remark} 

The space $B\GLoneof{R}$ comes equipped with the canonical functor \[F_R \colon B\GLoneof{R} \simeq \mathrm{Pic}_0(R) \longrightarrow R\text{-}\mathrm{Mod} \longrightarrow  \Sp_\infty,\] which witnesses the action of $\GLoneof{R}$ on $R$.

\begin{proposition}\label{prop:gamma_comp}
For any symmetric ring spectrum $R$, the equivalence from Proposition \ref{prop:infty comp} carries $\gamma_R \in \SpsymR$ to the image under the inclusion \[\mathrm{Fun}(B\GLoneof{R},\Sp_{\infty}) \longrightarrow \int_{\cS_{\infty}}\!\!\!\mathrm{Fun}(-,\Sp_{\infty}) = T\cS_\infty\] of the functor $F_R$.
If $R$ is commutative, then  $F_R$ is a lax symmetric monoidal functor, and the same is true as $E_\infty$-monoid objects of $T\cS_\infty$.
\end{proposition}

As preparation we record:

\begin{lemma}\label{Gl cpmp}
Let $R$ be a positive fibrant symmetric ring spectrum and $G$ a cofibrant replacement of $\GLoneIof{R}$. Then the image of $BG$ under the equivalence $\cS_\infty^\cI \simeq \cS_\infty$ is $B\GLoneof{R}$ and the evaluation map $\bS[G] \rightarrow R$ corresponds to the canonical map $\bS[\GLoneof{R}] \rightarrow R$. When $R$ is commutative, the same identifications hold in the $\infty$-categories $(\mathcal C\cS^\cI)_\infty \simeq \mathrm{Alg}_{E_\infty}(\cS_\infty)$ and $(\mathcal C\Spsym{})_\infty \simeq \mathrm{Alg}_{E_\infty}(\Sp_\infty)$.
\end{lemma}

\begin{proof}
All of the assertions follow immediately from the previous lemma and the commutative square 
\[\xymatrix@-1pc{(\Spsym{})_\infty^\wedge \ar[rr]^-\simeq \ar[d]_{\Omega^\cI} && \Sp_\infty^\wedge \ar[d]^{\Omega^\infty}\\
            (\cS^\cI)_\infty^\boxtimes \ar[rr]^-\simeq               && \cS_\infty^\times}\]
of symmetric monoidal $\infty$-categories and symmetric monoidal functors (lax in case of the vertical functors): By the monoidality of the horizontal functors the objects $\Omega^\cI R$ and $ \Omega^\infty R$ correspond as $E_1/E_\infty$ objects and the units are given by the same restriction to path components. The respective bar constructions then agree as simplicial objects by the cofibrancy assumption on $G$ and Hinich's result on the strong monoidality on cofibrant objects of the localization functor from a monoidal model category to its underlying $\infty$-category, see~\cite[Proposition~3.2.2]{Hinich_Dwyer-Kan} and also \cite[Appendix A]{Nikolaus-S_tc}. Finally, geometric realization in a simplicial model category models the colimit in its underlying $\infty$-category by \cite[Theorem 4.2.4.1]{Lurie_HTT}, which gives the claim about $B\GLoneof{R}$.

The maps from the spherical group rings to $R$ are the counits of the vertical adjunction with the suspension functor, so also corresponds under the above equivalences.
\end{proof}

\begin{proof}[Proof of Proposition \ref{prop:gamma_comp}]
The idea is to specify both $\gamma_R$ and $F_R$ in terms of data pinned down by Theorem \ref{thm: comp-mult-total} and the previous lemma. Namely, we will show that both objects are given by the relative $\barsm$-product of the diagram $\bS \leftarrow \bS_t[\GLoneof{R}] \rightarrow R$ (with the appropriate interpretations in $(\SpsymR)_\infty$ and $T\cS_\infty$, respectively). Note that this really is determined by the lemma, since the map $\bS_t[\GLoneof{R}] \rightarrow R$ factors by definition (as an $E_1$/$E_\infty$-map as appropriate) through the tautological map $\bS_t[\GLoneof{R}] \rightarrow \bS[\GLoneof{R}]$. For $\gamma_R \in (\SpsymR)_\infty$ the identification with the relative $\barsm$-product holds by \cite[Theorem 4.4.2.8 (ii)]{Lurie_HA}, since relative tensor products are computed by the bar construction. 

In case $R$ is commutative, we first note that  \cite[Example 3.2.4.4, Proposition 3.2.4.7 and Theorem 4.5.2.1]{Lurie_HA} together say that for an $E_\infty$-ring spectrum $A$, the coproduct of $E_\infty$-algebras in $A$-modules is given by the relative tensor product. Combining this with the facts that $E_\infty$-algebras in $A$-modules are the same thing as $E_\infty$-algebras under $A$ \cite[Corollary 3.4.1.7]{Lurie_HA} and that coproducts in slice categories are computed as pushouts in the original category, it follows that the above relative tensor product inherits an $E_\infty$-structure which makes it the pushout of $E_\infty$-rings.
In particular, this structure is again determined by data preserved under the equivalence of Theorem \ref{thm: comp-mult-total}. By applying \cite[Theorem 4.4.2.8]{Lurie_HA} to both the category of parametrized spectra and $E_\infty$ algebras therein, we find that the $E_\infty$ structure on the relative tensor product agrees with that coming from the termwise one on the bar construction (the forgetful functor from $E_\infty$ algebras commutes with geometric realization by \cite[Proposition 3.2.3.1]{Lurie_HA}). We therefore find that the $E_\infty$ structure on $\gamma_R \in (\SpsymR)_\infty$ agrees with that on the relative tensor product.

For the case of $F_R \in TS_\infty$ we argue by identifying both $F_R$ and the relative $\barsm$-product as colimits of the functor 
\begin{equation}\label{colim eq}
G_R \colon B\GLoneof{R} \xrightarrow{F_R} \Sp_{\infty} \longrightarrow T\cS_\infty
\end{equation}
whose second part is the inclusion over the one point space. If $R$ is commutative, it will be an identification as $E_\infty$-algebras. This has meaning since $F_R$ and thus $G_R$ are lax symmetric monoidal in this case, whence the colimit of $G_R$ inherits an $E_\infty$-structure for example by \cite[Proposition 3.3 and Corollary 3.8]{nik-stable}.

To see that $\colim G_R \simeq F_R$, we employ one direction of \cite[Proposition 7.3.1.12 (1)]{Lurie_HA}. Informally speaking, it says that
\[\colimsubscript_{b \in B\GLoneof{R}} F_R(b) \simeq \colimsubscript_{b \in B\mathrm{Gl}_1R} (\iota_b)_!F_R(b),\]
where $\iota_b \colon \ast \rightarrow B\GLoneof{R}$ denotes the inclusion of $b$. Formally, the right hand side arises as follows: Choose a colimit extension of $G_R$ to the cone category $B\GLoneof{R}^\vartriangleright$ and consider the  lifting problem consisting of the solid parts of
\[\xymatrix@-1pc{B\GLoneof{R}^\vartriangleright \times \{0\} \ar[rrr]^-{G_R}\ar[d] &&& T\cS_\infty\ar[d]^{\pi_b}\\ 
B\GLoneof{R}^\vartriangleright \times \Delta^1 \ar[r] \ar@{-->}[urrr] & B\GLoneof{R}^\vartriangleright \ar[rr]^-{\pi_b G_R} && \cS_\infty,}\]
using the map $B\GLoneof{R}^\vartriangleright \times \Delta^1 \rightarrow B\GLoneof{R}^\vartriangleright$ that is the evident projection on $B\GLoneof{R}^\vartriangleright \times \{0\}$ and collapses $B\GLoneof{R}^\vartriangleright \times \{1\}$ to the cone point. By (the duals of) \cite[Remark 2.4.1.9 and Propositions 3.1.2.1]{Lurie_HTT}, there exists an essentially unique diagonal filler mapping every edge $\{b\} \times \Delta^1$ to a cocartesian edge, since $\pi_b$ is cocartesian. The restriction $H_R$ of this filler to $B\GLoneof{R} \times \{1\}$ canonically factors through the inclusion $\mathrm{Fun}(B\GLoneof{R},\Sp_{\infty}) \rightarrow T\cS_\infty$, since $\pi_b \circ H_R$ is the constant functor with value $B\GLoneof{R}$ by the reverse implication of \cite[Proposition~7.3.1.12~(1)]{Lurie_HA}. Regarded as a functor $H_R\colon B\GLoneof{R} \rightarrow \mathrm{Fun}(B\GLoneof{R},\Sp_{\infty})$, the $H_R$ then defines the right hand side in the original assertion (\ref{colim eq}) and the agreement of the two colimits is an instance of \cite[Proposition 4.3.1.10]{Lurie_HTT}. Now $H_R$ is adjoint to $\Delta_!F_R \colon B\GLoneof{R} \times B\GLoneof{R} \rightarrow \Sp_\infty$: The value of the adjoint of $H_R$ on a pair of objects $(b,b') \in B\GLoneof{R} \times B\GLoneof{R}$ is by construction the colimit of the functor 
\[\iota_b/b' \longrightarrow \{b\} \xrightarrow{F_R} \Sp_\infty,\]
where $\iota_b :\ast \rightarrow B\GLoneof{R}$ denotes the inclusion of $b$ again, whereas the left Kan-extension along $\Delta \colon B\GLoneof{R} \rightarrow B\GLoneof{R} \times B\GLoneof{R}$ by definition evaluates to the colimit of
\[\Delta/(b,b') \longrightarrow B\GLoneof{R} \xrightarrow{F_R} \Sp_\infty.\]
But taking products with $\id_b$ induces an equivalence $\iota_b/b' \rightarrow \Delta/(b,b')$ making the above triangle commute, from which the claim follows by the uniqueness of Kan extensions. We thus find 
\[\mathrm{Hom}_{B\GLoneof{R}}(\colim H_R, -) \simeq \mathrm{Hom}_{B\GLoneof{R} \times B\GLoneof{R}}(\Delta_! F_R, \mathrm{pr}_2^* -) \simeq \mathrm{Hom}_{B\GLoneof{R}}(F_R,-)\]
as functors on $\mathrm{Fun}(B\GLoneof{R},\Sp_{\infty})$, the first by definition and the second by the adjunction $(\Delta_!,\Delta^*)$. It follows that $\mathrm{colim} G_R\simeq \mathrm{colim} H_R \simeq F_R$ as desired. 

For $R$ commutative the functor $G_R$ extends to a map $G^\otimes_R \colon (B\GLoneof{R})^\wedge \rightarrow (T\cS)^\barsm$ witnessing the lax monoidality of $G_R$. Repeating the same argument then shows that in this case $\colim G_R \simeq F_R$ as lax symmetric monoidal functors.

To see that the relative $\barsm$-product is also a colimit of $G_R$, we follow the proof of \cite[Proposition 3.26]{Ando-B-G-H-R_infinity-Thom}: They consider the inclusion of the category $B\GLoneof{R}$ into the category $\mathrm{Fun}(B\GLoneof{R}, \cS_\infty)$ of all $\GLoneof{R}$-spaces as the automorphisms of $\GLoneof{R}$. Then the diagram
\[\xymatrix@-1pc{B\mathrm{Aut}_{\mathrm{Fun}(B\GLoneof{R}, \cS_\infty)}(\GLoneof{R}) \ar[r] \ar[d]     & \mathrm{Fun}(B\GLoneof{R}, \cS_\infty) \ar[d] \\
            B\mathrm{Aut}_{R-\mathrm{Mod}}(R) \ar[r]^-{G_R} & T\cS_\infty}\]
with the vertical maps given by $\bS_t[-] \barsm_{\bS_t[\GLoneof{R}]} R$ is commutative. Since the left vertical map is an equivalence of $\infty$-groupoids the colimit over $G_R$ may equally well be computed using the upper composite. Since the right vertical map preserves colimits, this amounts to computing the relative $\barsm$-product of $R$ with the suspension spectrum of the colimit of the inclusion \[B\mathrm{Aut}_{\mathrm{Fun}(B\GLoneof{R}, \cS_\infty)}(\GLoneof{R}) \rightarrow \mathrm{Fun}(B\GLoneof{R}, \cS_\infty).\] This is the terminal object by the argument preceding \cite[Proposition 3.26]{Ando-B-G-H-R_infinity-Thom} and we obtain the desired identification. 

In case $R$ is commutative all functors in sight are lax symmetric monoidal (for the Day convolution on the upper right corner), whence the identification of colimits preserves $E_\infty$-structures.
\end{proof}

We can now also compare our definition of twisted (co)homology from Section~\ref{sec:tw-coho} with the $\infty$-categorical one given in \cites{Ando-B-G-H-R_infinity-Thom, Ando-B-G_parametrized}.

\begin{proposition}\label{cohom infty comp}
Let $R$ be a positive fibrant symmetric ring spectrum, let $G \to \GLoneIof(R)$ be a cofibrant replacement, and let $(\gamma_R, BG)$ be the universal line bundle. Given a map $\tau \colon K \rightarrow BG_{h\cI}$, there are canonical isomorphisms relating our $(\gamma_R, BG)_n(K,\tau)$ and $(\gamma_R, BG)^n(K,\tau)$ with \cite[Definition~1.2]{Ando-B-G_parametrized} applied to \[K \xrightarrow{\tau} \{n\} \times B\GLoneof{R} \rightarrow \mathrm{Pic}(R).\]
\end{proposition}
In the formulation we have again used Lemma~\ref{Gl cpmp}  to identify $BG_{h\cI}$ with the $\infty$-categorical $B\GLoneof{R}$.

\begin{proof}
By the suspension isomorphism it suffices to consider the case $n=0$. Let us then recall the relevant definitions (in our notation) for $\sigma \colon L \rightarrow B\GLoneof{R}$:
\[R^0(L,\sigma) = \pi_0F^R(\Theta\check\sigma^*\gamma_R,R) \quad \text{and} \quad R_0(L,\sigma) = \pi_0F^R(R,\Theta \sigma^* \gamma_R),\]
where $F^R$ denotes the spectrum of $R$-linear maps and $\check\sigma$ denotes the composite
\[L \xrightarrow{\sigma} B\GLoneof{R} \xrightarrow{(-)^{-1}} B\GLoneof{R}.\]
The functor $\Theta$ being left adjoint to pullback along the constant map makes it the $\infty$-categorical colimit. In the $\infty$-category of spectra we then have
\[F^R(\Theta\check\sigma^*\gamma_R,R) \simeq F_L^R(\check\sigma^*\gamma_R,R_L) 
                                      \simeq F_L^R(R_L, \sigma^*\gamma_R) \simeq F^R(R,\Gamma\sigma^*\gamma_R),\] where $(-)_L$ denotes the constant diagram functor $\Sp_\infty \rightarrow \Fun(L,\Sp_\infty)$, the second equivalence is given by fiberwise smashing (over $R$ and $L$) with $\sigma^*\gamma_R$ and the remainder by the definition of $\Theta$ and $\Gamma$ as adjoints to $(-)_L$. Identifying $F^R(R,-)$ with the identity functor on $R$-module spectra we therefore obtain from Lemma~\ref{lem:SpsymX-infty-functorial} and Proposition~\ref{prop:gamma_comp} that both the definition of homology  and cohomology agree with ours.
\end{proof}

Finally we record the agreement of our products~\eqref{eq:pairings} with their $\infty$-categorical counterparts. 

\begin{proposition}\label{prop:comparison-pairings}
Let $R$ be a positive fibrant commutative symmetric ring spectrum. Then the products~\eqref{eq:pairings} agree with those from \cite[Theorem 4.21]{Ando-B-G_parametrized} under the isomorphisms from Proposition \ref{cohom infty comp}.
\end{proposition}
\begin{proof}
This follows immediately from the conjunction of Theorem~\ref{thm:SpsymM-infty-identification}, Proposition~\ref{Gl cpmp} and Proposition~\ref{prop:gamma_comp} by unwinding the definitions.
\end{proof}

% \bibliography{retractive}

\end{document}